\documentclass{article}
\usepackage{amsmath}
\usepackage{latexsym}
\usepackage{amssymb}
\usepackage{colonequals}
\newtheorem{thm}{Theorem}[section]
\newtheorem{la}[thm]{Lemma}
\newtheorem{Defn}[thm]{Definition}
\newtheorem{Remark}[thm]{Remark}
\newtheorem{Conj}[thm]{Conjecture}
\newtheorem{prop}[thm]{Proposition}
\newtheorem{cor}[thm]{Corollary}
\newtheorem{Example}[thm]{Example}
\newtheorem{Number}[thm]{\!\!}
\newenvironment{defn}{\begin{Defn}\rm}{\end{Defn}}
\newenvironment{example}{\begin{Example}\rm}{\end{Example}}
\newenvironment{rem}{\begin{Remark}\rm}{\end{Remark}}

\newenvironment{numba}{\begin{Number}\rm}{\end{Number}}
\newenvironment{proof}{{\noindent\bf Proof.}}%
                  {\nopagebreak\hspace*{\fill}$\Box$\medskip\par}

\newcommand{\cO}{{\mathcal O}}
\newcommand{\cK}{{\mathcal K}}
\newcommand{\cE}{{\mathcal E}}
\newcommand{\cL}{{\mathcal L}}

\newcommand{\cA}{{\mathcal A}}

\newcommand{\cT}{{\mathcal T}}

\newcommand{\ve}{\varepsilon}
\newcommand{\R}{{\mathbb R}}
\newcommand{\N}{{\mathbb N}}

\newcommand{\mto}{\mapsto}
\newcommand{\sub}{\subseteq}

\DeclareMathOperator{\id}{id}

\DeclareMathOperator{\graph}{graph}

\DeclareMathOperator{\Supp}{supp}
\DeclareMathOperator{\pr}{pr}

\DeclareMathOperator{\Evol}{Evol}

\DeclareMathOperator{\ev}{ev}

\DeclareMathOperator{\im}{im}

\DeclareMathOperator{\fb}{fb}

\DeclareMathOperator{\comp}{comp}
\newcommand{\coloneq}{\colonequals}

\newcommand{\pl}{{\displaystyle \lim_{\longleftarrow}}}
\newcommand{\cg}{{\mathfrak g}}
\newcommand{\ch}{{\mathfrak h}}
\begin{document}
$\;$\\[-14mm]
\begin{center}
{\bf\Large Manifolds of mappings on cartesian products}\\[4mm]
{\bf Helge Gl\"{o}ckner\footnote{Supported by German Academic Exchange Service, DAAD grant 57568548} and Alexander Schmeding}\vspace{2mm}
\end{center}
\begin{abstract}
\noindent Given smooth manifolds
$M_1,\ldots, M_n$
(which may have a boundary or corners),
a smooth manifold~$N$
modeled on locally convex spaces
and $\alpha\in(\N_0\cup\{\infty\})^n$,
we consider the set $C^\alpha(M_1\times\cdots\times M_n,N)$
of all mappings $f\colon M_1\times\cdots\times M_n\to N$
which are $C^\alpha$ in the sense of Alzaareer.
Such mappings admit, simultaneously,
continuous iterated directional derivatives of orders $\leq \alpha_j$
in the $j$th variable for $j\in\{1,\ldots, n\}$,
in local charts.
We show that $C^\alpha(M_1\times\cdots\times M_n,N)$
admits a canonical smooth manifold structure
whenever each $M_j$ is compact and $N$ admits a local addition.
The case of non-compact domains is also considered.\vspace{1.4mm}
\end{abstract}
%
46E40,
46T20 (secondary)\\[2.3mm]
%
%
compact-open topology; exponential law; evaluation map; mapping group;\\[.3mm]
regularity; box product; non-compact manifold

\tableofcontents
\section{Introduction and statement of the results}
As known from classical work by Eells \cite{Eel},
the set $C^\ell(M,N)$ of all $C^\ell$-maps
$f \colon M\to N$ can be given a smooth Banach manifold
structure
for each $\ell\in\N_0$, compact smooth manifold~$M$
and $\sigma$-compact finite-dimensional smooth manifold~$N$.
More generally, $C^\ell(M,N)$ is a smooth manifold
for each $\ell\in \N_0\cup\{\infty\}$,
locally compact, paracompact
smooth manifold~$M$ with rough boundary in the sense of~\cite{GaN}
(this includes finite-dimensional manifolds with boundary,
and manifolds with corners as
in \cite{Cer, Dou, Mic})
and each smooth
manifold~$N$ modeled on locally convex
spaces such that $N$ admits a local addition
(a concept recalled in Definition~\ref{def-loa});
see \cite{Ham, Mic, Mil, Wit, AGS, Rou}
for discussions in different levels of\linebreak
generality,
and~\cite{KaM} for manifolds of smooth maps in the convenient setting
of analysis. For compact~$M$, the modeling space of $C^\ell(M,N)$
around $f\in C^\ell(M,N)$ is the locally convex space
$\Gamma_{C^\ell}(f^*(TN))$ of all $C^\ell$-sections in the pullback bundle
$f^*(TN)\to M$, which can be identified with
\[
\Gamma_f:=\{\tau\in C^\ell(M,TN)\colon \pi_{TN}\circ \tau=f\};
\]
if $M$ is not compact,
the locally convex space of compactly supported $C^\ell$-sections
of $f^*(TN)$ is used.
Let $L$ be a smooth manifold modeled
on locally convex spaces (possibly with rough boundary),
and $k\in \N_0\cup\{\infty\}$.
For compact~$M$, it is known from
\cite[Proposition 1.23 and Definition 1.17]{AGS} that a map
\[
g \colon L\to C^\ell(M,N)
\]
is $C^k$ if and only if the corresponding map of two variables,
\[
g^\wedge\colon L\times M\to N,\;\, (x,y)\mto g(x)(y)
\]
is $C^{k,\ell}$ in the sense of \cite{AaS},
i.e., a continuous map which in local charts admits
up to $\ell$ directional derivatives
in the second variable, followed by up to~$k$
directional derivatives in the first variable,
with continuous dependence on point and directions
(see \ref{defn-C-alpha} and \ref{C-alpha-in-mfd} for details).
We thus obtain a bijection
\[
\Phi\colon C^k(L,C^\ell(M,N))\to C^{k,\ell}(L\times M,N),\;\, g\mto g^\wedge.
\]
As our first result, for compact~$L$ we construct a smooth
manifold structure on $C^{k,\ell}(L\times M,N)$
which turns $\Phi$ into a $C^\infty$-diffeomorphism.
More generally, analogous to the $n=2$ case
of $C^{k,\ell}$-maps,
we consider $N$-valued $C^\alpha$-maps
on an $n$-fold product
$M_1\times\cdots\times M_n$ of smooth manifolds
for any $n\in\N$ and $\alpha=(\alpha_1,\ldots, \alpha_n)
\in (\N_0\cup\{\infty\})^n$.
With terminology explained presently, we get:
\begin{thm}\label{thmA}
Given $\alpha=(\alpha_1,\ldots,\alpha_n)\in(\N_0\cup\{\infty\})^n$,
let $M_j$ for $j\in\{1,\ldots,n\}$
be a compact smooth manifold with rough boundary.
Let $N$ be a smooth manifold modeled on locally
convex spaces such that~$N$ can be covered by local additions.
Then $C^\alpha(M_1\times\cdots\times M_n,N)$
admits a smooth manifold structure which is canonical.
The following hold for this canonical manifold
structure:
\begin{itemize}
\item[\rm(a)]
$C^\alpha(M_1\times\cdots\times M_n,N)$
can be covered by local additions.
If $N$ admits a local addition, then also
$C^\alpha(M_1\times\cdots\times M_n,N)$
admits a local addition.
\item[\rm(b)]
Given $m\in \N$ and $\beta=(\beta_1,\ldots,\beta_m)\in (\N_0\cup\{\infty\})^m$,
let $L_j$ be a compact smooth manifold with rough boundary for $j\in \{1,\ldots, m\}$.
Then canonical smooth manifold structures turn the bijection
\[
\!\!\!\!\!\!\hspace{-2mm}
C^\beta(L_1\times\cdots\times L_m,C^\alpha(M_1\times\cdots\times M_n,N))
\!\to C^{\beta,\alpha}(L_1\times\cdots\times L_m\times M_1\times\cdots\times
M_n,N)
\]
taking $g$ to $g^\wedge$
into a $C^\infty$-diffeomorphism.
\end{itemize}
\end{thm}
The following terminology was used:
We say that a smooth manifold~$N$
\emph{can be covered by local additions}
if $N$ is the union of an upward directed family
$(N_j)_{j\in J}$ of open submanifolds~$N_j$
which admit a local addition.
For instance, any (not necessarily paracompact)
finite-dimensional smooth manifold
has this property, e.g.\ the long line.
We also used
canonical manifold structures.\\[2mm]
Note that if a map $f\colon L_1\times \cdots\times L_m\times
M_1\times\cdots\times M_n\to N$
on a product of smooth manifolds
with rough boundary is $C^{\beta,\alpha}$
with $\alpha\in (\N_0\cup\{\infty\})^n$
and $\beta\in (\N_0\cup\{\infty\})^m$,
then the map
\[
f^\vee(x):=f(x,\cdot)\colon M_1\times \cdots\times M_n\to N
\]
is $C^\alpha$ for each $x\in L_1\times\cdots\times L_m$
(see \cite[Lemma~3.3]{Alz}).
\begin{defn}\label{defn-can}
Let $N$ be a smooth manifold modeled
on a locally convex space, $M_1,\ldots, M_n$
be finite-dimensional smooth manifolds with rough boundary
and $\alpha\in(\N_0\cup\{\infty\})^n$.
A smooth manifold structure on $C^\alpha(M_1\times\cdots\times M_n,N)$
is called
\emph{pre-canonical}
if the following condition is satisfied
for each $m\in\N$
and each $\beta\in (\N_0\cup\{\infty\})^m$:
If $L_j$ for $j\in \{1,\ldots, m\}$ is a smooth manifold
with rough boundary modeled on locally convex spaces,
then a map
\[
g \colon L_1\times\cdots\times L_m\to C^\alpha(M_1\times\cdots\times M_n,N)
\]
is $C^\beta$ if and only if the map
\[
g^\wedge \colon L_1\times\cdots\times L_m\times M_1\times\cdots\times M_n\to N
\]
given by $g^\wedge(x_1,\ldots, x_m,y_1,\ldots, y_n):=
g(x_1,\ldots, x_m)(y_1,\ldots, y_n)$ is $C^{\beta,\alpha}$.
Thus
\[
C^\alpha(L_1\times\cdots\times L_m,C^\beta(M_1\times\cdots\times M_n,N))
\hspace*{-.3mm}\to\hspace*{-.3mm} C^{\beta,\alpha}(L_1\times\cdots\times L_n\times
M_1\times\cdots\times M_n,N),\vspace{-1mm}
\]
\begin{equation}\label{pre-can-bij}
\quad \;\;  g\mto g^\wedge
\end{equation}
is a bijection.
The manifold structure is called
\emph{canonical} if, moreover,
its underlying topology
is the compact-open $C^\alpha$-topology
(as in Definition~\ref{defn-alpha-top}).
\end{defn}
Canonical manifold structures are essentially unique whenever they exist,
and so are pre-canonical ones (see Lemma~\ref{base-cano}\,(b) for details).\\[2.3mm]
We address two further topics for not necessarily
compact domains:
\begin{itemize}
\item[(i)]
We formulate criteria ensuring that
$C^\alpha(M_1\times\cdots\times M_n,G)$
admits a canonical smooth manifold structure
(making the latter a Lie group),
for a Lie group~$G$
modeled on a locally convex space;
\item[(ii)]
Manifold structures
on $C^\alpha(M_1\times\cdots\times M_n,N)$
which are modeled on certain spaces
of compactly supported
$TN$-valued functions,
in the spirit of~\cite{Mic}.
\end{itemize}
To discuss (i), we use a generalization
of the regularity concept introduced by John Milnor~\cite{Mil}
(the case $r=\infty$).
If $G$ is a Lie group modeled on a locally convex
space, with neutral element $e$,
we write $\lambda_g\colon G\to G$, $x\mto gx$
for left translation with $g\in G$
and consider the smooth left action
\[
G\times TG\to TG,\quad (g,v)\mto g.v:=T\lambda_g(v)
\]
of $G$ on its tangent bundle.
We write $\cg:=T_eG$ for the Lie algebra of $G$.
Let $r\in \N_0\cup\{\infty\}$.
The Lie group $G$ is called \emph{$C^r$-semiregular}
if, for each $C^r$-curve $\gamma\colon [0,1]\to \cg$,
the initial value problem
\[
\dot{\eta}(t)=\eta(t).\gamma(t),\quad \eta(0)=e
\]
has a (necessarily unique)
solution $\eta\colon [0,1]\to G$. Write $\Evol(\gamma):=\eta$.
If, moreover, the map
\[
\Evol\colon C^r([0,1],\cg)\to C^{r+1}([0,1],G)
\]
is smooth, then $G$ is called \emph{$C^r$-regular}
(cf.\ \cite{Sem}).
If $s\leq r$ and $G$ is $C^s$-regular, then $G$ is $C^r$-regular
(see \cite{Sem}). We show:
\begin{thm}\label{thmB}
Let $G$ be a $C^r$-regular Lie group modeled
on a locally convex space with $r\in \N_0\cup\{\infty\}$.
For some $n\in\N$, let $M_1,\ldots, M_n$
be locally compact smooth manifolds
with rough boundary
and $\alpha\in (\N_0\cup\{\infty\})^n$.
For each $j\in\{1,\ldots, n\}$ such that $M_j$
is not compact, assume that
$\alpha_j\geq r+1$
and $M_j$ is $1$-dimensional
with finitely many connected components.
Then we have:
\begin{itemize}
\item[\rm(a)]
$C^\alpha(M_1\times\cdots\times M_n,G)$
admits a canonical smooth manifold structure;
\item[\rm(b)]
The canonical manifold structure from~{\rm(a)}
makes $C^\alpha(M_1\times\cdots\times M_n,G)$
a $C^r$-regular Lie group.
\end{itemize}
\end{thm}
The Lie algebra of $C^\alpha(M_1\times \cdots\times M_n,G)$
can be identified with the topological Lie algebra
$C^\alpha(M_1\times\cdots\times M_n,L(G))$
in a standard way (Proposition~\ref{amendment}).
Of course,
we are most interested in the case
that the non-compact $1$-dimensional factors
are $\sigma$-compact
and hence intervals, or finite disjoint unions
of such.
But we did not need to assume
$\sigma$-compactness in the theorem,
and thus $M_j$ with $\alpha_j\geq r+1$
might well be a long line,
or a long ray.\\[2.3mm]
Disregarding the issue of being canonical,
the Lie group structure on\linebreak
$C^\infty(M_1\times\cdots\times M_n,G)=C^\alpha(M_1\times\cdots\times M_n,G)$
with $\alpha_1:=\cdots:=\alpha_n=\infty$
was first obtained in~\cite{NaW},
for smooth manifolds $M_j$ without boundary
which are compact
or diffeomorphic to~$\R$.
The Lie group structure for $n=1$
was first obtained in~\cite{Hmz}
for domains diffeomorphic to intervals,
together with a sketch for the case $n=2$
(assuming additional conditions,
e.g. $\alpha_1\geq r+3$
and $\alpha_2\geq r+1$ if $M_1=M_2=\R$).
Our approach differs:
While the studies in \cite{NaW} and \cite{Hmz}
assume regularity of~$G$ from the start to enforce
exponential laws,
and build it into a notion of Lie group structures on mapping groups
that are ``compatible
with evaluations,''
we take canonical and pre-canonical manifold structures
as the starting point (independent of regularity)
and combine them with regularity
or compatibility with evaluations
(adapted to $C^\alpha$-maps
in Definition~\ref{defn-expected})
only when needed.\\[2.3mm]
As to topic~(b),
our constructions show:
\begin{thm}\label{thmC}
Given $\alpha=(\alpha_1,\ldots,\alpha_n)\in(\N_0\cup\{\infty\})^n$,
let $M_j$ for $j\in\{1,\ldots,n\}$
be a paracompact, locally compact
smooth manifold with rough boundary;
abbreviate $M:=M_1\times\cdots\times M_n$.
Let $N$ be a smooth manifold modeled on locally
convex spaces such that~$N$ admits
a local addition.
Let $\pi_{TN}\colon TN\to N$ be the canonical map.
For $f\in C^\alpha(M,N)$
and a compact subset $K\sub M$, the set
\[
\!\! \Gamma_{f,K}:=
\left\{\tau\in C^\alpha(M,TN)\colon \mbox{$\pi_{TN}\circ \tau=f$ {\rm\&} $\tau(x)=0
\in T_{f(x)}N$ for all $x\in M\setminus K$}\right\}
\]
is a vector subspace of $\prod_{x\in M}T_{f(x)}N$,
and a locally convex space in the topology induced by
$C^\alpha(M,TN)$. Give
$\Gamma_f=\bigcup_K\Gamma_{f,K}$ the locally convex direct
limit topology.
Then $C^\alpha(M,N)$
admits a unique smooth manifold structure
modeled on the set $\cE:=\{\Gamma_f\colon f\in C^\alpha(M,N)\}$
of locally convex spaces such that,
for each $f\in C^\alpha(M,N)$
and local addition $\Sigma\colon TN\supseteq U \to N$
of~$N$, the map
\[
\Gamma_f\cap C^\alpha(M,U)\to C^\alpha(M,N),\;\,
\tau\mto \Sigma\circ\tau
\]
is a $C^\infty$-diffeomorphism onto an open subset of $C^\alpha(M,N)$.
\end{thm}
In the case that $n=1$, $k=\infty$
and $M:=M_1$ is a smooth manifold with corners,
we recover the smooth manifold structure
on $C^\infty(M,N)$ discussed by Michor~\cite{Mic}.\\[2.3mm]
Using manifold structures on
infinite cartesian products of manifolds
making them ``fine box products''
(a concept recalled in
Section~\ref{sec-box}),
Theorem~\ref{thmC} turns into a
corollary to Theorem~\ref{thmA}.\\[2.3mm]
In the case $n=1$,
for compact~$M$ and $\ell\in\N_0\cup\{\infty\}$, canonical manifold\linebreak
structures on $C^\ell(M,N)$ as in Theorem~\ref{thmA}
have already been considered in \cite{AGS},
in a weaker sense (fixing $m=1$
in Definition~\ref{defn-can}).
Parts of our discussion adapt arguments from~\cite{AGS}
to the more difficult case
of $C^\alpha$-maps.\\[2.3mm]
\textbf{Acknowledgement.}
The authors would like to thank the mathematical\linebreak
institute at NTNU Trondheim for its hospitality while conducting the work presented in this article,
as well as Nord universitet Levanger.
\section{Preliminaries and notation}\label{secprels}
We write $\N:=\{1,2,\ldots\}$
and $\N_0:=\N\cup\{0\}$.
If $\alpha,\beta\in(\N_0\cup\{\infty\})^n$
with $n\in \N$,
we write $\alpha\leq\beta$
if $\alpha_j\leq \beta_j$ for all $j\in\{1,\ldots, n\}$.
We let $|\alpha|:=\alpha_1+\cdots+\alpha_n\in\N_0\cup\{\infty\}$.
As usual, $\infty+k:=\infty$ for all $k\in\N_0\cup\{\infty\}$.
For $j\in \{1,\ldots, n\}$,
let $e_j:=(0,\ldots,0,1,0,\ldots, 0)\in(\N_0)^n$
with $1$ in the $j$th slot.
We abbreviate
``Hausdorff locally convex topological
$\R$-vector space''
as ``locally convex space.''
We work in the setting of differential calculus
going back to Andr\'{e}e Bastiani~\cite{Bas}
(see \cite{Res, GaN, Ham, Mic, Mil, Nee}
for discussions in
varying generality),
also known as Keller's $C^k_c$-theory~\cite{Kel}.
For $C^\alpha$-maps, see
\cite{Alz} (cf.\ \cite{AaS}
and \cite{GaN} for the case of two variables,
$\alpha\in (\N_0\cup\{\infty\})^2$).
We now introduce concepts for later use and
collect basic facts.
For proofs, see the appendix.
%
\begin{numba}\label{defn-Ck}
Consider locally convex spaces $E$, $F$
and a map $f\colon U\to F$
on an open subset $U\sub E$.
Write
\[
(D_yf)(x):=\frac{d}{dt}\Big|_{t=0}f(x+ty)
\]
for the directional derivative of~$f$ at $x\in U$
in the direction $y\in E$, if it exists.
Let $k\in \N_0\cup\{\infty\}$.
If $f$ is continuous, the iterated directional derivatives
\[
d^jf(x,y_1,\ldots, y_j):=(D_{y_j}\ldots D_{y_1}f)(x)
\]
exist for all $j\in\N_0$ such that
$j\leq k$, $x\in U$ and $y_1,\ldots, y_j\in E$,
and the maps $d^jf\colon U\times E^j\to F$
are continuous, then $f$ is called~$C^k$.
If $U$ may not be open,
but has dense interior~$U^o$
and is locally convex in the sense
that each $x\in U$ has a convex neighbourhood
in~$U$, following~\cite{GaN}
a map $f\colon U\to F$ is called $C^k$
if it is continuous,
$f|_{U^o}$ is $C^k$
and $d^j(f|_{U^o})$ has a continuous
extension $d^jf\colon U\times E^j\to F$
for all $j\in \N_0$ with $j\leq k$.
The $C^\infty$-maps are also called~\emph{smooth}.
\end{numba}
\begin{rem}\label{ext-ops}
If $E=\R^n$ and $U$ is relatively open in
$[0,\infty[^n$,
then $f$ as above is $C^k$ if and only $f$
has a $C^k$-extension to an open set
in $\R^n$ (see \cite{EXT}, cf.\ \cite{Han}).
\end{rem}
\begin{numba}\label{defn-mfd}
Let $k\in\N\cup\{\infty\}$.
A \emph{manifold
with rough boundary}
modeled on a non-empty set $\cE$ of locally
convex spaces is a Hausdorff topological
space~$M$, together with a set $\cA$
of homeomorphisms (``charts'')
$\phi\colon U_\phi\to V_\phi$
from an open subset $U_\phi\sub M$
onto a locally convex subset $V_\phi\sub E_\phi$
with dense interior for some $E_\phi\in\cE$,
such that $\phi\circ \psi^{-1}$
is $C^k$ for all $\phi,\psi\in \cA$,
the union $\bigcup_{\phi\in \cA}U_\phi$
equals~$M$, and $\cA$ is maximal.
If $k=0$, assume
in addition that $\phi(x)\in\partial V_\phi$
if and only if $\psi(x)\in\partial V_\psi$
for all $\phi,\psi\in\cA$ with
$x\in U_\phi\cap U_\psi$
(which is automatic if $k\geq 1$).
Let $\partial M$ be the set of
all $x\in M$ such that $\phi(x)\in \partial V_\phi$
for some (and hence any) chart $\phi$
around~$x$.
If $\cE$ is a singleton,
$M$ is called \emph{pure}.
If $M$ is a $C^k$-manifold with rough boundary
and $\partial M=\emptyset$,
then $M$ is called a \emph{$C^k$-manifold}
or a \emph{$C^k$-manifold without boundary}, for emphasis.
(See \cite{GaN} for all of this in the pure case;
cf.\ \cite{AGS} for modifictions
in the general case).
\end{numba}
\begin{numba}\label{conventions-mfd}
All manifolds and Lie groups considered
in the article are modeled on locally
convex spaces which may be infinite-dimensional,
unless the contrary is stated.
Finite-dimensional manifolds need not be paracompact
or $\sigma$-compact, unless stated explicitly.
As we are interested in manifolds
of mappings, consideration of pure manifolds
would not be sufficient.
\end{numba}
\begin{numba}\label{defn-df}
If $U$ is an open subset
of a locally convex space~$E$
(or a locally convex subset with dense interior),
we identify its tangent bundle $TU$ with $U\times E$,
as usual,
with bundle projection $(x,y)\mto x$.
If $M$ is a $C^k$-manifold with
rough boundary and $f\colon M\to U$
a $C^k$-map with
$k\geq 1$, we write $df$ for the second component
of $Tf\colon TM\to TU=U\times E$.
Thus $Tf=(f\circ \pi_{TM},df)$,
using the bundle projection $\pi_{TM}\colon TM\to M$.
\end{numba}
\begin{numba}\label{lie-functor}
If $G$ is a Lie group
with neutral element $e$,
we write $L(G):=T_eG$ (or $\cg$)
for its tangent space at~$e$,
endowed with its natural topological
Lie algebra structure.
If $\psi\colon G\to H$
is a smooth homomorphism between
Lie groups, we let $L(\psi):=T_e\psi\colon L(G)\to L(H)$
be the associated continuous Lie algebra homomorphism.
\end{numba}
\begin{numba}\label{leftlog}
If $G$ is a Lie group
with Lie algebra $\cg$ and $I$ a non-degenerate
interval with $0\in I$, we define $\delta^\ell(\eta)$
for $\eta\in C^1(I,G)$ via
$\delta^\ell(\eta)(t):=\eta(t)^{-1}.\dot{\eta}(t)$,
with $\dot{\eta}(t):=T\eta(t,1)$.
\end{numba}
\begin{la}\label{reparametrize}
Let $k,r\in \N_0\cup\{\infty\}$
with $k\geq r$.
If $G$ is $C^r$-semiregular
and $\gamma\in C^k(I,\cg)$,
then there exists a unique
$\eta\in C^1(I,\cg)$
such that $\eta(0)=e$ and $\delta^\ell(\eta)=\gamma$.
Moreover, $\eta$ is $C^{k+1}$.
\end{la}
\begin{numba}\label{submfd}
Let $M$ be a smooth manifold
(without boundary).
A subset $N\sub M$ is called a \emph{submanifold}
if, for each $x\in N$, there exist
a chart $\phi\colon U_\phi\to V_\phi\sub E_\phi$
of~$M$ around~$x$ and a closed vector subspace
$F\sub E_\phi$ such that $\phi(U_\phi\cap N)=V_\phi\cap F$.
\end{numba}
\begin{numba}\label{defn-full-sub}
Let $M$ be a smooth manifold
with rough boundary.
A subset $N\sub M$
is called a \emph{full submanifold}
if, for each $x\in N$,
there exists a chart $\phi\colon U_\phi\to V_\phi \sub E_\phi$
of~$M$ around $x$ such that
$\phi(U_\phi\cap N)$
is a locally convex subset of~$E_\phi$
with dense interior.
\end{numba}
\begin{numba}\label{defn-C-alpha}
Let $F$ and $E_1,\ldots, E_n$
be locally convex spaces,
$U_j\sub E_j$ be an open subset
for $j\in\{1,\ldots, n\}$
and $f\colon U\to F$ be a map
on $U:=U_1\times\cdots \times U_n$.
Identifying $E:=E_1\times\cdots\times E_n$ with
$E_1\oplus\cdots\oplus E_n$,
we can identify each $E_j$
with a vector subspace of~$E$,
and simply write $D_yf(x)$ for
a directional derivative with $x\in U$,
$y\in E_j$ (rather than $D_{(0,\ldots,0,y,0,\ldots, 0)}f(x)$
with $j-1$ zeros on the left
and $n-j$ zeros on the right-hand side).
For $y=(y_1,\ldots, y_k)\in E_j^k$,
abbreviate
\[
D_y:=D_{y_k}\ldots D_{y_1}.
\]
Let $\alpha\in (\N_0\cup\{\infty\})^n$.
Following \cite{Alz}, we say that $f$ is
$C^\alpha$ if $f$ is continuous, the iterated directional derivatives
\[
d^\beta f(x,y_1,\ldots, y_n)
:=(D_{y_n}\cdots D_{y_1}f)(x)
\]
exist for all $\beta\in\N_0^n$
with $\beta\leq\alpha$,
$x\in U$ and $y_j=(y_{j,1},\ldots, y_{j,\beta_j})\in (E_j)^{\beta_j}$
for $j\in\{1,\ldots, n\}$, and
\[
d^\beta f\colon U\times E_1^{\beta_1}\times\cdots\times E_n^{\beta_n}\to
F
\]
is continuous. If $U_j$ may not be open but is
a locally convex subset of~$E_j$ with dense
interior, we say that $f\colon U\to F$
is $C^\alpha$ if $f$ is continuous,
$f|_{U^o}$ is $C^\alpha$
and $d^\beta(f|_{U^o})$ has a continuous
extension $d^\beta f\colon U\times E_1^{\beta_1}\times\cdots\times E_n^{\beta_n}\to F$
for all $\beta\in (\N_0)^n$ such that $\beta\leq \alpha$.
\end{numba}
\begin{numba}\label{C-alpha-in-mfd}
Let $M_1,\ldots, M_n$ be
$C^\infty$-manifolds with rough boundary,
$\alpha\in (\N_0\cup\{\infty\})^n$
and $N$ be a $C^k$-manifold with $k\geq|\alpha|$.
We say that a map $f\colon M_1\times \cdots \times M_n\to N$
is $C^\alpha$ if, for each $x=(x_1,\ldots, x_n)\in M_1\times\cdots\times M_n$,
there are charts $\phi_j\colon U_j\to V_j$
for~$M_j$ around $x_j$ for $j\in\{1,\ldots,n\}$
and a chart $\psi\colon U_\psi\to V_\psi$ for~$n$
around $f(x)$ such that $f(U_1\times\cdots\times U_n)\sub U_\psi$
and
\[
\psi\circ f\circ (\phi_1\times\cdots\times \phi_n)^{-1}\colon
V_1\times\cdots\times V_n\to V_\psi
\]
is $C^\alpha$. The latter then holds for any such
charts, by the Chain Rule for $C^\alpha$-maps
(as in \cite[Lemma~3.16]{Alz}).
\end{numba}
\begin{numba}\label{reorder}
Let $N$ and $M_1,\ldots, M_n$ be
$C^\infty$-manifolds with rough boundary,
$\sigma$ be a permutation of $\{1,\ldots, n\}$,
and $\alpha\in (\N_0\cup\{\infty\})^n$.
If $f\colon M_{\sigma(1)}\times\cdots\times M_{\sigma(n)}\to N$
is $C^{\alpha\circ\sigma}$,
then the map
\[
M_1\times\cdots\times M_n\to N,\quad (x_1,\ldots, x_n)\mto
f(x_{\sigma(1)},\ldots,x_{\sigma(n)})
\]
is $C^\alpha$.
This follows from Schwarz' Theorem
(in the form of \cite[Proposition~3.5]{Alz}).
\end{numba}
We shall use
simple facts:
\begin{la}\label{la-sub-PL}
Let $E_j$ for $j\in\{1,\ldots, n\}$
and $F$ be locally convex spaces,
and $U_j\sub E_j$ be a locally convex subset
with dense interior. Let
$E:=E_1\times\cdots\times E_n$,
$U:=U_1\times\cdots\times U_n$,
$\alpha\in (\N_0\cup\{\infty\})^n$
and $f\colon U\to F$ be a map.
\begin{itemize}
\item[\rm(a)]
If $Y\sub F$ is a closed
vector subspace and $f(U)\sub Y$,
then $f$ is $C^\alpha$
if and only if its co-restriction
$f|^Y\colon U\to Y$
is $C^\alpha$.
\item[\rm(b)]
If $F$ is the projective limit
of a projective system $((F_a)_{a\in A},(\lambda_{a,b})_{a\leq b})$
of locally convex spaces $F_a$ and continuous linear mappings
$\lambda_{a,b}\colon F_b\to F_a$,
with limit maps $\lambda_a\colon F\to F_a$,
then $f$ is $C^\alpha$ if and only if $\lambda_a\circ f\colon U\to F_a$
is $C^\alpha$ for all $a\in A$.
\end{itemize}
\end{la}
\begin{la}\label{specialized-PL}
Let $M$, $N$, and $L_1,\ldots, L_n$ be
smooth manifolds with rough boundary,
$F$ be a locally convex space,
$\psi\colon M\to F\times N$
be a $C^\infty$-diffeomorphism,
and $f\colon L_1\times\cdots\times L_n\to M$
be a map. Assume that
$F$ is the projective limit
of a projective system $((F_a)_{a\in A},(\lambda_{a,b})_{a\leq b})$
of locally convex spaces $F_a$ and continuous linear mappings
$\lambda_{a,b}\colon F_b\to F_a$,
with limit maps $\lambda_a\colon F\to F_a$.
For $a\in A$, let $M_a$ be a smooth manifold
and $\rho_a\colon M\to M_a$
be a $C^ \infty$-map.
Assume that there exist $C^\infty$-maps
$\psi_a\colon M_a\to F_a\times N$
making the diagram
\[
\begin{array}{rcl}
M & \stackrel{\psi}{\longrightarrow} & F\times N\\
\rho_a \downarrow\; & & \;\;\;\; \downarrow \lambda_a\times\id_N\\
M_a & \stackrel{\psi_a}{\longrightarrow} & F_a\times N
\end{array}
\]
commute.
Then $f$ is $C^\alpha$ if and only if
$\rho_a\circ f$ is $C^\alpha$
for all $a\in A$.
\end{la}
\begin{numba}
If $\alpha=(\alpha_1,\ldots,\alpha_n)\in (\N_0\cup\{\infty\})^n$
and $\beta=(\beta_1,\ldots, \beta_m)\in (\N_0\cup\{\infty\})^m$,
we shall write $(\alpha,\beta)$ as a shorthand
for $(\alpha_1,\ldots, \alpha_n,\beta_1,\ldots,\beta_m)$
and abbreviate $C^{(\alpha,\beta)}$ as $C^{\alpha,\beta}$.
Likewise for higher numbers of multiindices.
\end{numba}
Let $r\in \N_0\cup\{\infty\}$,
$E_1,\ldots, E_n$ and $F$ be locally convex spaces
and $U_j$ be a locally convex subset
of~$E_j$ with dense interior, for $j\in\{1,\ldots, n\}$.
We mention
that a map $f\colon U_1\times\cdots\times U_n\to F$
is $C^r$ if and only if it is $C^\beta$
for all $\beta\in (\N_0\cup\{\infty\})^n$
such that $|\beta|\leq r$.
More generally, the following is known
(as first formulated and proved in
the unpublished work~\cite{Ing}):
\begin{la}\label{ingrisch}
For $i\in\{1,\ldots, n\}$,
let $E_i$ be a locally convex space of the form
$E_i=E_{i,1}\times\cdots\times E_{i,m_i}$
for some $m_i\in \N$
and locally convex spaces $E_{i,1},\ldots, E_{i,m_i}$.
Let $U_{i,j}$ be a locally convex subset of $E_{i,j}$
with dense interior for all $i\in\{1,\ldots, n\}$
and $j\in\{1,\ldots, m_i\}$;
define $U_i:=U_{i,1}\times\cdots\times U_{i,m_i}$.
Let $\alpha\in (\N_0\cup\{\infty\})^n$.
Then a map $f\colon U_1\times\cdots \times U_n\to F$
is $C^\alpha$ if and only if $f$ is $C^{\beta_1,\ldots,\beta_n}$
on $\prod_{i=1}^n\prod_{j=1}^{m_i}U_{i,j}$
for all $(\beta_1,\ldots,\beta_n)\in \prod_{i=1}^n (\N_0\cup\{\infty\})^{m_i}$
such that $|\beta_i|\leq \alpha_i$ for all $i\in\{1,\ldots, n\}$.
\end{la}
\section{The compact-open {\boldmath$C^\alpha$}-topology}\label{sec-alpha-tops}
As a further preliminary,
we introduce a topology on $C^\alpha(M_1\times \cdots\times M_n,N)$
which parallels the familiar compact-open $C^k$-topology
on $C^k(M,N)$. Basic properties are recorded,
with proofs in Appendix~\ref{appA}.\\[2.3mm]
As usual, $T^0M:=M$, $T^1M:=TM$
and $T^kM:=T(T^{k-1}M)$ for a smooth manifold $M$
with rough boundary and integers $k\geq 2$
(see \cite{GaN}).
\begin{numba}
In \ref{defn-alpha-bundle} -- \ref{maps-to-tvs-1},
$M_1,\ldots, M_n$
will be smooth manifolds with rough boundary,
and $M:=M_1\times\cdots\times M_n$.
In \ref{tangent-maps} -- \ref{maps-to-sub}, we let $N$
be a smooth manifold with rough boundary
and $\alpha\in (\N_0\cup\{\infty\})^n$.
\end{numba}
\begin{numba}\label{defn-alpha-bundle}
We define the \emph{$\beta$-tangent bundle of~$M$}
as
$T^\beta M:=T^{\beta_1}M_1\times \cdots\times T^{\beta_n}M_n$
for $\beta=(\beta_1,\ldots,\beta_n)\in (\N_0)^n$.
\end{numba}
\begin{numba}\label{tangent-maps}
Let $f\colon M\to N$ be a $C^\alpha$-map.
For $\beta=(\beta_1,\ldots,\beta_n)\in (\N_0)^n$ with $\beta\leq \alpha$,
we define
\[
T^\beta(f)\colon T^\beta(M)\to T^{|\beta|}N
\]
recursively, as follows: We first note that, by Lemma~\ref{partial-tangents},
\[
T^{(0,\ldots, 0, \beta_n)}f\colon
M_1\times\cdots \times M_{n-1}\times T^{\beta_n}M_n\to T^{\beta_n}N,
\]
$(x_1,\ldots, x_{n-1},v_n)\mto T^{\beta_n}(f(x_1,\ldots, x_{n-1},\cdot))(v_n)$
is a $C^{(\alpha_1,\ldots,\alpha_{n-1},0)}$-map.
If\vspace{.4mm}
a $C^{(\alpha_1,\ldots,\alpha_{k-1},0,\ldots, 0)}$-map
$g:=T^{(0,\ldots, 0, \beta_k,\ldots,\beta_n)}f\colon
T^{(0,\ldots, 0, \beta_k,\ldots,\beta_n)}M\to T^{\beta_k+\cdots+\beta_n}N$
has already been constructed for
$k\in \{2,\ldots, n\}$,
then the map
\[
T^{(0,\ldots, 0, \beta_{k-1},\ldots,\beta_n)}f\colon
T^{(0,\ldots, 0, \beta_{k-1},\ldots,\beta_n)}M\to \,T^{\beta_{k-1}+\cdots+\beta_n}N
\]
taking $(x_1,\ldots, x_{k-2},v_{k-1},\ldots, v_n)$ to
$T^{\beta_{k-1}}(g(x_1,\ldots, x_{k-2},\cdot,v_k,\ldots,v_n))(v_{k-1})$
is a $C^{(\alpha_1,\ldots,\alpha_{k-2},0,\ldots,0)}$-map
(see Lemmas~\ref{reorder}
and \ref{partial-tangents}).
\end{numba}
\begin{defn}\label{defn-alpha-top}
The \emph{compact-open $C^\alpha$-topology}
on $C^\alpha(M,N)$ is the
initial topology with respect to the mappings
\[
T^\beta\colon C^\alpha(M,N)\to C(T^\beta M, T^{|\beta|}N),\;\;
f\mto T^\beta f
\]
for $\beta\in (\N_0)^n$ with $\beta\leq \alpha$,
using the compact-open topology on $C(T^\beta M,T^{|\beta|}N)$.
\end{defn}
Pushforwards and pullbacks
are continuous.
\begin{la}\label{pull-and-push}
Using compact open $C^\alpha$-topologies, we have:
\begin{itemize}
\item[\rm(a)]
If $L$ is a smooth manifold with rough boundary
and $g\colon N\to L$ a smooth map,
then the following map is continuous:
\[
g_*:=C^\alpha(M,g)\colon C^\alpha(M,N)\to C^\alpha(M,L),\;\,
f\mto g\circ f.
\]
\item[\rm(b)]
Let $L_j$ be a smooth manifold with rough boundary for
$j\in\{1,\ldots, n\}$ and $g_j\colon L_j\to M_j$
be a smooth map. Abbreviate $L:=L_1\times\cdots \times L_n$
and $g:=g_1\times\cdots\times g_n$.
Then the following map is continuous:
\[
g^*:=C^\alpha(g,N)\colon
C^\alpha(M,N)\to C^\alpha(L,N),\;\,
f\mto f\circ g.
\]
\end{itemize}
\end{la}
\begin{rem}\label{rem-restriction}
If $L_j$ is a full submanifold of~$M_j$
for $j\in\{1,\ldots, m\}$,
then the inclusion map $g_j\colon L_j\to M_j$, $x\mto x$
is smooth. By Lemma~\ref{pull-and-push}\,(b),
the map
\[
\rho:=C^\alpha(g_1\times\cdots\times g_n,N)\colon
C^\alpha(M,N)\to C^\alpha(L,N)
\]
is continuous, which is the restriction map
$C^\alpha(M,N)\to C^\alpha(L,N)$, $f\mto f|_L$.
\end{rem}
\begin{la}\label{la:ascending-union}
Let $(K_i)_{i\in I}$
be a family
of subsets $K_i\sub M$
whose interiors $K_i^o$
cover $M$,
such that $K_i=K_{i,1}\times\cdots\times K_{i,n}$
for certain full submanifolds
$K_{i,j}\sub M_j$ for $j\in\{1,\ldots, n\}$.
Then the compact-open $C^\alpha$-topology on $C^\alpha(M,N)$
is initial with respect to the restriction maps
$C^\alpha(M,N)\to C^\alpha(K_i,N)$ for $i\in I$.
\end{la}
\begin{la}\label{top-sub}
For $j\in \{1,\ldots, n\}$,
let $S_j$ be a full submanifold of~$M_j$.
Abbreviate $S:=S_1\times\cdots\times S_n$.
Then $T^\beta S$ is a full submanifold
of $T^\beta M$
for all $\beta\in (\N_0)^n$,
and the smooth manifold structure on $T^\beta S$
as the $\beta$-tangent bundle
of~$S$ coincides with the smooth manifold
structure as a full submanifold of $T^\beta M$.
Analogous conclusions $($with submanifolds in place of full submanifolds$)$
hold
if $\partial M_j=\emptyset$
for all $j\in\{1,\ldots, n\}$
and $S_j\sub M_j$
is a submanifold.
\end{la}
\begin{la}\label{maps-to-sub}
If $S$ is a full submanifold of~$N$
or $\partial N=\emptyset$ and $S\sub N$ is a submanifold,
then the compact-open $C^\alpha$-topology
on $C^\alpha(M,S)$ coincides with
the topology on $C^\alpha(M,S)$ induced by
$C^\alpha(M,N)$.
\end{la}
\begin{la}\label{maps-to-tvs-1}
If $F$ is a locally convex space,
then $C^\alpha(M,F)$ is a vector subspace of
$F^M$.
The compact-open $C^\alpha$-topology
makes $C^\alpha(M,F)$ a locally convex space.
\end{la}
\begin{la}\label{c-alpha-top-product}
Let $M_1,\ldots, M_n$ be smooth
manifolds with rough boundary, $M:=M_1\times\cdots\times M_n$,
and $\alpha\in (\N_0\cup\{\infty\})^n$.
\begin{itemize}
\item[\rm(a)]
If $F$ is a locally convex space
whose topology is initial with respect to a family
$(\lambda_i)_{i\in I}$
of linear mappings $\lambda_i\colon F\to F_i$
to locally convex spaces~$F_i$,
then the compact-open $C^\alpha$-topology
on $C^\alpha(M,F)$
is initial with respect to the mappings
$((\lambda_i)_*)_{i\in I}\colon C^\alpha(M,F)\to C^\alpha(M,F_i)$.
\item[\rm(b)]
If $F$ is a locally convex space
and $F=\prod_{i\in I}F_i$
for a family
$(F_i)_{i\in I}$
of locally convex spaces,
let $\pr_i\colon F\to F_i$
be the projection onto the $i$th component
and $(\pr_i)_*\colon C^\alpha(M,F)\to C^\alpha(M,F_i)$.
Then\vspace{-.7mm}
\[
\Theta:=((\pr_i)_*)_{i\in I}\colon C^\alpha(M,F)\to\prod_{i\in I}C^\alpha(M,F_i)\vspace{-.9mm}
\]
is an isomorphism of topological vector spaces.
\item[\rm(c)]
Assume that all of $M_1,\ldots, M_n$ are locally compact.
Let $N_i$ be a smooth manifold with rough
boundary for $i\in \{1,2\}$
and $\pr_i\colon N_1\times N_2\to N_i$
be the projection onto the $i$th component.
Using the compact-open $C^\alpha$-topology
on sets of $C^\alpha$-maps, we get a homeomorphism
\[
\Psi:=((\pr_1)_*,(\pr_2)_*)
\colon C^\alpha(M,N_1\times N_2)\to C^\alpha(M,N_1)\times C^\alpha(M,N_2) .
\]
\end{itemize}
\end{la}
Using the multiplication
$\R\times TN\to TN$, $(t,v)\mto tv$
with scalars, we have:
\begin{la}\label{c-alpha-top-mult}
Let $M_1,\ldots, M_n$ be locally compact smooth
manifolds with rough boundary, $M:=M_1\times\cdots\times M_n$,
$\alpha\in (\N_0\cup\{\infty\})^n$,
and $N$ be a smooth manifold with rough
boundary. Then the map
\[
\mu\colon C^\alpha(M,\R)\times C^\alpha(M,TN)\to C^\alpha(M,TN)
\]
determined by $\mu(f,g)(x):=f(x)g(x)$
is continuous.
\end{la}
In \cite{Alz},
Exponential Laws were provided
for function spaces on products of pure manifolds.
The one we need remains valid for
manifolds which need not be pure:
\begin{la}\label{exp-law-not-pure}
Let $N_1,\ldots, N_m$ and $M_1,\ldots, M_n$
be smooth manifolds
with rough boundary $($none of which needs to be pure$)$.
Let $\alpha\!\in \!(\N_0\cup\{\infty\})^m\!$,
$\beta\!\in\! (\N_0\cup\{\infty\})^n$
and $E$ be a locally convex space.
Abbreviate $N:=N_1\times \cdots\times N_m$ and
$M:=M_1\times\cdots\times M_n$.
For $f\in C^{\alpha,\beta}(N\times M,E)$,
we then have
$f_x:=f(x,\cdot)\in C^\beta(M,E)$ for each $x\in N$ and the map
$f^\vee\colon N\to C^\beta(M,E)$, $x\mto f_x$
is $C^\alpha$.
The map
\[
\Phi\colon C^{\alpha,\beta}(N\times M,E)\to
C^\alpha(N,C^\beta(M,E)),\;\,
f\mto f^\vee
\]
is linear and a homeomorphism onto its image.
If $M_j$ is locally compact for all
$j\in\{1,\ldots,n\}$, then $\Phi$ is a homeomorphism.
The inverse map~$\Phi^{-1}$ sends $g\in C^\alpha(N,C^\beta(M,E))$
to the map $g^\wedge$ defined via
$g^\wedge(x,y):=g(x)(y)$.
\end{la}
We mention that the $C^\alpha$-topology on $C^\alpha(U,F)$ can be described more
explicitly.
\begin{la}\label{maps-to-tvs-2}
Let $E_j$ be a
locally convex space for $j\in\{1,\ldots, n\}$
and $U_j\sub E_j$ be a locally convex subset with dense
interior.
Let $F$ be a locally convex space,
$\alpha\in (\N_0\cup\{\infty\})^n$,
and $U:=U_1\times\cdots\times U_n$.
Then the compact-open $C^\alpha$-topology
on $C^\alpha(U,F)$
is initial with respect to the maps
\[
d^\beta\colon C^\alpha(U,F)\to C(U\times E_1^{\beta_1}\times\cdots
\times E_n^{\beta_n},F),\;\, f\mto d^\beta f
\]
for $\beta\in(\N_0)^n$ with $\beta\leq\alpha$,
using the compact-open topology on
the ranges.
\end{la}
\section{(Pre-)Canonical manifold structures}\label{sec-canonical}
In this section, we establish basic
properties of canonical manifolds of mappings,
and pre-canonical ones. We begin with examples.
\begin{example}\label{basic-examples}
 Let $n \in\N$ and $\alpha\in (\N_0\cup \{\infty\})^n$.
 \begin{enumerate}
  \item[(a)] Let $M_1,\ldots, M_n$ be locally compact smooth
manifolds with rough boundary
and $E$ a locally convex space. Then $C^\alpha (M_1 \times \cdots \times M_n,E)$ is a canonical manifold due to Lemma~\ref{exp-law-not-pure}.
The same holds for $C^\alpha (M_1 \times \cdots \times M_n,N)$ if $N$ is a smooth manifold diffeomorphic to $E$,
endowed with the $C^\infty$-manifold structure making $\varphi_* \colon C^\alpha (M,N) \rightarrow C^\alpha (M,E)$ a diffeomorphism, where $\varphi \colon E \rightarrow N$ is a $C^\infty$-diffeomorphism.
  \item[(b)] Familiar examples of mapping groups turn out to be canonical,
notably loop groups $C^k (\mathbb{S}^1,G)$ for $G$ a Lie group, and certain Lie groups of the form $C^k(\R,G)$ discussed in \cite{Hmz,NaW}. We extend these constructions in
Section~\ref{sec-Lie}.
 \end{enumerate}
\end{example}

We will now establish general properties of canonical manifolds.

\begin{numba}\textbf{Conventions} We denote by $\alpha,\beta$ multiindices in $(\N_0\cup \{\infty\})^n$ for some $n \in \N$. Likewise we will usually adopt the shorthand $M \coloneq M_1 \times M_2 \times \cdots \times M_n$ where the $M_i$ are locally compact manifolds (possibly with rough boundary). If $M$ is the domain of definition of the function space $C^\alpha (M,N)$ we will assume that the number of entries of the multiindex $\alpha$ coincides with the number of factors in the product $M$.
\end{numba}

\begin{la}\label{base-cano}
If $C^\alpha (M,N)$ is endowed with a pre-canonical manifold structure, then
the following holds:
\begin{itemize}
\item[\textup{(a)}]
The evaluation map $\ev\colon C^\alpha(M,N)\times M\to N$, $\ev(\gamma,x):=\gamma(x)$ is $C^{\infty,\alpha}$.
\item[\textup{(b)}]
Pre-canonical manifold structures are unique in the following sense:
If we write $C^\alpha(M,N)'$ for $C^\alpha(M,N)$ with another pre-canonical manifold structure,
then $\id\colon C^\alpha(M,N)\to C^\alpha(M,N)'$, $\gamma\mto\gamma$ is a $C^\infty$-diffeomorphism.
\item[\textup{(c)}]
Let $S\sub N$ be a submanifold such that the set $C^\alpha(M,S)$ is a submanifold of $C^\alpha(M,N)$.
Then the submanifold structure on $C^\alpha(M,S)$ is pre-canonical.
\end{itemize}
\end{la}
\begin{proof}
(a) Since $\id\colon C^\alpha(M,N)\to C^\alpha(M,N)$ is~$C^\infty$ and $C^\alpha(M,N)$
is endowed with a pre-canonical manifold structure, it follows that $\id^\wedge\colon
C^\alpha(M,N)\times M\to N$, $(\gamma,x)\mto \id(\gamma)(x)=\gamma(x)=\ev(\gamma,x)$
is $C^{\infty,\alpha}$.

(b) The map $f:=\id\colon C^\alpha(M,N)\to C^\alpha(M,N)'$ satisfies
$f^\wedge=\ev$ where $\ev\colon C^\alpha(M,N)\times M\to N$ is~$C^{\infty,\alpha}$, by~(a). Since $C^\alpha(M.N)'$ is endowed with a pre-canonical manifold structure, it follows that~$f$ is~$C^\infty$. By the same reasoning,
$f^{-1}=\id\colon C^\alpha(M,N)'\to C^\alpha(M,N)$ is~$C^\infty$.

(c) As $C^\alpha(M,S)$ is a submanifold of $C^\alpha(M,N)$, the inclusion $\iota\colon C^\alpha(M,S)\to C^\alpha(M,N)$ is~$C^\infty$.
Likewise, the inclusion map $j\colon S\to N$ is~$C^\infty$.
Let $L =L_1 \times \cdots \times L_k$ be a product of smooth manifolds
(possibly with rough boundary)
modeled on locally convex spaces and $f\colon L\to C^\alpha(M,S)$ be a map.
If~$f$ is~$C^\beta$, then $\iota\circ f$ is~$C^\beta$, entailing that $(\iota\circ f)^\wedge\colon L\times M\to N$, $(x,y)\mto f(x)(y)$ is~$C^{\beta,\alpha}$.
As the image of this map is contained in~$S$, which is a submanifold of~$N$, we deduce that $f^\wedge=(\iota\circ f)^\wedge|^S$ is~$C^{\beta,\alpha}$.
For the converse, assume that $f^\wedge\colon L\times M\to S$ is $C^{\beta,\alpha}$. Then also $(\iota\circ f)^\wedge=j\circ (f^\wedge)\colon L\times M\to N$ is $C^{\beta,\alpha}$.
Hence $\iota\circ f\colon L\to C^\alpha(M,N)$ is~$C^\beta$ (the manifold structure on the range being pre-canonical). As $\iota\circ f$ is a $C^\beta$-map with image in $C^\alpha(M,S)$ which is a submanifold
of $C^\alpha(M,N)$, we deduce that~$f$ is~$C^\beta$.
\end{proof}

\begin{rem}
Note that due to Lemma \ref{base-cano}\,(a), the evaluation on a canonical manifold is a $C^{\infty,\alpha}$-map whence it is at least continuous. 
For a $C^k$-manifold~$M$
which is $C^k$-regular\footnote{Meaning that the topology
on~$M$ is initial with respect to $C^k(M,\R)$.
This holds if $M$ is a regular topological space
and all modeling spaces are $C^k$-regular, see \cite{GaN}.}
and a locally convex space $E\not=\{0\}$,
it is well known that for the compact-open $C^k$-topology the evaluation $\ev \colon C^k (M,E)\times M\rightarrow E$ is continuous
if and only if $M$ is locally compact. A similar statement holds for the compact-open $C^\alpha$-topology. Using a chart for~$N$ and
cut-off functions,
we deduce that the evaluation of
$C^\alpha (M,N)$ is discontinuous
if $M$ fails to be locally compact,
provided $N$ is not discrete
and $M$ is $C^{|\alpha|}$-regular;
then $C^\alpha(M,N)$ cannot admit a canonical
manifold structure.
\end{rem}

We now turn to smoothness properties of the composition map. 

\begin{la}
Assume that
$C^{|\alpha|+s}(N,L)$, $C^\alpha (M,N)$, and $C^\alpha (M,L)$ are endowed
with pre-canonical manifold structures. Then the
composition map 
$$\comp\colon C^{|\alpha|+s}(N,L) \times C^\alpha (M,N) \rightarrow C^\alpha (M,L),\quad (f,g) \mapsto f\circ g$$
is a $C^{\infty,s}$-map, for every $s\in \N_0\cup\{\infty\}$. 
\end{la}

\begin{proof}
Since $C^\alpha (M,L)$ is pre-canonical, $\comp$ is $C^{\infty,s}$ if and only if 
$$\comp^\wedge \colon  C^{|\alpha|+s}(N,L) \times C^\alpha (M,N) \times M \rightarrow L,\quad (f,g,x)\mapsto f(g(x))$$
is a $C^{\infty,s,\alpha}$-map. The formula shows that $\comp^\wedge (f,g,x) = \ev(f,\ev(g,x))$, where the outer evaluation map is
$C^{\infty, |\alpha|+s}$ and the inner one $C^{\infty,\alpha}$,
by Lemma
\ref{base-cano}\,(a),
as $C^{|\alpha|+s}(N,L)$ and $C^\alpha (M,N)$ are pre-canonical manifolds.
Using the chain rule \cite[Lemma~3.16]{Alz},
we deduce that $\comp^\wedge$ is $C^{\infty,s,\alpha}$.
\end{proof}

\begin{cor}\label{cor:pushforward}
 If $C^{\alpha} (M,N)$ and $C^\alpha (M,L)$ are endowed with pre-canonical manifold structures, then the pushforward $f_\ast \colon C^{\alpha} (M,N) \rightarrow C^\alpha (M,L),\ g \mapsto  f\circ g$ is a $C^s$-map for every $f \in C^{|\alpha|+s}(N,L)$.
\end{cor}

\begin{cor}\label{cor:PB}
 Let $C^{|\alpha|+s}(N,L)$ and $C^\alpha (M,L)$ be endowed with pre-canonical manifold structures. For a $C^\alpha$-map $g \colon M \rightarrow N$ the pullback $g^\ast\colon C^{|\alpha|+s}(N,L) \rightarrow C^\alpha (M,L), \ f \mapsto  f\circ g$ is smooth for every $s \in \N_0$.
\end{cor}

The chain rule also allows the following result to be deduced.
\begin{la}\label{la:pb}
 Let $C^\alpha(M,N)$ and $C^\alpha(L,N)$ be endowed with pre-canonical manifold structures where $\alpha = (\alpha_1,\ldots, \alpha_n)$, $M=M_1 \times \cdots \times M_n$ and $L=L_1\times \cdots \times L_n$. Assume that $g_i \colon L_i \rightarrow M_i$ is a $C^{\alpha_i}$-map for $i\in\{1,\ldots,n\}$. Then the pullback
 $$g^* \colon C^\alpha (M,N) \rightarrow C^\alpha (L,N), \;\,
 f \mapsto f \circ (g_1 \times \cdots \times g_n)$$
with $g:=g_1\times\cdots\times g_n$ is smooth.
\end{la}

\begin{proof}
Due to the chain rule, the pullback $g^*$ makes sense. Since $C^\alpha (L,N)$ is
pre-canonical, $g^*$ will be smooth if  
 $(g^*)^\wedge\colon  (f,\ell) \mto
 \ev(f,\ev( (g_1 \times \cdots \times g_n),\ell))$ is a $C^{\infty,\alpha}$-map. Again, this 
is a consequence of Lemma \ref{base-cano}\,(a). 
\end{proof}

The key point was the differentiability of the evaluation map together with a suitable chain rule. Thus, by essentially the same proof, one obtains from the chain rule \cite[Lemma 3.16]{Alz} the following statement whose proof we omit.

\begin{prop}
 Assume that all the manifolds of mappings occurring in the following are endowed with pre-canonical manifold structures. Further, we let $\beta = (\beta_1 ,\ldots , \beta_n) \in (\N_0\cup \{\infty\})^n$ such that for multiindices $\alpha^i \in (\N_0\cup \{\infty\})^{m_i}$, $i\in \{1,\ldots,n\}$ we have $\beta_i = |\alpha^i|+\sigma_i$ for some
$\sigma_i \in \N_0\cup\{\infty\}$. Let now $N = \prod_{1\leq i\leq n} N_i$ and
$M^i:=M^i_1\times\cdots\times M^i_{m_i}$
for certain locally compact manifolds $M^i_j$
with rough boundary
$($with $j\in \{1,\ldots, m_i\})$.
Then for $\sigma= (\sigma_1,\ldots,\sigma_n)$
and $\alpha=(\alpha^1,\ldots, \alpha^n)$, the composition map
  \begin{align*}
   C^\beta (N,L)\times \prod_{1\leq i \leq n} C^{\alpha^i} (M^i,N_i) &\rightarrow C^\alpha (M^1\times \cdots \times M^n,L),\\
    (f,g_1,\ldots,g_n) &\mapsto f \circ (g_1\times \cdots \times g_n)
  \end{align*}
 is a $C^{\infty,\sigma}$-map.
\end{prop}

The above discussion shows that composition, pushforward, and pullback maps inherit differentiability and continuity properties.
The following variant will be used
in the construction process of canonical manifold structures.

\begin{prop}\label{fstar-gen}
Let $K$ be a compact smooth manifold such that $C^\alpha(K,M)$ and $C^\alpha(K,N)$ admit canonical manifold structures.
If $\Omega\sub K\times M$ is an open subset and $f\colon \Omega \to N$ is a $C^{|\alpha|+k}$-map,
then
\[
\Omega':=\{\gamma\in C^\alpha(K,M)\colon \graph(\gamma)\sub\Omega\}
\]
is an open subset of $C^\alpha(K,M)$ and
\[
f_\star\colon \Omega \to C^\alpha(K,N),\;\, \gamma\mto f\circ (\id_K,\gamma)
\]
is a $C^k$-map.
\end{prop}
\begin{proof}
By compactness of $K$, the compact-open topology on $C(K,M)$ coincides with the graph topology (see, e.g., \cite[Proposition A.6.25]{GaN}).
Thus $\{\gamma\in C(K,M)\colon \graph(\gamma)\sub \Omega\}$ is open in $C(K,M)$. As a consequence, $\Omega'$ is open in $C^\alpha(K,M)$.
By Lemma~\ref{base-cano}\,(a), the evaluation $\ev\colon C^\alpha(K,M)\times K\to M$ is $C^{\infty,\alpha}$ and hence~$C^{k,\alpha}$,
whence also $C^\alpha(K,M)\times K\to K\times M$, $(\gamma,x)\mto (x,\gamma(x))$ is $C^{k,\alpha}$. Since $f$ is $C^{|\alpha|+k}$, the Chain Rule \cite[Lemma 3.16]{Alz} shows that
\[
(f_\star)^\wedge\colon \Omega' \times K\to N, \;\,
(\gamma,x)\mto f_\star(\gamma)(x)
=f(x,\gamma(x))
\]
is $C^{k,\alpha}$. So $f_\star$ is~$C^k$,
as the manifold structure on $C^\alpha(K,N)$ is canonical.
\end{proof}

For later use we record several observations on stability of (pre-)canonical structures under pushforward by diffeomorphisms.
\begin{la}\label{la:ISO1} 
Let $N_1$ and $N_2$ be smooth manifolds and $\alpha\in (\N_0\cup\{\infty\})^n,\beta\in(\N_0\cup\{\infty\})^m$. 
 \begin{itemize}
  \item[\textup{(a)}] If $C^\alpha(M,N_1)$ and $C^\alpha(M,N_2)$ are endowed with
$($pre-$)$canonical manifold structures, then the smooth manifold structure
on $C^\alpha(M,N_1 \times N_2)$ which turns the bijection $C^\alpha(M,N_1 \times N_2)\rightarrow C^\alpha(M,N_1)\times C^\alpha(M,N_2)$ sending a mapping to the pair of component functions into a $C^\infty$-diffeomorphism, is $($pre-$)$canonical.
  \item[\textup{(b)}] If $\psi \colon N_1 \rightarrow N_2$ is a $C^\infty$-diffeomorphism and $C^\alpha (M,N_2)$ is a $($pre-$)$canonical manifold, then the smooth manifold structure
on $C^\alpha(M,N_1)$
turning the\linebreak
bijection $$\psi_* \colon C^\alpha (M,N_1) \rightarrow C^\alpha (M,N_2), \;\, f\mapsto \psi \circ f$$ into a diffeomorphism is $($pre-$)$canonical.
  \item[\textup{(c)}] Let $C^\alpha (M,N)$ be endowed with a pre-canonical manifold structure and assume that both $C^\beta (L,C^\alpha(M,N))$ and $C^{\beta,\alpha}(L\times M,N)$ are smooth manifolds making the bijection 
  $$\Phi \colon C^{\beta,\alpha}(L\times M,N) \rightarrow C^\beta (L,C^\alpha(M,N)),\quad f \mapsto f^\vee$$
  a $C^\infty$-diffeomorphism. Then
$C^\beta (L,C^\alpha(M,N))$ is pre-canonical if and only if $C^{\beta,\alpha}(L\times M,N)$ is pre-canonical.
  \end{itemize}
  \end{la}

  \begin{proof}
Let $L = L_1 \times \cdots \times L_m$ be a product of manifolds.

(a) A map $f=(f_1,f_2)\colon L\to C^\alpha(M,N_1)\times C^\alpha(M,N_2)$ is~$C^\beta$ if and only if~$f_1$ and $f_2$ are~$C^\beta$.
As the manifold structures are (pre-)canonical, this holds if and only if
$f_i^\wedge\colon L\times M\to M_i$ is $C^{\beta,\alpha}$ for $i\in \{1,2\}$. However, this holds if and only if $f^\wedge=(f_1^\wedge,f_2^\wedge)$ is $C^{\beta,\alpha}$.

(b) A map $f \colon L \rightarrow C^\alpha(M,N_1)$ is $C^\beta$ if and only if $\psi_*\circ f$ is $C^\beta$. Since $C^\alpha(M,N_2)$ is pre-canonical, this is the case if and only if $(\psi_*\circ f)^\wedge = \psi\circ f^\wedge$ is $C^{\beta,\alpha}$. As $\psi$ is a smooth diffeomorphism we deduce from the chain rule that this is the case if and only if $f^\wedge$ is of class $C^{\beta,\alpha}$. Thus $C^{\alpha}(M,N_1)$ is pre-canonical. 
If $C^{\alpha}(M,N_2)$ is even canonical, the $C^\alpha$-topology is transported by the diffeomorphism $\psi_*$ to the $C^\alpha$-topology on $C^\alpha (M,N_1)$. Hence the manifold $C^\alpha (M,N_1)$ is also canonical in this case. 

(c) By construction, a map $f \colon K\rightarrow C^{\beta,\alpha}(L\times M,N)$ is
of class $C^\gamma$ (for some multiindex $\gamma$) if and only if $\Phi\circ f = (f(\cdot))^\vee$ is $C^\gamma$ as a mapping to $C^\beta (L,C^\alpha(M,N))$. As $C^\alpha (M,N)$ is pre-canonical, we observe that $(\Phi\circ f)^\wedge \colon K \times L \rightarrow  C^\alpha(M,N)$ is $C^{\gamma,\beta}$ if and only if $((\Phi\circ f)^\wedge)^\wedge = f^\wedge \colon K \times L\times M \rightarrow N$ is a $C^{\gamma,\beta,\alpha}$-map. Hence $C^{\beta,\alpha}(L\times M,N)$ is pre-canonical (i.e.\ $f$ is $C^\gamma$ if and only if $f^\wedge$ is $C^{\gamma,\beta,\alpha}$) if and only if $C^{\beta}(L,C^{\alpha}(M,N))$ is pre-canonical.
\end{proof}

\begin{la}\label{la:ISO2}
Fix $\alpha\in (\N_0\cup\{\infty\})^n$ and a permutation $\sigma$ of $\{1,\ldots,n\}$. Denote by $\phi_\sigma \colon M_1 \times \cdots \times M_n \to Q \coloneq M_{\sigma (1)} \times \cdots \times M_{\sigma (n)}$ the diffeomorphism taking $(x_i)_{i=1}^n$ to
$(x_{\sigma (i)})_{i=1}^n$. 
\begin{itemize}
  \item[\textup{(a)}] If $C^{\alpha\circ\sigma}(Q,N)$ and $C^\alpha(M,N)$
are smooth manifolds such that the bijection 
$$\phi^*_\sigma\colon C^{\alpha\circ\sigma}(Q,N) \rightarrow C^\alpha(M,N), \quad f \mapsto f\circ \phi_\sigma$$ 
from {\rm\ref{reorder}} becomes a diffeomorphism, then $C^\alpha(M,N)$ is
$($pre-$)$canonical if and only if $C^{\alpha\circ\sigma}(Q,N)$ is $($pre-$)$canonical. 
 \item[\textup{(b)}] If $C^\alpha (M,N)$ and $C^{\alpha\circ\sigma} (Q,N)$
are endowed with pre-canonical manifold structures,
then $\phi_\sigma^*$ is a $C^\infty$-diffeomorphism.
  \item[\textup{(c)}] If $\psi_i \colon L_i \rightarrow M_i$ is a smooth diffeomorphism for every $i \in \{1,\ldots,n\}$ and $C^{\alpha}(M,N)$ is $($pre-$)$canonical, then the smooth manifold structure on $C^\alpha(L,N)$ turning the bijection 
  $$(\psi_1\times \cdots \times \psi_n)^* \colon C^\alpha (M,N) \rightarrow C^\alpha (L,N)$$ into a diffeomorphism is $($pre-$)$canonical.
 \end{itemize}
\end{la}

\begin{proof}
(a) Assume that $C^\alpha(M,N)$ is (pre-)canonical. Then $f \colon K\rightarrow C^{\alpha\circ\sigma}(Q,N)$ is $C^\beta$ if and only if $\phi_\sigma^*\circ f$
is so. Now we deduce from $C^\alpha(M,N)$ being pre-canonical that this is equivalent to $(\phi_\sigma^*\circ f)^\wedge = f^\wedge \circ (\id_K\times
\phi_\sigma) \colon K \times M \rightarrow N$ being a $C^{\beta,\alpha}$-map.
Exploiting the Theorem of Schwarz \cite[Proposition~3.5]{Alz},
this is equivalent to $f^\wedge$ being $C^{\beta,\alpha\circ\sigma}$.
Thus
$C^{\alpha\circ \sigma}(Q,N)$ is pre-canonical. The converse assertion for $C^{\alpha \circ \sigma} (M,N)$ follows verbatim by replacing $\phi_\sigma$ with its inverse. Note that if one of the manifolds is even canonical, it follows directly from the definition of the $C^\alpha$-topology, Definition \ref{defn-alpha-top}, that reordering the factors induces a homeomorphism of the $C^{\alpha}$- and $C^{\alpha \circ \sigma}$-topology. Hence we see that one of the manifolds is canonical if and only if the other is so.

(b) Note that the inverse of $\phi_\sigma^*$ is $(\phi_\sigma^{-1})^*$ whence the situation is symmetric and it suffices to prove that $\phi_\sigma^*$ (and by an analogous argument also its inverse) is smooth. As $C^\alpha (M,N)$ is pre-canonical, smoothness of $\phi_\sigma^*$ is equivalent to $(\phi^*_\sigma)^\wedge \colon C^{\alpha\circ\sigma}(Q,N)
\times M \rightarrow N,\ (f,m) \mapsto \ev(f,\phi_\sigma (m))$ being a $C^{\infty,\alpha}$-mapping.
This follows from Lemma~\ref{base-cano}\,(a),
the chain rule, and Lemma~\ref{ingrisch}.

(c) Replacing $\phi_\sigma$ with $\psi_1\times \cdots \times
\psi_n$, the argument is analogous to (b). If $C^\alpha(M,N)$ is canonical, then the $C^\alpha$-topology pulls back to the $C^\alpha$-topology
under the diffeomorphism, by Lemma~\ref{pull-and-push}.
\end{proof}

An exponential law is available for pre-canonical smooth manifold structures.
\begin{prop}\label{explaw-precan}
Let $L_1,\ldots, L_m$ and $N$ be smooth manifolds with rough boundary, and $M_1,\ldots, M_n$ be locally compact smooth manifolds with rough boundary.
Assume that $C^\alpha(M,N)$ is endowed with a pre-canonical smooth manifold structure and also
$C^\beta(L,C^\alpha(M,N)$ and $C^{\beta,\alpha}(L\times M,N)$ are endowed with pre-canonical smooth manifold structures.
Then the bijection
\[
\Phi\colon C^{\beta,\alpha}(L\times M,N)\to C^\beta(L,C^\alpha(M,N))
\]
from {\rm(\ref{pre-can-bij})} is a $C^\infty$-diffeomorphism.
\end{prop}

\begin{proof}
If we give $C^\beta(L,C^\alpha(M,N))$ the smooth manifold structure making $\Phi$ a $C^\infty$-diffeomorphism,
then this structure is pre-canonical
by Lemma~\ref{la:ISO1}\,(c).
It therefore coincides with the given pre-canonical
smooth manifold structure thereon,
up to the choice of modeling spaces (Lemma~\ref{base-cano}\,(b)).
\end{proof}

\noindent
There is a natural identification of tangent vectors for pre-canonical manifolds,
in good cases.
If $C^\alpha (M,N)$ is pre-canonical, an element
$v \in T_f C^\alpha (M,N)$ corresponds to
an equivalence class of curves $\gamma_v \colon I \rightarrow C^\alpha (M,N)$ on some open interval $I$ around $0$ such that $\gamma_v(0)=f$ and $\dot{\gamma}_v(0)=v$. As $C^\alpha (M,N)$ is pre-canonical, the map $\gamma^\wedge_v \colon I \times M \rightarrow N$ is $C^{1,\alpha}$. Hence $T\ve_m(v)=T\ve_m(\dot{\gamma}_v(0))\in TN$
is $C^\alpha$ in $m\in M$, where we use
the point evaluation
$\ve_m\colon C^\alpha(M,N)\to N$,
$f\mto f(m)$
at~$m$.
We thus obtain a map 
\begin{align}\label{tangent:iso}
 \Psi \colon TC^\alpha (M,N) \rightarrow C^\alpha (M,TN),\ v \mapsto (m \mapsto T\ve_m(v)).
\end{align}
Under additional assumptions, one can show that $\Psi$ is a diffeomorphism,
allowing tangent vectors $v\in TC^\alpha(M,N)$
to be identified with $\Psi(v)$.
We will encounter a setting in which this statement becomes true in the next section (see Theorem \ref{thm:tangentident}).
\section{Constructions for compact domains}\label{sec-compact}
We now construct and study manifolds of $C^\alpha$-mappings on compact domains.
The results of this section subsume Theorem~\ref{thmA}.
They generalize constructions for
$C^{k,\ell}$-functions in \cite[Appendix A]{AGS}.

\begin{numba}\label{bundleconv}
Let $N$ be a smooth manifold, $\alpha \in(\N_0\cup\{\infty\})^n$ and $M = M_1 \times \cdots \times M_n$ be a locally compact smooth manifold with rough boundary.
If $\pi\colon E\to N$ is a smooth vector bundle over~$N$ and $f\colon M\to N$ is a $C^\alpha$-map, then we define 
\[
\Gamma_f \coloneq \{\tau\in C^\alpha(M,E)\colon \pi\circ \tau=f\}
\]
with the topology induced by $C^\alpha(M,E)$. Pointwise operations turn $\Gamma_f$ into a vector space. Let us prove that $\Gamma_f$ is a locally convex space. To this end, we cover $N$ with open sets $(U_i)_{i \in I}$ on which the restriction $E|_{U_i} \cong U_i \times E_i$ (with $E_i$ a suitable locally convex space) is trivial. Combining continuity of $f$ and local compactness of $M$ we can find families $\mathcal{K}_j$ of full compact submanifolds of $M_j$ with the following properties: 
The interiors of the sets in $\mathcal{K}_j$ cover $M_j$. There is a set $\mathcal{K} \subseteq \prod_{1 \leq j   \leq n} \mathcal{K}_j$ such that for every $K = K_1 \times \cdots \times K_n \in \mathcal{K}$ we have $f(K) \subseteq U_{i_K}$ for some $i_K \in I$ and the interiors of the submanifolds in $\mathcal{K}$ cover~$M$. Hence we deduce from Lemma \ref{la:ascending-union} that the map   
\begin{align*}
 \Psi\colon C^\alpha (M,E) \to\prod_{K \in \mathcal{K}} C^\alpha (K,E),\;\, \sigma\mto (\sigma|_K)_{K \in \mathcal{K}}
\end{align*}
is a topological embedding. Now by construction $\Gamma_f$ is contained in the open subset $\{G \in C^\alpha (M,E) \mid G(K)\subseteq \pi^{-1}(U_{i_K}), \forall K \in \mathcal{K}\}$. Restricting $\Psi$ to this subset we obtain a topological embedding
\begin{align}\label{top:emb}
 e \colon \Gamma_f \rightarrow \prod_{K \in \mathcal{K}} C^\alpha (K,\pi^{-1}(U_{i_K})) \cong \prod_{K \in \mathcal{K}} C^\alpha (K,U_{i_K}) \times C^\alpha (K,E_{i_K}),
\end{align}
where the identification exploits Lemma~\ref{c-alpha-top-product}
and the fact that pushforwards with smooth diffeomorphisms induce homeomorphisms of the $C^\alpha$-topology (see Lemma \ref{pull-and-push}).
The image of $e$ are precisely the mappings which coincide on the intersections of the compact sets $K$ (see \eqref{identify-image} and the explanations there). Hence we can exploit that point evaluations are continuous on $C^\alpha (K,E_{i_K})$ by \cite[Proposition 3.17]{Hmz} to see that the image of $e$ is a closed vector subspace of $\prod_{K \in \mathcal{K}} \{f|_K\} \times C^\alpha (K,E_{i_K})$. As the space on the right hand side is locally convex, we deduce that the co-restriction of $e$ onto its image is an isomorphism of locally convex spaces. Thus $\Gamma_f$ is a locally convex topological vector space. 

We will sometimes write $\Gamma_f (E)$ instead of $\Gamma_f$ to emphasize the dependence on the bundle $E$.
\end{numba}

The previous setup allows an essential Exponential Law
to be deduced.
\begin{la}\label{expsect}
In the situation of {\rm\ref{bundleconv}},
let $\beta\in (\N_0\cup\{\infty\})^m$ and $g\colon L\to \Gamma_f$ be a map,
where
$L_1,\ldots, L_m$ are smooth manifolds
with rough boundary and $L:=L_1\times\cdots\times L_m$.
Then $g$ is $C^\beta$ if and only if
\[
g^\wedge\colon L\times M\to E,\quad (x,y)\mto g(x)(y)
\]
is a $C^{\beta,\alpha}$-map.
\end{la}
\begin{proof}
With the notation as in \ref{bundleconv} we identify $\Gamma_f$ via $e$ with a closed subspace of the locally convex space $\prod_{K \in \mathcal{K}} C^\alpha (K, E_{i_K})$ (the identification will be suppressed in the notation). Thus Lemma \ref{la-sub-PL} (a) implies that the map $g$ is $C^\beta$ if and only if the components $g_K \colon L \rightarrow C^\alpha (K,E_{i_K})$ are $C^\beta$-maps. By the Exponential Law \cite[Theorem~4.4]{Alz}, the latter holds if and only if the mappings
\[
(g_K)^\wedge\colon L\times K\to E_{i_K},\quad (x,y)\mto g(x)(y)
\]
are of class $C^{\beta,\alpha}$. Since the interiors of sets $K \in \mathcal{K}$ cover $M$, we deduce that this is the case if and only if 
$g^\wedge$ is of class $C^{\beta,\alpha}$. 
\end{proof}
\begin{rem}
If all fibres of~$E$ are Fr\'{e}chet spaces and $K$ is $\sigma$-compact and locally compact, then $\Gamma_F$ is a Fr\'{e}chet space; if all fibres of~$E$ are Banach spaces,
$K$ is compact, and $|\alpha|<\infty$,
then $\Gamma_f$ is a Banach space. To see this, note that we can choose the family $\mathcal{K}$ in \ref{bundleconv} countable (resp., finite). 
Supressing again the identification,
\[
\psi\colon \Gamma_f\to\prod_{j\in J}C^\alpha(K_j,F_j),\quad \tau\mto (\tau|_{K_j})_{j\in J}
\]
is linear and a topological embedding with closed image. If all $F_j$ are Fr\'{e}chet
spaces, so is each $C^\alpha(K_j,F_j)$ (cf., e.g., \cite{GaN}) and hence also $\Gamma_f$. If all $F_j$ are Banach spaces and
$|\alpha|$ as well as $J$ is finite, then each $C^\alpha(K_j,F_j)$ is a Banach space
(cf.\ \emph{loc.\ cit.})
and hence also $\Gamma_f$.
\end{rem}
Observe that the exponential law for $\Gamma_f$ gives this space the
defining property of a pre-canonical manifold (and the only reason we do not call it pre-canonical is that it is only a subset of $C^\alpha (M,E)$). In particular, the proof of Lemma \ref{base-cano}\,(a) carries over and yields:

\begin{la}\label{evalsect}
In the situation of {\rm \ref{bundleconv}}, the evaluation map
\[
\ev\colon\Gamma_f\times M \to E,\quad (\tau,x)\mto \tau(x)
\]
is $C^{\infty,\alpha}$.
\end{la}

\begin{la}\label{Gammfunct}
Let $\pi_1\colon E_1 \to N$ and $\pi_2\colon E_2\to N$ be smooth vector bundles
over a smooth manifold~$N$. Let $\alpha\in(\N_0\cup\{\infty\})^m$ and $f\colon M\to N$
be a $C^\alpha$-map on a
product $M=M_1\times \cdots\times M_n$
of smooth manifolds with rough boundary.
Then the following holds:
\begin{itemize}
\item[\textup{(a)}]
If $\psi\colon E_1\to E_2$ is a mapping of smooth vector bundles over~$\id_M$,
then $\psi\circ\tau\in\Gamma_f(E_2)$ for each $\tau\in\Gamma_f(E_1)$ and
\[
\Gamma_f(\psi)\colon \Gamma_f(E_1)\to\Gamma_f(E_2),\quad \tau\mto \psi\circ \tau
\]
is a continuous linear map.
\item[\textup{(b)}]
$\Gamma_f(E_1\oplus E_2)$ is canonically isomorphic to $\Gamma_f(E_1)\times \Gamma_f(E_2)$.
\end{itemize}
\end{la}
\begin{proof}
(a) If $\tau\in \Gamma_f(E_1)$, then $\psi\circ\tau\colon M\to E_2$ is $C^\alpha$ by the chain rule and $\pi_2\circ
\psi\circ \tau =\pi_1\circ\tau=f$, whence $\psi\circ \tau\in \Gamma_f(E_2)$. Evaluating at points we see that the map
$\Gamma_f(\psi)$ is linear; being a restriction of the continuous map $C^\alpha(M,\psi)\colon
C^\alpha(M,E_1)\to C^\alpha(M,E_2)$ (see Lemma \ref{pull-and-push}), it is continuous.

(b) If $\rho_j\colon E_1\oplus E_2\to E_j$
is the map taking $(v_1,v_2)\in E_1 \times E_2$ to $v_j$ for $j\in \{1,2\}$
and $\iota_j\colon E_j\to E_1\oplus E_2$ is the map taking $v_j\in E_j$
to $(v_1,0)$ and $(0,v_2)$, respectively, then
\[
(\Gamma_f(\rho_1),\Gamma_f(\rho_2))\colon \Gamma_f(E_1\oplus E_2)\to\Gamma_f(E_1)\times
\Gamma_f(E_2)
\]
is a continuous linear map which is a homeomorphism as it has the continuous map $(\sigma,\tau)\mto \Gamma_f(\iota_1)(\sigma)+
\Gamma_f(\iota_2)(\tau)$ as its inverse.
\end{proof}
\subsection*{Construction of the canonical manifold structure}
Having constructed spaces of $C^\alpha$-sections as model spaces, we are now in a position to construct the canonical manifold structure on $C^\alpha(K,M)$,
assuming that~$M$ is covered by local additions and $K$ is compact. 
\begin{defn}\label{def-loa}
Let $M$ be a smooth manifold. A \emph{local addition} is a smooth map
\[
\Sigma \colon U \to M,
\]
defined on an open neighborhood $U \subseteq TM$ of the zero-section
$0_M:=\{0_p\in T_pM\colon p\in M\}$
such that $\Sigma(0_p)=p$ for all $p\in M$,
\[
U':=\{(\pi_{TM}(v),\Sigma(v))\colon v\in U\}
\]
is open in $M\times M$ (where $\pi_{TM}\colon TM\to M$ is the bundle
projection) and the map
\[
\theta:=(\pi_{TM},\Sigma)\colon U \to U'
\]
is a $C^\infty$-diffeomorphism. If
\begin{equation}\label{bettersigma}
T_{0_p}(\Sigma|_{T_pM})=\id_{T_pM}\;\,
\mbox{for all $p\in M$,}
\end{equation}
we say that the local addition $\Sigma$ is \emph{normalized}.
\end{defn}
Until Lemma~\ref{la:cano:mfdmap}, we fix the following setting, which allows a canonical
manifold structure on $C^\alpha(K,M)$ to be constructed.
\begin{numba}\label{thesetA}
We consider a product
$K=K_1 \times K_2 \times \cdots \times K_n$
of compact smooth manifolds with rough boundary,
a smooth manifold $M$ which admits a local addition $\Sigma \colon TM\supseteq U \rightarrow M$,
and $\alpha \in (\N_0 \cup \{\infty\})^n$.
\end{numba}
\begin{numba}\textbf{Manifold structure on $\boldmath{C^\alpha (K,M)}$ if $\boldmath{M}$ admits a local addition}\label{numba:mfdstruct}\\
For $f\in C^\alpha(K,M)$, let $\Gamma_f\coloneq\{\tau\in C^\alpha(K,TM)\colon \pi_{TM}\circ \tau=f\}$ be the locally convex space constructed in \ref{bundleconv}. Then
\begin{align}
O_f&\coloneq \Gamma_f\cap C^\alpha(K,U) \quad \text{is an open subset of~$\Gamma_f$,} \notag\\
O_f'&\coloneq\{g\in C^\alpha(K,M)\colon (f,g)(K)\subseteq U'\} \quad \text{is an open subset of $C^\alpha(K,M)$, and} \notag\\
\phi_f&\colon O_f\to O'_f,\quad\tau\mto \Sigma\circ \tau \label{thephif}
\end{align}
is a homeomorphism with inverse $g\mto \theta^{-1}\circ (f,g)$.
By the preceding, if also $h\in C^\alpha(K,M)$, then $\psi\coloneq\phi_h^{-1}\circ \phi_f$
has an open (possibly empty) domain $D\subseteq \Gamma_f$ and is a smooth map $D\to\Gamma_h$
by Lemma~\ref{expsect}, as $\psi^\wedge\colon D\times K\to TM$,
\[
(\tau,x)\mto (\phi_h^{-1}\circ \phi_f)(\tau)(x)=\theta^{-1}(h(x),\Sigma(\tau(x)))
=\theta^{-1}(h(x),\Sigma(\ve(\tau,x)))
\]
is a $C^{\infty,\alpha}$-map (exploiting that the evaluation map $\ve\colon \Gamma_f\times K\to TM$ is $C^{\infty,\alpha}$, by Lemma~\ref{evalsect}).
Hence $C^\alpha(K,M)$ endowed with the $C^\alpha$-topology has a smooth manifold structure for which each of the maps
$\phi_f^{-1}$ is a local chart.
\end{numba}
We now prove that the manifold structure on $C^\alpha (K,M)$ is canonical.
Together with Lemma~\ref{base-cano}\,(b), this implies
that the smooth manifold structure on $C^\alpha(K,M)$ constructed in \ref{numba:mfdstruct}
is independent of the choice of local addition.
\begin{la}\label{la:cano:mfdmap}
The manifold structure on $C^\alpha (K,M)$ constructed in {\rm \ref{numba:mfdstruct}} is canonical.
\end{la}
\begin{proof}
We first show that the evaluation map $\ev\colon C^\alpha_f(K,M)\times K\to M$ is $C^{\infty,\alpha}$. 
It suffices to show that $\ev(\phi_f(\tau),x)$ is $C^{\infty,\alpha}$ in $(\tau,x)\in O_f\times K$ for all $f\in C^\alpha(K,M)$. This follows from
\[
\ev(\phi_f(\tau),x)=\Sigma(\tau(x))=\Sigma(\ve(\tau,x)),
\]
where $\ve\colon\Gamma_f\times K\to TM$, $(\tau,x)\mto\tau(x)$
is $C^{\infty,\alpha}$ by Lemma~\ref{evalsect}.
Now let $\beta \in(\N_0\cup\{\infty\})^m$ and $h\colon N\to C^\alpha(K,M)$ be a map, where
$N=N_1\times\cdots\times N_n$ is a
product of smooth manifolds with rough boundary.
If $h$ is~$C^\beta$, then $h^\wedge=\ev\circ (h\times\id_K)$
is~$C^{\beta,\alpha}$.
Conversely, let $h^\wedge$ be a $C^{\beta,\alpha}$-map,
then $h$ is continuous as a map to $C(K,M)$ with the compact-open topology
(see \cite[Proposition~A.6.17]{GaN})
and $h(x)=h^\wedge(x,\cdot)\in C^\alpha(K,M)$ for each $x\in N$.
Given $x\in N$, let $f:=h(x)$. Then $\psi_f\colon C(K,M)\to C(K,M)\times C(K,M)\cong C(K,M\times M)$,
$g\mto (f,g)$ is a continuous map.
Since $\psi_f(g)$ is $C^\alpha$ if and only if~$g$ is $C^\alpha$, we see that
\begin{eqnarray*}
W &:=& h^{-1}(O_f')=h^{-1}(\psi_f^{-1}(C^\alpha(K,U')))=(\psi_f\circ h)^{-1}(C^\alpha(K,U'))\\
&=&(\psi_f\circ h)^{-1}(C(K,U'))
\end{eqnarray*}
is an open $x$-neighborhood in~$N$. As the map $(\phi_f^{-1}\circ h|_W)^\wedge\colon
W\times K\to TM$,
\[
(y,z)\mto ((\phi_f)^{-1}\circ h|_W)^\wedge(y,z)=(\theta^{-1}\circ (f,h(y)))(z)=\theta^{-1}(f(z),h^\wedge(y,z))
\]
is $C^{\beta,\alpha}$ by \cite[Lemma 3.16]{Alz}, the map $\phi_f^{-1}\circ h|_W\colon W\to\Gamma_f$ (and hence also $h|_W$)
is~$C^k$, by Lemma~\ref{expsect}.
\end{proof}

\begin{prop}\label{prop:cano:covered}
Let $K=K_1\times\cdots\times K_n$ be a
product of compact smooth manifolds with rough boundary
and $M$ be a manifold covered by local additions. For every $\alpha \in (\N_0 \cup\{\infty\})^n$, the set $C^\alpha (K,M)$ can be endowed with a canonical manifold structure.
\end{prop}

\begin{proof}
 Let $(M_j,\Sigma_j)_{j\in J}$ be an upward directed family of open submanifolds $M_j$ with local additions $\Sigma_i$ whose union coincides with $M$. 
 As $K$ is compact, we observe that the sets $C^\alpha (K,M_j) \coloneq \{f \in C^{\alpha} (K,M) \mid f(K) \subseteq M_j\}$ are open in the $C^\alpha$-topology.
 Following Lemma \ref{la:cano:mfdmap}, we can endow every $C^\alpha (K,M_j)$ with a canonical manifold structure. Now if $M_j \subseteq M_\ell$, Lemma \ref{base-cano}\,(c) implies that also the submanifold structure induced by the inclusion $C^{\alpha}(K,M_j) \subseteq C^\alpha (K,M_\ell)$ is canonical. Thus uniqueness of canonical structures, Lemma \ref{base-cano}\,(b), shows that the submanifold structure must coincide with the canonical structure constructed on $C^\alpha (K,M_j)$ via \ref{numba:mfdstruct}.  
 As $C^\alpha (K,M) = \bigcup_{j \in J} C^\alpha (K,M_j)$ and each
 step of the ascending union is canonical,
 the same holds for the union.
\end{proof}

\subsection*{The tangent bundle of the manifold of mappings}

In the rest of this section, we identify the tangent bundle of $C^\alpha (K,M)$ as the manifold $C^\alpha (K,TM)$ (under the assumption that $K$ is compact and $M$ covered by local additions). To explain the idea, let us have a look at $C^\alpha (K,TM)$.

\begin{numba}\label{numba:TVBun}
 Consider a smooth manifold $M$ covered by local additions. Then also $TM$ is covered by local additions, cf.\ \cite[A.11]{AGS} for the construction. Thus for $K$ a compact manifold $C^\alpha (K,M)$ and $C^\alpha (K,TM)$ are canonical manifolds.
 If we denote by $\pi \colon TM \rightarrow M$ the bundle projection, Corollary \ref{cor:pushforward} shows that the pushforward $\pi_\ast \colon C^\alpha(K,TM) \rightarrow C^\alpha (K,M)$ is smooth. The fibres of $\pi_\ast$ are the locally convex spaces $\pi_\ast^{-1} (f)=\Gamma_f$ from \ref{bundleconv}. We deduce that $\pi_\ast \colon C^\alpha (K,TM) \rightarrow C^\alpha (K,M)$ is a vector bundle (see Theorem \ref{thm:tangentident} for a detailed proof).  
\end{numba}

We will first identify the fibres of the tangent bundle.

\begin{numba}\label{numba:curves:tan}
The tangent space $T_f C^\alpha (K,M)$ is given by equivalence classes $[t\mapsto c(t)]$ of $C^1$-curves $c \colon ]{-\varepsilon}, \varepsilon[\, \rightarrow C^\alpha (K,M)$ with $c (0)= f$, where the equivalence relation $c \sim c'$ holds for two such curves if and only if $\dot{c}(0) = \dot{c}'(0)$.
Since the manifold structure is canonical (Lemma \ref{prop:cano:covered}) we see that $c$ is $C^1$ if and only if the adjoint map $c^\wedge \colon ]{-\varepsilon},\varepsilon[\, \times K \rightarrow N$ is a $C^{1,\alpha}$-map.
The exponential law shows that the derivative of $c$ corresponds to the (partial) derivative of $c^\wedge$, i.e.\ the mapping $\Psi$ from \eqref{tangent:iso} restricts to a bijection
\begin{align}\label{eq:PHI}
\Psi_f \colon T_\gamma C^\ell (K,M) &\rightarrow \Gamma_f = \{h \in C^\ell (K,TM)\mid \pi\circ h =f\},\\  [c] &\mapsto (k \mapsto [t \mapsto c^\wedge (t,k)]).\notag
\end{align}
\end{numba} 

We wish to glue the bijections on the fibres to identify the tangent manifold as the bundle from \ref{numba:TVBun}. To this end, we recall a fact from \cite[Lemma A.14]{AGS}:

\begin{numba}\label{numba:norm:locadd}
 If a manifold $M$ admits a local addition, it also admits a normalized local addition.
\end{numba}
Hence we may assume without loss of generality that the local additions in the following are normalized. Moreover, we will write $\varepsilon_x\colon C^\alpha(K,M) \rightarrow M$ for the point evaluation in $x \in K$. Then the tangent bundle of $C^\alpha(K,M)$ can be described as follows. 

\begin{thm}\label{thm:tangentident}
Let $K=K_1\times\cdots\times K_n$ be a product
of compact smooth manifolds with rough boundary
and~$M$ be covered by local additions. Then 
\[
(\pi_{TM})_*\colon C^\ell(K,TM)\to C^\ell(K,M)
\]
is a smooth vector bundle with fibre $\Gamma_f$ over $f\in C^\ell(K,M)$.
For each $v\in T(C^\ell(K,M))$, we have $\Psi(v):=(T\ve_x(v))_{x\in K}\in C^\alpha(K,TM)$
and the map \eqref{tangent:iso},
\[
\Psi\colon TC^\alpha(K,M)\to C^\alpha(K,TM),\quad v\mto \Psi(v)
\]
is an isomorphism of smooth vector bundles $($over the identity$)$.
\end{thm}

If we wish to emphasize the dependence on $M$, we write $\Psi_M$
instead of~$\Psi$.

\begin{proof}
Since $M$ is covered by local additions, there is a family of open submanifolds (ordered by inclusion) $(M_j)_{j\in J}$ which admit local additions $\Sigma_j$. 
Now by compactness of $K$ the image of $f \in C^\alpha (K,M)$ is always contained in some $M_j$ and similarly for $\tau \in \Gamma_f$ we then have $\tau(K) \subseteq \pi^{-1}(M_j) =TM_j$, where $\pi \coloneq \pi_{TM}$ is the bundle projection of $TM$. As the family $(M_j)_j$ of open manifolds exhausts $M$, we have $C^\alpha(K,M) = \bigcup_{j\in J} C^\alpha (K,M_j)$ and all of these subsets are open. Hence it suffices to prove that $\Psi$ restricts to a bundle isomorphism for every $M_j$. In other words we may assume without loss of generality that $M$ admits a local addition $\Sigma$. 
Given $f\in C^\alpha(K,M)$, the map $\phi_f\colon O_f\to O_f'\subseteq C^\alpha(K,M)$
is a $C^\infty$-diffeomorphism with $\phi_f(0)=f$, whence
$
T\phi_f(0,\cdot)\colon\Gamma_f\to T_f(C^\alpha(K,M))
$
is an isomorphism of topological vector spaces. For $\tau\in\Gamma_f$,
we have for $x\in K$
\begin{eqnarray*}
T\ve_xT\phi_f(0,\tau)&=&T\ve_x([t\mto \Sigma\circ (t\tau)])
=[t\mto \Sigma(t\tau(x))]\\
&=& [t\mto\Sigma|_{T_{f(x)}M}(t\tau(x))]
=T\Sigma|_{T_{f(x)}M}(\tau(x))=\tau(x),
\end{eqnarray*}
as $\Sigma$ is assumed normalized. Thus $\Psi(T\phi_f(0,\tau))=\tau\in\Gamma_f\subseteq
C^\alpha(K,TM)$, whence $\Psi(v)\in \Gamma_f\subseteq C^\alpha(K,TM)$ for each $v\in T_f(C^\alpha(K,M))$ and $\Psi$ takes $T_f(C^\alpha(K,M))$ bijectively
and linearly onto~$\Gamma_f$. Now the manifolds $T(C^\alpha(K,M))$ and $C^\alpha(K,TM)$ are the disjoint union of the sets $T_f(C^\alpha(K,M))$ and $\Gamma_f=\pi_\ast^{-1}(\{f\})$,
respectively, we see that $\Psi$ is a bijection. If we can show that $\Psi$ is a $C^\infty$-diffeomorphism, $\pi_\ast\colon
C^\alpha(K,TM)\to C^\alpha(K,M)$ will be a smooth vector bundle over $C^\alpha(K,M)$ (like $T(C^\alpha(K,M))$). Finally, $\Psi$ will then be an isomorphism of smooth vector bundles over~$\id_M$.\\[2.3mm]
For the proof we recall some results from the Appendix of \cite{AGS}: Denote by $0 \colon M \rightarrow TM$ the zero-section and by $0_M \coloneq 0(M)$ its image. Let now $\lambda_p \colon T_p M \rightarrow TM$ be the canonical inclusion and $\kappa \colon T^2M \rightarrow T^2M$ the canonical flip (given in charts by $(x,y,u,v) \mapsto (x,u,y,v)$) then \cite[Lemma A.20\,(b)]{AGS} yields a natural isomorphism $\Theta \colon TM \oplus TM \rightarrow \pi_{T^2M}^{-1} (0_M) \subseteq T^2M, \Theta(v,w) = \kappa (T\lambda_{\pi(v)}(v,w))$. On the level of function spaces\footnote{While the results in \cite{AGS} were only established for the case of $C^{k,\ell}$-mappings, they carry over (together with their proofs) without any change to the more general case of the $C^\alpha$-mappings considered here.}
$\Theta$ induces a diffeomorphism (cf.\ \cite[Lemma A.20\,(e)]{AGS})  
$$\Theta_f\colon O_f\to O_{0\circ f},\quad \gamma \mapsto \Theta \circ (0\circ f,\gamma).$$
Here for $f \in C^\alpha (K,M)$ we have considered the composition $0\circ f \in C^\alpha (K,TM)$. Then the sets $S_f:=T\phi_f(O_f\times \Gamma_f)$ form an open cover of $T(C^\alpha(K,M))$ for $f\in C^\alpha(K,M)$, whence the sets $\Psi(S_f)$ form a cover of $C^\alpha(K,TM)$ by sets which are open as $\Psi(S_f)
=(\phi_{0\circ f}\circ \phi_f)(O_f\times \Gamma_f)=\phi_{0\circ f}(O_{0\circ f})$. Hence it suffices to prove that the bijective map $\Psi$ restricts to a $C^\infty$-diffeomorphism on these open sets. In other words it suffices to show that
\[
\Phi\circ T\phi_f=\phi_{0\circ f}\circ \Theta_f
\]
for each $f\in C^\ell(K,M)$ (as all other mappings in the formula are smooth diffeomorphisms). Now
\[
T\phi_f(\sigma,\tau)=[t\mto\Sigma\circ (\sigma+t\tau)]
\]
for all $(\sigma,\tau)\in O_f\times\Gamma_f$, and thus we can rewrite $\Psi(T\phi_f(\sigma,\tau))$ as
\begin{align*}
& ([t\mto \Sigma(\sigma(x)+t\tau(x))])_{x\in K}
= ([t\mto (\Sigma\circ \lambda_{f(x)})(\sigma(x)+t\tau(x))])_{x\in K}\\
=& (T(\Sigma\circ \lambda_{f(x)})(\sigma(x),\tau(x)))_{x\in K}
= (\Sigma_{TM}((\kappa \circ T\lambda_{f(x)})(\sigma(x),\tau(x))))_{x\in K}\\
=& ((\Sigma_{TM}\circ \Theta_f)(\sigma,\tau)(x))_{x\in K}
=(\phi_{0\circ f}\circ \Theta_f)(\sigma,\tau). 
\end{align*}
Thus the desired formula holds and shows that $\Psi$ is a $C^\infty$-diffeomorphism. This concludes the proof.
\end{proof}

\begin{rem}\label{alsothis}
Assume that the local additions $\Sigma\colon U_i\to M_i$ covering $M$ are normalized. Then the proof of Theorem~\ref{thm:tangentident}
shows that
\[
\Psi\circ T\phi_f(0,\cdot)\colon\Gamma_f\to C^\alpha(K,TM)
\]
is the inclusion map $\tau\mto \tau$,
for each $f\in C^\alpha(K,M)$ (where $\phi_f$ is as in (\ref{thephif})).\smallskip
\end{rem}

Using canonical manifold structures, we have:
\begin{cor}\label{formtang}
Let $K=K_1\times\cdots\times K_n$ be a product
of compact smooth manifolds with rough boundary,
$\alpha\in (\N_0\cup\{\infty\})^n$
and $g\colon M\to N$ be a $C^{|\alpha|+1}$-map between smooth manifolds~$M$
and $N$ covered by local additions.
Then the tangent map of the $C^1$-map
\[
g_*\colon C^\alpha(K,M)\to C^\alpha(K,N),\quad f\mto g\circ f
\]
is given by $T(g_*)=\Psi_N^{-1}\circ (Tg)_*\circ \Psi_M$.
For each $f\in C^\alpha(K,M)$, we have $\Psi_M(T_f(C^\alpha(K,M)))=\Gamma_f(TM)$,
$\Psi_N(T_{g\circ f}(C^\alpha(K,N)))=\Gamma_{g\circ f}(TN)$
and $(Tg)_*$ restricts to the map
\begin{equation}\label{nearly}
\Gamma_f(TM)\to\Gamma_{g\circ f}(TN),\quad \tau\mto Tg\circ \tau
\end{equation}
which is continuous linear and corresponds to $T_f(g_*)$.
\end{cor}
Moreover, the identification of the tangent bundle allows us to lift local additions (cf.\ \cite[Remark A.17]{AGS}).

\begin{la}\label{la:lift:locadd}
Let $K=K_1\times \cdots\times K_n$ be a product
of compact smooth manifolds with rough boundary,
$\alpha \in (\N_0 \cup \{\infty\})^n$ and $M$ a manifold covered by local additions. 
Then the canonical manifold $C^\alpha (K,M)$ is covered by local additions.
\end{la}

\begin{proof}
 Consider first the case that $M$ admits a local addition $\Sigma \colon U \rightarrow M$ with $\theta= (\pi_{TM} , \Sigma) \colon U \rightarrow U' \subseteq M \times M$ the associated diffeomorphism. 
 Since also $TM$ admits a local addition, we have canonical manifold structures on $C^\alpha (K,TM)$ and $C^\alpha (K,M \times M ) \cong C^\alpha (K,M) \times C^\alpha (K,M)$. Now $K$ is compact, whence $C^\alpha (K,U) \subseteq C^\alpha (K,TM)$ is an open submanifold, whence canonical by Lemma \ref{base-cano}\,(c). In particular, $\Sigma_* \colon C^\alpha(K,U) \rightarrow C^\alpha(K,M)$ and $\theta_* \colon C^\alpha (K,U) \rightarrow C^\alpha (K,U') \subseteq C^\alpha (K,M\times M)$ are smooth by Corollary \ref{cor:pushforward}. As also the inverse of $\theta$ is smooth, we deduce that $\theta_*$ is again a diffeomorphism mapping $C^\alpha (K,U)$ to $C^\alpha (K,U')$ and we can identify the latter manifold with an open subset of $C^\alpha (K,M) \times C^\alpha (K,M)$ containing the diagonal. 
 Hence we only need to verify that $0_f \in T_f C^\alpha (K,TM)$ is mapped to $f$. However, using the point evaluation $\ve_x (\Sigma_* (0_f))= \Sigma (0(f(x))) = f(x)$ (where $0$ is again the zero-section of $TM$), we obtain the desired equality pointwise and thus also on the level of functions. This proves that $C^\alpha (K,M)$ admits a local addition if $M$ admits a local addition.
 
 If now $M$ is covered by open submanifolds $(M_j)_{j\in J}$ each admitting a local addition, it suffices to see that $C^\alpha (K,M_j)$ is an  open submanifold of $C^\alpha (K,M)$ which admits a local addition by the above considerations. Thus $C^\alpha (K,M)$ is covered by the open submanifolds $(C^\alpha (K,M_j))_{j \in J}$ and as each of those admits a local addition, $C^\alpha (K,M)$ is covered by local additions.  
\end{proof}

\begin{prop}\label{prop:currying}
Let $K=K_1\times\cdots\times K_m$ and $L=L_1\times\cdots\times L_n$
be products of compact manifolds with rough boundary and $M$ be a manifold covered by local additions. Fix $\alpha \in (\N_0 \cup \{\infty\})^n,\beta\in (\N_0 \cup \{\infty\})^m$.
Then $C^{\beta,\alpha} (L \times K , M)$, $C^\alpha(K,M)$
and $C^\beta (L,C^\alpha (K,M))$ admit canonical manifold
structures.
Using these,
the bijection $C^{\beta,\alpha} (L \times K , M) \rightarrow C^\beta (L,C^\alpha (K,M))$ is a $C^\infty$-diffeomorphism.
\end{prop}

\begin{proof}
We apply
Proposition~\ref{prop:cano:covered}
to obtain canonical manifold structures on $C^\alpha (K,M)$
and $C^{\beta,\alpha} (L\times K,M)$.
By Lemma \ref{la:lift:locadd}, $C^\alpha (K,M)$ is covered by local additions. Hence we may apply Proposition \ref{prop:cano:covered} again to obtain
a canonical manifold structure on
$C^\beta(L,C^\alpha(K,M))$.
By Proposition \ref{explaw-precan},
the bijection $C^{\beta,\alpha} (L \times K , M) \rightarrow C^\beta (L,C^\alpha (K,M))$
is a diffeomorphism.
\end{proof}
\section{Lie groups of Lie group-valued mappings}\label{sec-Lie}
We now prove Theorem~\ref{thmB},
starting with observations.
\begin{la}\label{la-is-gp}
Let $M_1,\ldots, M_n$
be locally compact smooth manifolds with rough boundary,
$G$ be a Lie group, and $\alpha\in (\N_0\cup\{\infty\})^n$.
Setting $M:=M_1\times \cdots\times M_n$,
the following holds:
\begin{itemize}
\item[\rm(a)]
$C^\alpha(M,G)$
is a group.
\item[\rm(b)]
If a pre-canonical smooth manifold structure
exists on $C^\alpha(M,G)$,
then it makes $C^\alpha(M,G)$ a Lie group.
Moreover, it turns the point evaluation
$\ve_x\colon C^\alpha(M,G)\to G$,
$f\mto f(x)$
into a smooth group homomorphism
for each $x\in M$.
\end{itemize}
\end{la}
\begin{proof}
(a)
The group inversion $\iota\colon G\to G$
is smooth, whence $\iota\circ f$
is $C^\alpha$ for all $f\in C^\alpha(M,G)$
(by the Chain Rule \cite[Lemma 3.16]{Alz},
applied in local charts).
Let $\mu\colon G\times G\to G$
be the smooth group multiplication
and $f,g\in C^\alpha(M,G)$.
Then $(f,g)\colon M\to G\times G$
is $C^\alpha$ by \cite[Lemma~3.8]{Alz}.
By the Chain Rule,
$fg=\mu\circ (f,g)$
is $C^\alpha$.\\[.7mm]
(b) The group inversion in $C^\alpha(M,G)$
is the map $C^\alpha(M,\iota)$
and hence smooth, by Corollary~\ref{cor:pushforward}.
Identifying $C^\alpha(M,G)\times C^\alpha(M,G)$
with $C^\alpha(M,G\times G)$ as a smooth manifold
(as in Lemma~\ref{la:ISO1}\,(a)), the group multiplication
of $C^\alpha(M,G)$ is the map $C^\alpha(M,\mu)$
and hence smooth.
The group multiplication in $C^\alpha(M,G)$
being pointwise, $\ve_x$ is a homomorphism
of groups for each $x\in M$.
By Lemma~\ref{base-cano}\,(a),
$\ev\colon C^\alpha(M,G)\times M\to G$
is $C^{\infty,\alpha}$.
Thus $\ve_x=\ev(\cdot,x)$ is smooth.
\end{proof}
Another concept is useful, with notation as in~\ref{lie-functor}.
\begin{defn}\label{defn-expected}
Let $M_1,\ldots, M_n$
be locally compact smooth manifolds with rough boundary,
$G$ be a Lie group, and $\alpha\in (\N_0\cup\{\infty\})^n$.
For $x\in M:=M_1\times\cdots\times M_n$,
let $\ve_x\colon C^\alpha(M,G)\to G$
be the point evaluation.
A smooth manifold
structure on $C^\alpha(M_1\times\cdots\times M_n,G)$
making it a Lie group is said to be \emph{compatible with evaluations}
if $\ve_x$ is smooth for each $x\in M$,
we have $\phi(v):=(L(\ve_x)(v))_{x\in M}\in C^\alpha(M,L(G))$
for each $v\in L(C^\alpha(M,G))$, and the Lie algebra homomorphism
\[
\phi\colon L(C^\alpha(M,G))\to C^\alpha(M,L(G)),\;\; v\mto \phi(v)
\]
so obtained is an isomorphism of topological vector spaces.
\end{defn}
\begin{rem}
In the case that $n=1$ and $\alpha=\infty$,
compatibility with evaluations was introduced
in \cite[Proposition~1.9 and page 19]{NaW}
(in different words), assuming that $G$ is regular.
Likewise, $G$ is assumed regular
in \cite[Proposition~3.1]{Ham},
where the case $n=1$, $\alpha\in \N_0\cup\{\infty\}$
is considered.
\end{rem}
\begin{la}\label{wedge-is-lie}
Let $M_1,\ldots, M_n$ and $N_1,\ldots, N_m$
be locally compact smooth manifolds with rough boundary,
$\alpha\in (\N_0\cup\{\infty\})^m$,
$\beta\in (\N_0\cup\{\infty\})^n$,
$M:=M_1\times\cdots\times M_n$,
$N:=N_1\times\cdots\times N_m$,
and $G$ be a Lie group.
Assume that $C^\beta(M,G)$ is endowed with a pre-canonical
smooth manifold structure
which is compatible with evaluations
and that $C^\alpha(N,C^\beta(M,G))$,
whose definition uses the latter structure,
is endowed with a pre-canonical
smooth manifold structure which is compatible with evaluations.
Endow $C^{\alpha,\beta}(N\times M,G)$
with the smooth manifold structure turning
the bijection
\[
\Phi\colon C^{\alpha,\beta}(N\times M,G)\to C^\alpha(N,C^\beta(M,G)),\;\,
f\mto f^\vee
\]
into a $C^\infty$-diffeomorphism.
Then
the preceding smooth manifold structure on $C^{\alpha,\beta}(N\times M,G)$
is pre-canonical and compatible with evaluations.
\end{la}
\begin{proof}
By Lemma~\ref{la:ISO1}\,(c),
the $C^\infty$-manifold structure on $C^{\alpha,\beta}(N\times M,G)$
is pre-canonical, whence the latter is a Lie group.
The $C^\infty$-diffeomorphism
$\Phi$ is a homomorphism of groups.
Hence
\[
L(\Phi)\colon L(C^{\alpha,\beta}(N\times M,G))\to L(C^\alpha(N,C^\beta(M,G)))
\]
is an isomorphism of topological Lie algebras.
Consider the point evaluations $\ve_x\colon C^\alpha(N,C^\beta(M,G))\to C^\beta(M,G)$,
$\ve_{(x,y)}\colon C^{\alpha,\beta}(N\times M,G)\to G$
and
$\ve_y\colon C^\beta(M,G)\to G$
for $x\in N$, $y\in M$.
By hypothesis, we have isomorphisms of topological Lie algebras
\[
\Psi\colon L(C^\beta(M,G))\to C^\beta(M,L(G)),\;\, w\mto (L(\ve_y)(w))_{y\in M}
\]
and $\Theta\colon L(C^\alpha(N,C^\beta(M,G)))\to C^\alpha(N,L(C^\beta(M,G)))$,
$v\mto (L(\ve_x)(v))_{x\in N}$. Then also
\[
\Psi_*\colon C^\alpha(N,L(C^\beta(M,G)))\to C^\alpha(N,C^\beta(M,L(G))),\;\,
f\mto \Psi\circ f
\]
is an isomorphism of topological Lie algebras and so is
\[
\Xi\colon C^\alpha(N,C^\beta(M,L(G))\to C^{\alpha,\beta}(N\times M,L(G)),\;\,
f\mto f^\wedge,
\]
by the Exponential Law (Lemma~\ref{exp-law-not-pure}).
Hence
\[
\phi:=\Xi\circ \Psi_*\circ \Theta\circ L(\Phi)\colon L(C^{\alpha,\beta}(N\times M,G))\to
C^{\alpha,\beta}(M\times N, L(G))
\]
is an isomorphism of topological Lie algebras.
Regard $v\in L(C^{\alpha,\beta}(N\times M,G))$
as a geometric tangent vector $[\gamma]$
for a smooth curve $\gamma\colon \;]{-\ve},\ve[\to
C^{\alpha,\beta}(N\times M,G)$ with $\gamma(0)=e$.
Then $L(\Phi)(v)=[\Phi\circ\gamma]$
and $\Theta(L(\Phi)(v))=([\ve_x\circ \Phi\circ \gamma])_{x\in N}=:g$.
Thus
\begin{eqnarray*}
\phi(v)(x,y)\!\!\! &\!=\! &\!\!  \Psi_*(g)(x)(y)
=(\Psi\circ g)(x)(y)=\Psi([\ve_x\circ \Phi\circ\gamma])(y)\\
\!\! &\! =\! &\!\! L(\ve_y)([\ve_x\circ \Phi\circ\gamma])
=[\ve_x\circ\ve_y\circ\Phi\circ\gamma]
=[t\mto \ve_x(\ve_y(\Phi(\gamma(t))))]\\
\!\! & \! =\! & \!\! [t\mto \gamma(t)(x,y)]
= L(\ve_{(x,y)})([\gamma])=L(\ve_{(x,y)})(v).
\end{eqnarray*}
We deduce that $(L(\ve_{(x,y)})(v))_{(x,y)\in N\times M}=\phi(v)
\in C^{\alpha,\beta}(N\times M, L(G))$.
Since $\phi$ is an isomorphism of topological Lie algebras,
the Lie group structure on $C^{\alpha,\beta}(N\times M,G)$
is compatible with evaluations.
\end{proof}
\begin{la}\label{expected-lie-regular}
Let $M_1,\ldots, M_n$
be locally compact smooth manifolds with rough boundary,
$M:=M_1\times\cdots\times M_n$,
$\alpha\in (\N_0\cup\{\infty\})^n$,
and $G$ be a Lie group.
Assume that $C^\alpha(M,G)$ is endowed with a pre-canonical
smooth manifold structure
which is compatible with evaluations.
If the Lie group $G$ is $C^r$-regular
for some $r\in \N_0\cup\{\infty\}$,
then also the Lie group $C^\alpha(M,G)$ is $C^r$-regular.
\end{la}
\begin{proof}
Consider the smooth evolution map
$\Evol\colon C^r([0,1],\cg)\to C^{r+1}([0,1],G)$,
where $\cg:=L(G)$.
For $x\in M$, let $\ve_x\colon C^\alpha(M,G)\to G$,
$f\mto f(x)$ be evaluation at~$x$.
By hypothesis,
$\phi\colon L(C^\alpha(M,G))
\to C^\alpha(M,\cg)$,
$v\mto (L(\ve_x)(v))_{x\in M}$
is an isomorphism of topological Lie algebras.
Then also
\[
\phi_* \colon C^r([0,1],L(C^\alpha(M,G)))\to
C^r([0,1],C^\alpha(M,\cg)),\;
f\mto\phi\circ f
\]
is an isomorphism of topological Lie algebras.
By Example~\ref{basic-examples},
the smooth manifold structures on all of the locally convex spaces
$C^r([0,1], C^\alpha(M,\cg))$,
\[
C^{r,\alpha}([0,1]\times M,\cg),\;\,
C^{\alpha,r}(M\times [0,1],\cg),\;\,\mbox{and}\;\,
C^\alpha(M,C^r([0,1],\cg))
\]
are canonical. By Lemma~\ref{exp-law-not-pure},
the Lie algebra homomorphism
\[
\psi\colon C^r([0,1],C^\alpha(M,\cg))\to C^{r,\alpha}
([0,1]\times M,\cg),\;\, f\mto f^\wedge
\]
is an isomorphism of topological
Lie algebras. Flipping the factors $[0,1]$ and $M$
(with Lemma~\ref{la:ISO2}\,(b))
and using the Exponential Law
again, we obtain
an isomorphism of topological Lie algebras
\[
\theta\colon C^{r,\alpha}([0,1]\times M,\cg)
\to C^{\alpha}(M,C^r([0,1],\cg))
\]
determined by $\theta(f)(x)(t)=f(t,x)$.
By Theorem~\ref{thmA}, $C^{r+1}([0,1],C^\alpha(M,G))$
has a canonical smooth manifold structure.
Using Lemma~\ref{la:ISO1}\,(c), Lemma~\ref{la:ISO2}\,(a), and
Lemma~\ref{la:ISO1}\,(c) in turn,
we can give $C^\alpha(M,C^{r+1}([0,1],G))$
a pre-canonical smooth manifold structure making the map
\[
\beta\colon C^\alpha(M,C^{r+1}([0,1],G))\to
C^{r+1}([0,1],C^\alpha(M,G))
\]
determined by $\beta(f)(t)(x)=f(x)(t)$
a $C^\infty$-diffeomorphism.
The structures being pre-canonical,
\[
\Evol_* \colon
C^\alpha(M,C^r([0,1],\cg))\to
C^\alpha(M,C^{r+1}([0,1],G)),\;\,
f\mto \Evol\circ f
\]
is smooth. Hence also
$\cE:=\beta\circ
\Evol_*  \circ \, \theta \circ \psi \circ \, \phi_*$
is smooth as a map
\[
C^r([0,1], L(C^\alpha(M,G)))\to
C^{r+1}([0,1],C^\alpha(M,G)).
\]
It remains to show that $\cE$
is the evolution map of $C^\alpha(M,G)$.
As the $L(\ve_x)$ separate points
on ${\mathfrak h}:=L(C^\alpha(M,G))$ for $x\in M$,
it suffices to show that\linebreak
$\ve_x\circ \cE(\gamma)=\Evol(L(\ve_x)\circ \gamma)$
for all $\gamma\in C^r([0,1],{\mathfrak h})$ and $x\in M$
(see \cite[Lemma~10.1]{Sem}).
Note that $(\phi\circ\gamma)(t)(x)=L(\ve_x)(\gamma(t))$,
whence
\[
((\psi\circ\theta)(\phi\circ\gamma))(x)(t)=L(\ve_x)(\gamma(t))
\]
and $\big(\!\Evol_*((\psi\circ\theta)(\phi\circ\gamma))\big)(x)=
\Evol(((\psi\circ\theta)(\phi\circ\gamma))(x))
=\Evol(L(\ve_x)\circ\gamma)$. So
$(\ve_x\circ \cE(\gamma))(t)=
(\Evol_*  \circ \, \theta \circ \psi \circ \phi_*)(\gamma)(x)(t)
=\Evol (L(\ve_x)\circ \gamma)(t)$.\vspace{1.9mm}
\end{proof}
We establish Theorem~\ref{thmB}
in parallel with the first conclusion
of the following proposition,
starting with two basic cases:\\[2.3mm]
Case 1:
The manifolds $M_1,\ldots, M_n$ are compact;\\[2.3mm]
Case 2: $M$ is $1$-dimensional with finitely many connected components.\vspace{.9mm}
\begin{prop}\label{amendment}
In Theorem~{\rm\ref{thmB}},
the Lie group structure on $C^\alpha(M,G)$
is compatible with evaluations,
writing $M:=M_1\times\cdots\times M_n$.
Moreover, there is a unique
canonical pure smooth manifold
structure on $C^\alpha(M,G)$
which is\linebreak
modeled on $C^\alpha(M,L(G))$.
\end{prop}
The final assertion is clear:
Starting with any canonical structure on $C^\alpha(M,G)$
and a chart $\phi\colon U_\phi\to V_\phi\to E_\phi$
around the constant map~$e$,
using left translations (which are $C^\infty$-diffeomorphisms)
we can create charts around every $f\in C^\alpha(M,G)$
which are modeled on the given $E_\phi$.
We can therefore select a subatlas making
$C^\alpha(M,G)$ a pure smooth manifold.
Since $E_\phi$ is isomorphic to $L(C^\alpha(M,G))$,
which is isomorphic to $E:=C^\alpha(M,L(G))$
as a locally convex space (by compatibility with evaluations),
we can replace $E_\phi$ with~$E$.
The pure canonical
structure modeled on~$E$ is unique, since $\id_{C^\alpha(M,G)}$
is a $C^\infty$-diffeomorphism for any two canonical structures
(cf.\ Lemma~\ref{base-cano}\,(b)).
\begin{la}\label{basic-case-2}
Let $M_1,\ldots, M_n$
be compact smooth manifolds with rough boundary,
$G$ be a Lie group and $\alpha\in (\N_0\cup\{\infty\})^n$.
Abbreviate $M:=M_1\times \cdots\times M_n$.
Then $C^\alpha(M,G)$
admits a canonical smooth manifold
structure which is compatible with evaluations.
If $G$ is $C^r$-regular for $r\in \N_0\cup\{\infty\}$,
then so is $C^\alpha(M,G)$.
\end{la}
\begin{proof}
By Theorem~\ref{thmA},
$C^\alpha(M,G)$
admits a canonical smooth manifold structure.
Let $\theta\colon M\to G$ be the constant map
$x\mto e$.
By Theorem~\ref{thm:tangentident},
the diffeomorphism $(T\ve_x)_{x\in M}$
maps
$L(C^\alpha(M,G))=T_\theta(C^\alpha(M,G))$
onto
\[
\Gamma_\theta=\{\tau\in C^\alpha(M,TG)\colon
\pi_{TG}\circ \tau=\theta\}=C^\alpha(M,L(G)).
\]
By Lemma~\ref{maps-to-sub},
$C^\alpha(M,TG)$ induces on
$C^\alpha(M,L(G))$ the compact-open $C^\alpha$-topology.
Thus, the Lie group structure on
$C^\alpha(M,G)$
is compatible with evaluations.
For the last assertion,
see Lemma~\ref{expected-lie-regular}.
\end{proof}
\begin{la}\label{basic-case-3}
Let $M$ be a $1$-dimensional smooth
manifold with rough boundary,
such that $M$ has only
finitely many connected
components $($which need not be $\sigma$-compact$)$.
Let $r\in \N_0\cup\{\infty\}$,
$G$ be a $C^r$-regular Lie group,
and $k \in\N\cup\{\infty\}$
such that $k\geq r+1$.
Then $C^k(M,G)$
admits a canonical smooth manifold structure
which makes it a $C^r$-regular Lie group
and is compatible with evaluations.
\end{la}
\begin{proof}
We first assume that~$M$ is connected.
Let $\cg:=L(G)$ be the Lie algebra of~$G$.
If $N$ is a full submanifold of~$M$,
we write
$\Omega_{C^{k-1}}^1(N,\cg)\sub C^{k-1}(TN,\cg)$
for the locally convex space of $\cg$-valued $1$-forms
on~$N$, of class $C^{k-1}$.
Using the Maurer-Cartan form
\[
\kappa\colon TG\to \cg,\quad v\mto\pi_{TG}(v)^{-1}.v,
\]
a $\cg$-valued
$1$-form
\[
\delta_N(f):=\kappa\circ T f\in \Omega^1_{C^{k-1}}(N,\cg)
\]
can be associated to each $f\in C^k(N,G)$,
called its left logarithmic derivative.
Fix $x_0\in M$.
For every $\sigma$-compact, connected,
full submanifold $N\sub M$ such that $x_0\in N$,
there exists a $C^\infty$-diffeomorphism
$\psi\colon I\to N$
for some non-degenerate interval $I\sub \R$,
such that $0\in I$ and $\psi(0)=x_0$.
Then the diagram
\[
\begin{array}{rcl}
C^k(N,G) & \stackrel{\delta_N}{\longrightarrow} & \Omega^1_{C^{k-1}}(N,\cg)\\[.5mm]
\psi^* \downarrow \;\;\;\; &  & \;\;\;\; \downarrow \theta \\[.1mm]
C^k(I,G) & \stackrel{\delta^\ell}{\longrightarrow} & C^{k-1}(I,\cg),
\end{array}
\]
is commutative,
where $\psi^*\colon C^k(N,G)\to C^k(I,G)$, $f\mto f\circ\psi$
and the vertical map $\theta$ on the right-hand side,
which takes $\omega$ to $\omega\circ \dot{\psi}$,
are bijections.
For each $\omega\in \Omega^1_{C^{k-1}}(N,\cg)$,
there is a unique $f\in C^k(N,G)$
such that $f(x_0)=e$ and $\delta_N(f)=\omega$:
In fact, Lemma~\ref{reparametrize}
yields a unique
$\eta \in C^k(I,G)$ with $\eta(0)=e$
and $\delta^\ell(\eta)=\theta(\omega)$;
then $f:=(\psi^*)^{-1}(\eta)$
is as required. We set
$\Evol_N(\omega):= f$.\\[2.3mm]
If $\omega \in \Omega^1_{C^{k-1}}(M,\cg)$,
we have $\Evol_L(\omega|_{TL})=\Evol_N(\omega|_{TN})|_L$
for all $\sigma$-compact, connected open
submanifolds $N,L$ of~$M$ such that
$L\sub N$. As such submanifolds~$N$
form a cover of~$M$ which is directed
under inclusion,
we can define $f\colon M\to G$
piecewise via $f(x):=\Evol_N(\omega|_{TN})(x)$
if $x\in N$ and obtain a well-defined $C^k$-map
$f\colon M\to G$ such that $\delta_M(f)=\omega$.
Thus
\[
\delta_M(C^k(M,G))=\Omega^1_{C^{k-1}}(M,\cg),
\]
which is a submanifold of $\Omega^1_{C^{k-1}}(M,\cg)$.
Let $\cK$ be the set of all
connected, compact full submanifolds
$K\sub M$ such that $x_0\in K$.
By the preceding, $\delta_K(C^k(K,G))=\Omega^1_{C^{k-1}}(K,\cg)$,
which is a submanifold of $\Omega^1_{C^{k-1}}(K,\cg)$.
Since
\begin{equation}\label{union-of-K}
M=\bigcup_{K\in\cK}K^o,\vspace{-.7mm}
\end{equation}
\cite[Theorem~3.5]{Ham}
provides a smooth manifold structure
on $C^k(M,\cg)$ which makes it a $C^r$-regular
Lie group, is compatible with evaluations,
and turns
\[
\psi\colon C^k(M,G)\to \Omega^1_{C^{k-1}}(M,\cg)\times G,\;\,
f\mto (\delta_M(f),f(x_0))
\]
into a $C^\infty$-diffeomorphism.
It remains to show that the smooth manifold structure
is canonical.
To prove the latter, we first note that $\cK$ is directed under inclusion.
In fact, if $K_1,K_2\in\cK$,
then $K_1\cup K_2$ is contained
in a $\sigma$-compact, connected
open submanifold~$N$ of $M$
(a union of chart domains diffeomorphic to convex
subsets of~$\R$, around finitely many points in
the compact set $K_1\cup K_2$).
Pick a $C^\infty$-diffeomorphism
$\psi\colon I\to N$ as above.
Then $\psi^{-1}(K_1)$ and $\psi^{-1}(K_2)$
are compact intervals containing $0$,
whence so is their union.
Thus $K_1\cup K_2$ is a connected, compact
full submanifold of $N$ and hence of~$M$.\\[2mm]
For $K,L\in\cK$ with $K\sub L$,
let $r_{K,L}\colon \Omega^1_{C^{k-1}}(L,\cg)\to \Omega^1_{C^{k-1}}(K,\cg)$
be the restriction map.
As a consequence of Lemma~\ref{la:ascending-union}
and (\ref{union-of-K}),
\[
\Omega^1_{C^{k-1}}(M,\cg)={\pl}_{K\in\cK}\Omega^1_{C^{k-1}}(K,\cg)\vspace{-.7mm}
\]
holds as a locally convex space, using 
the restriction maps $r_K\colon \Omega^1_{C^{k-1}}(M,\cg)\to
\Omega^1_{C^{k-1}}(K,\cg)$ as the limit maps.
For $K\in\cK$, let $\rho_K\colon C^k(M,G)\to C^k(K,G)$
be the restriction map;
endow $C^k(K,G)$ with its canonical smooth manifold
structure
(as in Lemma~\ref{basic-case-2}),
which is compatible with evaluations
(the ``ordinary'' Lie group structure
in~\cite{Ham}).
Then
\[
\psi_K\colon C^k(K,G)\to \Omega^1_{C^{k-1}}(K,\cg)\times G,\;\,
f\mto (\delta_K(f),f(x_0))
\]
is a $C^\infty$-diffeomorphism
(see \cite[proof of Theorem~3.5]{Ham}).
Note that $\rho_K=\psi_K^{-1}\circ (r_K\times \id_G)\circ\psi$
is smooth on $C^k(M,G)$,
using the above Lie group structure
making $\psi$ a $C^\infty$-diffeomorphism.
Let $\alpha\in(\N_0\cup\{\infty\})^m$,
$L_1,\ldots, L_m$ be smooth manifolds with rough boundary,
$L:=L_1\times\cdots\times L_m$
and $f\colon L\to C^k(M,G)$ be a map.
If $f$ is $C^\alpha$, then also $\rho_K\circ f$
is $C^\alpha$. Since $C^k(K,G)$ is canonical,
the map
\[
f^\wedge|_{L\times K}=
(\rho_j\circ f)^\wedge\colon L\times K\to G
\]
is $C^{\alpha,k}$. Using (\ref{union-of-K}),
we deduce that $f^\wedge$ is $C^{\alpha,k}$.
If, conversely, $f^\wedge$ is $C^{\alpha,k}$,
then $(\rho_K\circ f)^\wedge=f^\wedge|_{L\times K}$
is $C^{\alpha,k}$. The smooth manifold structure on $C^k(K,G)$
being canonical, we deduce that $\rho_K\circ f$
is $C^\alpha$.
The hypotheses of Lemma~\ref{specialized-PL} being satisfied
with $A:=\cK$, $C^k(M,G)$ in place of~$M$,
$M_K:=C^k(K,G)$,
$F:=\Omega^1_{C^{k-1}}(M,\cg)$,
$F_K:=\Omega^1_{C^{k-1}}(K,\cg)$, and $N:=G$,
we see that $f$ is~$C^\alpha$.
The smooth manifold structure on $C^k(M,G)$
is therefore pre-canonical.
The topology on the projective limit
$\Omega^1_{C^{k-1}}(M,\cg)$ is initial
with respect to the limit maps
$r_K$,
whence the topology on $\Omega^1_{C^{k-1}}(M,\cg)\times G$
is initial with resspect to the maps $r_K\times\id_G$.
Since $\psi$ is a homeomorphism,
we deduce that
the topology $\cO$ on the Lie group $C^k(M,G)$
is initial with respect to the maps $(r_K\times\id_G)\circ \psi=
\psi_K\circ \rho_K$.
Since $\psi_K$ is a homeomorphism,
$\cO$ is initial just as well
with respect to the family $(\rho_K)_{K\in\cK}$.
But also the compact-open $C^k$-topology
$\cT$ on $C^k(M,G)$
is initial with respect to this family of maps
(see Lemma~\ref{la:ascending-union}),
whence $\cO=\cT$ and $C^k(M,G)$ is canonical.\\[2mm]
If $M$ has finitely many components $M_1,\ldots, M_n$,
we give $C^k(M,G)$ the smooth manifold
structure turning the bijection\vspace{-.7mm}
\[
\rho\colon C^k(M,G)\to\prod_{j=1}^n C^k(M_j,G),\quad f\mto
(f|_{M_j})_{j=1}^n\vspace{-.7mm}
\]
into a $C^\infty$-diffeomorphism.
Let $\rho_j$ be its $j$th component.
Since $\rho$ is a homeomorphism for the compact-open $C^k$-topologies
(cf.\ Lemma~\ref{la:ascending-union})
and an isomorphism of groups, the preceding smooth
manifold structure makes $C^k(M,G)$
a Lie group and is compatible with the compact-open $C^k$-topology.
As each of the Lie groups $C^k(M_j,G)$ is $C^r$-regular,
also their direct product (and thus $C^k(M,G)$) is $C^r$-regular.
Since $\rho=(\rho_j)_{j=1}^n$
is an isomorphism of Lie groups,
\[
(L(\rho_1),\ldots L(\rho_n))\colon L(C^k(M,G))\to L(C^k(M_1,G))\times\cdots\times
L(C^k(M_n,G))
\]
is an isomorphism of topological Lie algebras.
For $x\in M_j$, the point evaluation
$\ve_x\colon C^k(M,G)\to G$ is smooth,
as the point evaluation
$\bar{\ve}_x\colon C^k(M_j,G)\to G$
is smooth and $\ve_x=\bar{\ve}_x\circ \rho_j$.
We know that $\phi_j(v):=(L(\bar{\ve}_x)(v))_{x\in M_j}\in C^k(M_j,\cg)$
for all $v\in L(C^k(M_j,G))$ and that
$\phi_j\colon L(C^k(M_j, G))\to C^k(M_j,\cg)$
is an isomorphism of topological Lie algebras.
For each $v\in L(C^k(M,G))$,
we have
\[
(L(\ve_x)(v))_{x\in M_j}=(L(\bar{\ve}_x)(L(\rho_j)(v)))_{x\in M_j}
=\phi_j(L(\rho_j)(v))\in C^k(M_j,\cg)
\]
for $j\in\{1,\ldots, n\}$, whence $\phi(v):=(L(\ve_x)(v))_{x\in M}\in C^k(M,\cg)$.
Let us show that the Lie algebra homomorphism
$\phi\colon L(C^k(M,G))\to C^k(M,\cg)$
is a homeomorphism.
Lemma~\ref{la:ascending-union} entails that the map\vspace{-.7mm}
\[
r=(r_j)_{j=1}^n\colon C^k(M,\cg)\to \prod_{j=1}^n C^k(M_j,\cg),\;\,
f\mto (f|_{M_j})_{j=1}^n\vspace{-.7mm}
\]
is a homeomorphism. By the preceding,
$r\circ\phi=(\phi_1\times\cdots\times \phi_n)\circ (L(\rho_j))_{j=1}^n$
is a homeomorphism, whence so is $\phi$.
Thus, the Lie group structure on $C^k(M,G)$
is compatible with evaluations.
If $\alpha$, $L=L_1\times\cdots\times L_m$
and $f\colon L\to C^k(M,G)$
are as above and $f$ is $C^\alpha$,
then $f^\wedge$ is $C^{\alpha, k}$ by the above argument.
If, conversely, $f^\wedge$ is $C^{\alpha, k}$,
then $f^\wedge|_{L\times M_j}$
is $C^{\alpha,k}$, whence $(f^\wedge|_{L\times M_j})^\vee=\rho_j\circ f$
is $C^\alpha$ for all $j\in\{1,\ldots, n\}$.
As a consequence, $\rho\circ f$ is $C^\alpha$
and thus also~$f$.
We have shown that the smooth manifold structure
on $C^k(M,G)$ is pre-canonical and
hence canonical, as
compatibility with the compact-open $C^k$-topology
was already established.
\end{proof}
Another lemma is useful.
\begin{la}\label{ind-correct-top}
Let $N_1,\ldots, N_m$ and $M_1,\ldots, M_n$
be locally compact smooth manifolds
with rough boundary, $\alpha\in(\N_0\cup\{\infty\})^m$,
$\beta\in (\N_0\cup\{\infty\})^m$,
and $G$ be a Lie group.
Abbreviate $N:=N_1\times\cdots\times N_m$
and $M:= M_1\times \cdots\times M_n$.
Assume that $C^\beta(M,G)$
has a pre-canonical smooth manifold structure,
using which
$C^\alpha(N,C^\beta(M,G))$
has a canonical smooth manifold structure.
Endow\linebreak
$C^{\alpha,\beta}(N\times M,G)$
with the pre-canonical smooth manifold structure
turning
\[
\Phi\colon C^{\alpha,\beta}(N\times M,G)\to C^\alpha(N, C^\beta(M,G)),\;\,
f\mto f^\vee
\]
into a $C^\infty$-diffeomorphism.
Assume that there exists a family $(K_i)_{i\in I}$
of compact full submanifolds $K_i$ of $N$
whose interiors cover $N$,
with the following properties:
\begin{itemize}
\item[\rm(a)]
For each $i\in I$, we have $K_i=K_{i,1}\times \cdots\times
K_{i,m}$ with certain compact full submanifolds $K_{i,\ell}\sub N_\ell$;
and
\item[\rm(b)]
$C^\beta(M,C^\alpha(K_i,G))$
admits a canonical smooth manifold structure
for each $i\in I$,
using the canonical smooth manifold
structure on $C^\alpha(K_i,G)$
provided by Theorem~{\rm \ref{thmA}}.
\end{itemize}
Then the pre-canonical
manifold structure on $C^{\alpha,\beta}(N\times M,G)$
is canonical.
\end{la}
\begin{proof}
Let $\cO$ be the topology on
$C^{\alpha,\beta}(N\times M,G)$,
equipped with its pre-canonical smooth manifold structure.
Using Theorem~\ref{thmA}, for $i\in I$ we endow
$C^\alpha(K_i,C^\beta(M,G))$
with a canonical smooth manifold structure;
the underlying topology is the compact-open $C^\alpha$-topology.
The given smooth manifold structure
on $C^\alpha(N,C^\beta(M,G))$ being canonical,
its underlying topology is the compact-open
$C^\alpha$-topology,
which is initial with respect to the
restriction maps
\[
\rho_i\colon C^\alpha(N,C^\beta(M,G))\to C^\alpha(K_i,C^\beta(M,G))
\]
for $i\in I$.
We have bijections
\[
C^\alpha(K_i,C^\beta(M,G))\cong
C^{\alpha,\beta}(K_i\times M,G)\cong
C^{\beta,\alpha}(M\times K_i,G)\cong C^\beta(M, C^\alpha(K_i,G))
\]
using in turn the Exponential Law
(in the form (\ref{pre-can-bij})),
a flip in the factors (cf.\ Lemma~\ref{la:ISO2}\,(a)),
and again the Exponential Law.
If, step by step,
we transport the
smooth manifold structure from the left to the right,
we obtain a pre-canonical smooth manifold structure in each step
(see Lemmas~\ref{la:ISO1}\,(c) and \ref{la:ISO2}\,(a)).
As pre-canonical structures are unique,
the pre-canonical structure obtained on $C^\beta(M,C^\alpha(K_i,G))$
must coincide with the canonical structure
which exists by hypothesis. Hence, using this canonical structure,
the map
\[
\Psi_i\colon C^\alpha(K_i,C^\beta(M,G))\to C^\beta(M,C^\alpha(K_i,G))
\]
determined by $\Psi(f)(y)(x)=f(x)(y)$
is a $C^\infty$-diffeomorphism.
Let $\cL_k$ be the set
of compact full submanifolds
of $M_k$ for $k\in\{1,\ldots,n\}$.
Write $\cL_1\times\cdots\times \cL_n=:J$.
If $j\in J$,
then $j=(L_{j,1},\ldots, L_{j,n})$
with certain compact full submanifols $L_{j,k}\sub M_k$;
we define $L_j:=L_{j,1}\times\cdots\times L_{j,n}$.
By Lemma~\ref{la:ascending-union},
the topology on $C^\beta(M,C^\alpha(K_i,G))$
is initial with respect to the restriction maps
\[
r_{i,j}\colon C^\beta(M,C^\alpha(K_i,G))\to C^\beta(L_j,C^\alpha(K_i,G)),
\]
using the compact-open $C^\alpha$-topology
on the range which underlies the canonical
smooth manifold structure given by Theorem~\ref{thmA}.
Let $\Theta_{i,j}$ be the composition of the bijections
\[
C^\beta(L_j,C^\alpha(K_i,G))\to C^{\beta,\alpha}(L_j\times K_i,G)
\to C^{\alpha,\beta}(K_i\times L_j,G);
\]
thus $\Theta_{i,j}(f)(x,y)=f(y)(x)$.
As each of the domains and ranges admits a canonical
smooth manifold structure (by Theorem~\ref{thmA}), all of the maps
have to be homeomorphisms
(see Proposition~\ref{explaw-precan}
and Lemma~\ref{la:ISO2}\,(b)).
Thus $\Theta_{i,j}$ is a homeomorphism.
By transitivity of initial topologies, $\cO$ is initial
with respect to the mappings
\[
\rho_{i,j}:=\Theta_{i,j}\circ r_{i,j}\circ \Psi_i\circ \rho_i\circ \Phi
\;\;
\mbox{for $i\in I$ and $j\in J$,}
\]
which are the restriction maps $C^{\alpha,\beta}(N\times
M,G)\to C^{\alpha,\beta}(K_i\times L_j,G)$.
Also the compact-open $C^{\alpha,\beta}$-topology
on $C^{\alpha,\beta}(N\times M,G)$
is initial with respect to
the maps $\rho_{i,j}$, and hence coincides with~$\cO$.
The given pre-canonical smooth manifold
structure on $C^{\alpha,\beta}(N\times M,G)$
therefore is canonical.
\end{proof}
\begin{la}\label{reorder-eval}
Let $M_1,\ldots, M_n$ be locally compact,
smooth manifold with rough boundary,
$M:=M_1\times\cdots\times M_n$
$\alpha\in (\N_0\cup\{\infty\})^n$,
and $G$ be a Lie group.
Assume that the group $C^\alpha(M,G)$
is endowed with a smooth manifold structure which makes it a Lie group
and is compatible with evaluations.
Let $\sigma$ be a permutation of $\{1,\ldots, n\}$ and
$Q:=M_{\sigma(1)}\times\cdots \times M_{\sigma(n)}$.
Consider $\phi_\sigma\colon M\to Q$, $x\mto x\circ\sigma$.
Then the smooth manifold $($and Lie group$)$ structure
on the group $C^{\alpha\circ \sigma}(Q,G)$
making the bijective group homomorphism
\[
(\phi_\sigma)^*\colon C^{\alpha\circ\sigma}(Q,G)\to C^\alpha(M,G),\;\, f\mto f\circ\phi_\sigma
\]
a $C^\infty$-diffeomorphism
is compatible with evaluations.
\end{la}
\begin{proof}
The map
$\psi\colon C^{\alpha\circ\sigma}(Q,L(G))\to C^\alpha(M,L(G))$,
$f\mto f\circ \phi_\phi$
is an isomorphism of topological vector spaces,
by Example~\ref{basic-examples} and Lemma~\ref{la:ISO2}\,(b).
Write $\bar{\ve}_y\colon C^{\alpha\circ\sigma}(Q,G)\to G$
for the point evaluation at $y\in Q$
and $\ve_x\colon C^\alpha(M,G)\to G$
for the point evaluation at $x\in M$.
For $v\in L(C^\alpha(M,G))$, let $\phi(v):=(L(\ve_x(v))_{x\in M}$.
Then $\ve_x\circ (\phi_\sigma)^*=\bar{\ve}_{\phi_\sigma(x)}$.
As a consequence,
\[
\bar{\phi}(v):=(L(\bar{\ve}_y)(v))_{y\in Q}=
(\psi^{-1}\circ \phi\circ L((\phi_\sigma)^*))(v)
\in C^{\alpha\circ\sigma}(Q,L(G))
\]
for all $v\in L(C^{\alpha\circ\sigma}(Q,G))$.
Moreover,
$\bar{\phi}=(\psi^{-1})^*\circ \phi\circ L((\phi_\sigma)^*)$
is an isomorphism of topological
vector spaces, being a composition of such.
\end{proof}
{\bf Proof of Theorem~\ref{thmB} and Proposition~\ref{amendment}.}
Step~1. We first assume that $M_j$
is $1$-dimensional with finitely many components
for all $j\in\{1,\ldots, n\}$,
and prove the assertions by induction on~$n$.
The case $n=1$ was treated in Lemma~\ref{basic-case-3}.
We may therefore assume that $n\geq 2$ and assume that
the conclusions hold for $n-1$ factors.
We abbreviate $k:=\alpha_1$, $\beta:=(\alpha_2,\ldots, \alpha_n)$,
and $L:=M_2\times\cdots\times M_n$.
By the inductive hypothesis, $C^\beta(L,G)$
admits a canonical smooth manifold structure which makes it
a $C^r$-regular Lie group
and is compatible with evaluations.
By the induction base, $C^k(M_1,C^\beta(L,G))$
admits a canonical smooth manifold structure
making it a $C^r$-regular Lie group.
Since $C^\beta(L,G)$ is canonical, the group homomorphism
\[
\Phi\colon C^{k,\beta}(M_1\times L,G)\to C^k(M_1,C^\beta(L,G)),\;\,
f\mto f^\vee
\]
is a bijection (see (\ref{explaw-precan})). We endow
\[
C^\alpha(M,G)=C^{k,\beta}(M_1\times L,G)
\]
with the smooth manifold structure
turning $\Phi$ into a $C^\infty$-diffeomorphism.
By Lemma~\ref{wedge-is-lie}, this structure is pre-canonical,
makes $C^\alpha(M,G)$ Lie group,
and is compatible with evaluations.
The Lie group $C^\alpha(M,G)$ is $C^r$-regular, as $\Phi$
is an isomorphism of Lie groups.
Let $C_1,\ldots, C_\ell$
be the connected components of $M_1$.
Let $\cK$ be the set
of compact, full submanifolds $K$ of~$M_1$.
Then the interiors $K^o$ cover~$M_1$
(as the interiors of connected,
compact full submanifolds
cover each connected component
of $M_1$, by the proof of Lemma~\ref{basic-case-3}).
Now $C^k(K,G)$
admits a canonical smooth manifold structure making it a $C^r$-regular Lie group,
by Lemma~\ref{basic-case-2}.
Thus $C^\beta(L,C^k(K,G))$ admits a canonical
smooth manifold structure, by the inductive hypothesis.
By Lemma~\ref{ind-correct-top},
the pre-canonical smooth manifold
structure on $C^\alpha(M,G)$ is canonical.\\[1mm]
Step 2 (the general case).
Let $M_1,\ldots, M_n$
be arbitrary.
Using~Lemma~\ref{la:ISO2}\,(a), we may re-order the factors
and assume that there
exists an $m\in \{0,\ldots,n\}$
such that $M_j$ is compact for all $j\in \{1,\ldots,n\}$
with $j\leq m$, while $M_j$ is $1$-dimensional
with finitely many components for
all $j\in\{1,\ldots,n\}$ such that $j>m$.
If $m=0$, we have the special case just settled.
If $m=n$, then all conclusions hold by Lemma~\ref{basic-case-2}.
We may therefore assume that $1\leq m<n$.
We abbreviate $K:=M_1\times\cdots\times M_m$
and $N:=M_{m+1}\times\cdots\times M_n$.
Let $\gamma:=(\alpha_1,\ldots,\alpha_m)$
and $\beta:=(\alpha_{m+1},\ldots,\alpha_n)$.
By Step~1, $C^\beta(N,G)$
admits a canonical smooth manifold structure
which makes it a
$C^r$-regular Lie group and is compatible
with evaluations.
By Lemma~\ref{basic-case-2},
$C^\gamma(K,C^\beta(N,G))$
admits a canonical smooth manifold
structure which makes it a $C^r$-regular Lie group
and is compatible with evaluations.
We give $C^\alpha(M,G)=C^{\gamma,\beta}(K\times N,G)$
the smooth manifold structure making the bijection
\[
\Phi\colon C^{\gamma,\beta}(K\times N,G)\to C^\gamma(K,C^\beta(N,G)),\;\,
f\mto f^\vee
\]
a $C^\infty$-diffeomorphism.
By Lemma~\ref{wedge-is-lie}, this smooth manifold structure
is pre-canonical, makes $C^\alpha(M,G)$ a Lie group,
and is compatible with evaluations.
The Lie group $C^\alpha(M,G)$ is $C^r$-regular as $\Phi$
is an isomorphism of Lie groups.
Now $C^\gamma(K,G)$ admits a canonical
smooth manifold structure, which makes it
a $C^r$-regular Lie group (Lemma~\ref{basic-case-2}).
By Step~1, $C^\beta(N,C^\gamma(K,G))$
admits a canonical smooth manifold structure.
The pre-canonical smooth manifold
structure on $C^\alpha(M,G)$
is therefore canonical,
by Lemma~\ref{ind-correct-top}. $\square$\\[2.3mm]
The following result complements Theorem~\ref{thmB}.
Under a restrictive hypothesis, it provides
a Lie group structure without recourse to
regularity.
\begin{prop}\label{basic-case-1}
Let $M_1,\ldots, M_n$
be locally compact smooth manifolds with rough boundary,
$\alpha\in (\N_0\cup\{\infty\})^k$
and $G$ be a Lie group
that is $C^\infty$-diffeomorphic
to a locally convex space~$E$.
Abbreviate $M:=M_1\times\cdots\times M_n$.
Then $C^\alpha(M,G)$
admits a canonical $C^\infty$-manifold
structure, which is compatible with evaluations.
If~$G$ is $C^r$-regular for some $r\in\N_0\cup\{\infty\}$,
then also $C^\alpha(M,G)$
is $C^r$-regular.
\end{prop}
\begin{proof}
By Example~\ref{basic-examples},
$H:=C^\alpha(M,G)$ admits a canonical smooth manifold
structure and this structure makes it a Lie
group (see Lemma~\ref{la-is-gp}).
Let $\psi\colon G\to E$ be a $C^\infty$-diffeomorphism
such that $\psi(e)=0$. Abbreviating $\cg:=L(G)$
and $\ch:=L(H)$,
the map $\alpha:=d\psi|_{\cg}\colon \cg \to E$
is an isomorphism of topological vector spaces.
Then also $\phi:=\alpha^{-1}\circ\psi\colon G\to E$
is a $C^\infty$-diffeomorphism such that $\phi(e)=0$;
moreover, $d\phi|_{\cg}=\id_{\cg}$. Now
\[
\phi_*\colon C^\alpha(M,G)\to C^\alpha(M,\cg), \;\,
f\mto\phi\circ f
\]
is a $C^\infty$-diffeomorphism, and thus
$\beta:=d(\phi_*)|_{\ch}\colon \ch\to C^\alpha(M,\cg)$ is an isomorphism
of topological vector spaces. For $x\in M$,
let $\ve_x\colon H\to G$
and $e_x\colon C^\alpha(M,\cg)\to \cg$
be the respective point evaluation at~$x$.
We show that $\beta(v)=(L(\ve_x)(v))_{x\in M}$
for each $v\in \ch$, whence the Lie group
structure on~$H$ is compatible with evaluations.
Regard $v=[\gamma]$ as a geometric tangent vector.
As $L(\ve_x)(v)\in\cg$, we have
\begin{eqnarray*}
L(\ve_x)(v) &=& d\phi(L(\ve_x)(v))=d(\phi\circ\ve_x)(v)
=\frac{d}{dt}\Big|_{t=0}(\phi\circ \ve_x\circ \gamma)(t)\\
&=&\frac{d}{dt}\Big|_{t=0}(e_x\circ \phi_*\circ\gamma)(t)
=e_x\frac{d}{dt}\Big|_{t=0}(\phi_*\circ \gamma)(t)=
d(\phi_*)(v)(x),
\end{eqnarray*}
since $(\phi\circ\ve_x\circ\gamma)(t)=\phi(\gamma(t)(x))=(\phi\circ \gamma(t))(x)=e_x (\phi_*(\gamma(t)))
=(e_x\circ \phi_*\circ\gamma)(t)$
and $e_x$ is continuous and linear.
For the final assertion, see
Lemma~\ref{expected-lie-regular}.
\end{proof}
\section{Manifolds of maps with finer topologies}\label{sec-box}
We now turn to manifold structures
on $C^\alpha(M,N)$ for non-compact~$M$,
which are modeled on suitable spaces of compactly supported
$C^\alpha$-functions.
Notably, a proof for Theorem~\ref{thmC} will be provided.
Such manifold structures need not be compatible
with the compact-open $C^\alpha$-topology,
and need not be pre-canonical.
But we can essentially reduce their structure
to the case of canonical structures
for compact domains,
using box products of manifolds
as a tool. We recall pertinent concepts from~\cite{Rou}.
\begin{numba}\label{fine-box-top}
If $I$ is a non-empty set and $(M_i)_{i\in I}$
a family of $C^\infty$-manifolds modeled
on locally convex spaces, then the
\emph{fine box topology} $\cO_{\fb}$
on the
cartesian product $P:=\prod_{i\in I}M_i$
is defined as
the final topology with respect to the mappings\vspace{-.7mm}
\begin{equation}\label{para-box}
\Theta_\phi \colon \bigoplus_{i\in I}V_i:=
\left(\bigoplus_{i\in I}E_i\right)\cap\prod_{i\in I}V_i\to P,
\;\, (x_i)_{i\in I}\mto(\phi_i^{-1}(x_i))_{i\in I},\vspace{-.7mm}
\end{equation}
for $\phi:=(\phi_i)_{i\in I}$ ranging through the
families of charts $\phi_i\colon U_i\to V_i\sub E_i$
of~$M_i$ such that $0\in V_i$;
here $E_\phi:=\bigoplus_{i\in I}E_i$
is endowed with the locally convex direct sum topology,
and the left-hand side $V_\phi$ of (\ref{para-box}),
which is an open subset of~$E_\phi$,
is endowed with the topology induced by~$E_\phi$.
Let $U_\phi:=\Theta_\phi(V_\phi)$.
Thus
\[
U_\phi=\Big\{(y_i)_{i\in I}\in \prod_{i\in I}U_i\colon
\mbox{$y_i\not=\phi_i^{-1}(0)$ for only finitely many $i\in I$}\Big\}.\vspace{-.7mm}
\]
Note that the projection $\pr_i\colon P\to M_i$
is continuous for each $i\in I$, entailing that
the fine box topology is Hausdorff.
In fact, using the continuous linear projection
$\pi_i\colon E_\phi\to E_i$ onto the $i$th component,
we deduce from the continuity of $\pr_i\circ\, \Theta_\phi=\phi_i^{-1}\circ \pi_i|_{V_\phi}$
for each $\phi$
that $\pr_i$ is continuous.
\end{numba}
\begin{numba}\label{basics-fine-box}
Let $\phi$ be as before and $\psi$ be an analogous family
of charts
$\psi_i\colon R_i\to S_i\sub F_i$.
If $\phi_i^{-1}(0)=\psi_i^{-1}(0)$ for all
but finitely many $i\in I$, then
\[
(\Theta_\phi)^{-1}(U_\phi\cap U_\psi)=\bigoplus_{i\in I}\phi_i(U_i\cap R_i),\vspace{-.7mm}
\]
which is an
open $0$-neighbourhood in $\bigoplus_{i\in I}E_i$.
The transition map
\[
(\Theta_\phi)^{-1}\circ \, \Theta_\psi\colon \bigoplus_{i\in I}\psi_i(U_i\cap R_i)
\to \bigoplus_{i\in I}\phi_i(U_i\cap R_i),\;\,
(x_i)_{i\in I}\mto ((\phi_i\circ \psi_i^{-1})(x_i))_{i\in I}\vspace{-.7mm}
\]
is $C^\infty$
(as follows from
\cite[Proposition~7.1]{Mea})
and in fact a $C^\infty$-diffeomorphism,
and hence a homeomorphism,
since $\Theta_\psi^{-1}\circ \, \Theta_\phi$
is the inverse map.
If $\phi_i^{-1}(0)\not=\psi_i^{-1}(0)$ for
infinitely many $i\in I$,
then $(\Theta_\phi)^{-1}(U_\phi\cap U_\psi)=\emptyset$
and the transition map trivially is a homeomorphism.
Using a standard agrument, we now deduce that
$U_\phi=\Theta_\phi(V_\phi)$ is open in $(P,\cO_{\fb})$
for all~$\phi$
and $\Theta_\phi$ is a homeomorphism onto its image
(see, e.g., \cite[Exercise~A.3.1]{GaN}).
By the preceding,
the maps $\Phi_\phi:=(\Theta_\phi|^{U_\phi})^{-1}\colon U_\phi\to V_\phi\sub E_\phi$
are smoothly compatible 
and hence form an atlas for a $C^\infty$-manifold
structure on~$P$. Following~\cite{Rou},
we write $P^{\fb}$ for~$P$, endowed with the topology~$\cO_{\fb}$
and the smooth manifold structure just described,
and call $P^{\fb}$ the \emph{fine box product}.
\end{numba}
Some auxiliary results are needed. We use notation as in \ref{numba:mfdstruct}
and Theorem~\ref{thmC}.
\begin{la}\label{operations-with-Gamma}
Let $M:=M_1\times \cdots \times M_n$
be a product of locally compact smooth manifolds
with rough boundary,
$N$ be a smooth manifold, $\alpha\in (\N_0\cup\{\infty\})^n$
and $f\in C^\alpha(M,N)$.
\begin{itemize}
\item[\rm(a)]
If $M_1,\ldots, M_n$ are compact,
then the following bilinear map is continuous:\vspace{-.4mm}
\[
C^\alpha(M,\R)\times \Gamma_f\to\Gamma_f,\;\,
(h,\tau)\mto h\tau\;\,\mbox{with $(h\tau)(x)=h(x)\tau(x)$.}\vspace{-.9mm}
\]
\item[\rm(b)]
If $M_1,\ldots, M_n$ are paracompact,
$L\sub M$ is a compact subset
and $K:=K_1\times\cdots\times K_n$
with compact full submanifolds $K_j\sub M_j$
for $j\in\{1,\ldots,n\}$,
then the linear map
$\Gamma_{f,L}\to \Gamma_{f|_K}$, $\tau\mto \tau|_K$ is continuous.
\item[\rm(c)]
If $M_1,\ldots, M_n$ are paracompact,
$K:=K_1\times\cdots\times K_n$
with compact full submanifolds $K_j\sub M_j$
for $j\in\{1,\ldots,n\}$
and $L\sub K$ be compact.
Then
$r\colon \Gamma_{f,L}\to \Gamma_{f|_K,L}$,
$\tau\mto \tau|_K$ is an isomorphism of topological vector spaces.
\end{itemize}
\end{la}
\begin{proof}
(a) The bilinear map is a restriction
of the continuous mapping\linebreak
$\mu\colon C^\alpha(M,\R)\times C^\alpha(M,TN)\to
C^\alpha(M,TN)$ from Lemma~\ref{c-alpha-top-mult}.\\[1mm]
(b) The map is a restriction
of the restriction map $C^\alpha(M,TN)\to C^\alpha(K,TN)$,
which is continuous (see Remark~\ref{rem-restriction}).\\[1mm]
(c) For each $x$ in the open subset
$M\setminus K$ of~$M$,
there exist compact full submanifolds
$K_{x,j}\sub M_j$ for $j\in\{1,\ldots,n\}$
such that $K_x:=K_{x,1}\times\cdots\times K_{x,n}\sub M\setminus K$
and $x\in K_x^o$.
Lemma~\ref{la:ascending-union} implies that
the compact-open $C^\alpha$-topology on $\Gamma_{f,L}$
is initial with respect to the restriction maps
$\rho \colon \Gamma_{f,L}\to C^\alpha(K,TN)$
and $\rho_x\colon \Gamma_{f,L}\to C^\alpha(K_x,TN)$
for $x\in M\setminus K$.
As each $\rho_x$ is constant (its value
is the function $K_x\in y\mto 0_{f(y)}\in T_{f(y)}N$),
it can be omitted without affecting the initial topology.
The topology on $\Gamma_{f,K}$ is therefore
initial with respect to $\rho$,
and hence also with the co-restriction
$r$ of $\rho$. Thus $r$ is a topological embedding
and hence an homeomorphism, as $r(\tau)=\sigma$
can be achieved for $\sigma\in\Gamma_{f|_K,L}$
if we define
$\tau\colon M\to TN$ piecewise via $\tau(x):=\sigma(x)$ if $x\in K$,
$\tau(x):=0_{f(x)}\in T_{f(x)}N$ if $x\in M\setminus L$.
Being linear, $r$ is an isomorphism of topological
vector spaces.\vspace{.8mm}
\end{proof}
{\bf Proof of Theorem~\ref{thmC}.}
For $j\in \{1,\ldots, n\}$,
let $(K_{j,i})_{i\in I_j}$ be a
locally finite family of 
compact, full submanifolds $K_{j,i}$
of~$M_j$ whose interiors cover~$M_j$.
Let $I:=I_1\times \cdots\times I_n$.
Then the sets $K_i:=K_{1,i_1}\times \cdots\times K_{n,i_n}$
form a locally finite family
of compact full submanifolds of~$M$
whose interiors cover~$M$,
for $i=(i_1,\ldots, i_n)\in I$.
The map
\[
\rho\colon C^\alpha(M,N)\to\prod_{i\in I}C^\alpha(K_i,N),\;\,
f\mto (f|_{K_i})_{i\in I}\vspace{-1.3mm}
\]
is injective with image
\begin{equation}\label{identify-image}
\hspace*{-1.7mm}\im(\rho)\hspace*{-.3mm}=\hspace*{-.5mm}\Big\{\hspace{-.1mm}(f_i)_{i\in I}\hspace*{-.5mm}\in
\hspace*{-.5mm}\prod_{i\in I}\!C^\alpha(K_i,N)
\colon \!(\forall \hspace*{.1mm}i,j\in I)\,(\forall x\in K_i\cap K_j)\,
f_i(x)\hspace*{-.3mm}=\hspace*{-.3mm}f_j(x)\hspace{-.1mm}\Big\}\hspace*{-.1mm}.\vspace{-.7mm}
\end{equation}
In fact, the inclusion ``$\sub$''
is obvious. If $(f_i)_{i\in I}$
is in the set on the right-hand side,
then a piecewise definition,
$f(x):=f_i(x)$ if $x\in K_i$,
gives a well-defined function $f\colon M\to N$
which is $C^\alpha$ since $f|_{(K_i)^o}=f_i|_{(K_i)^o}$
is $C^\alpha$ for each $i\in I$.
Then $\rho(f)=(f_i)_{i\in I}$.\\[2mm]
For each $i\in I$,
endow $C^\alpha(K_i,N)$
with the canonical smooth manifold structure,
as in Theorem~\ref{thmA}, modeled
on the set $\{\Gamma_f\colon f\in C^\alpha(K_i,N)\}$
of the locally convex spaces
$\Gamma_f:=\{\tau\in C^\alpha(K_i,TN)\colon \pi_{TN}\circ
\tau=f\}$
for $f\in C^\alpha(K_i,N)$.
Let
$\Sigma\colon TN\supseteq U\to N$
be a local addition for~$N$;
as in Section~\ref{sec-compact}, write
$U':=\{(\pi_{TN}(v),\Sigma(v))\colon v\in U\}$
and $\theta:=(\pi_{TN}|_U,\Sigma)\colon U\to U'$.
For $f\in C^\alpha(K_i,N)$,
consider
$O_f:=\Gamma_f\cap C^\alpha(K_i,U)$,
$O_f':=\{g\in C^\alpha(K_i,N)\colon (f,g)\in C^\alpha(K_i,U')\}$,
and $\phi_f\colon O_f\to O_f'$, $\tau\mto \Sigma\circ \tau$
as in Section~\ref{sec-compact}.
For $f\in C^\alpha(M,N)$,
let
$\Gamma_f$
be the set of all $\tau\in C^\alpha(M,TN)$
such that $\pi_{TN}\circ
\tau=f$ and
\[
\{x\in M\colon \tau(x)\not= 0_{f(x)}\in T_{f(x)}N\}
\]
is relatively compact in~$M$.
Define $O_f:=\Gamma_f\cap C^\alpha(M,U)$
and let~$O_f'$ be the set of all
$g\in C^\alpha(M,N)$ such that
\[
(f,g)\in C^\alpha(M,U')\;\;\mbox{and}\;\; \mbox{$g|_{M\setminus K}=f|_{M\setminus K}$
for some compact subset $K\sub M$.}
\]
Then $\phi_f\colon O_f\to O_f'$, $\tau\mto \Sigma\circ \tau$
is a bijection with $(\phi_f)^{-1}(g)=\theta^{-1}\circ (f,g)$.
The linear map
\[
s\colon \Gamma_f\to \bigoplus_{i\in I}\Gamma_{f|_{K_i}},\quad
\tau\mto (\tau|_{K_i})_{i\in I}
\]
is continuous on $\Gamma_{f,L}$ for each
compact subset $L\sub M$ (see Lemma~\ref{operations-with-Gamma}\,(b))
and hence continuous on the locally convex direct
limit $\Gamma_f$.
As above, we see that
\begin{equation}\label{id-image-2}
\im(s)=\{(\tau_i)_{i\in I}\in\bigoplus_{i\in I}\Gamma_{f|_{K_i}}\colon
(\forall i,j\in I)\,(\forall x\in K_i\cap K_j)\; \tau_i(x)=\tau_j(x)\},
\end{equation}
which is a closed vector subspace of $\bigoplus_{i\in I}
\Gamma_{f|_{K_i}}$.
We now show that $s$ is a homeo\-{}morphism onto its
image. In fact, $s$
admits a continuous
linear left inverse.
To see this,
pick a $C^\infty$-partition of unity
$(h_i)_{i\in I}$ on~$M$
subordinate to $(K_i^o)_{i\in I}$;
then $L_i:=\Supp(h_i)$ is a closed subset of $K_i$
and thus compact.
The multiplication operator
$\beta_i\colon \Gamma_{f|_{K_i}}\to \Gamma_{f|_{K_i},L_i}$,
$\tau\mto h_i\tau$
is continuous linear (by
Lemma~\ref{operations-with-Gamma}\,(a)).
Moreover, the restriction operator
$s_i\colon \Gamma_{f,L_i}\to\Gamma_{f|_{K_i},L_i}$
is an isomorphism of topological vector spaces
(Lemma~\ref{operations-with-Gamma}\,(c)).
Thus $s_i^{-1}\circ \beta_i\colon \Gamma_{f|_{K_i}}\to\Gamma_{f,L_i}\sub \Gamma_f$
is a continuous linear map. By the universal property
of the locally convex direct sum, also the linear map
\[
\sigma\colon \bigoplus_{i\in I}\Gamma_{f|_{K_i}}\to \Gamma_f,\;\,
(\tau_i)_{i\in I}\mto \sum_{i\in I}(s_i^{-1}\circ \beta_i)(\tau_i)
\]
is continuous.
We easily verify that
$\sigma\circ s=\id_{\Gamma_f}$.\\[2mm]
Abbreviate $\phi_i:=(\phi_{f|_{K_i}})^{-1}$
and $\phi:=(\phi_i)_{i\in I}$.
We now use the $C^\infty$-diffeomorphism
\[
\Theta_\phi\colon
\bigoplus_{i\in I} O_{f|_{K_i}}\to U_\phi,\;
(\tau_i)_{i\in I}\mto (\phi_i^{-1}(\tau_i))_{i\in I}=
(\Sigma \circ \tau_i)_{i\in I}
\]
from \ref{fine-box-top}, the inverse of which is the chart
\[
\Phi_\phi\colon U_\phi\to
\bigoplus_{i\in I}O_{f|_{K_i}}, \;\,
(g_i)_{i\in I}\mto
(\phi_i(g_i))_{i\in I}\vspace{-.7mm}
\]
of $\prod_{i\in I}^{\fb}C^\alpha(K_i,N)$
around $(f|_{K_i})_{i\in I}$.
For $(\tau_i)_{i\in I}\in\bigoplus_{i\in I}O_{f|_{K_i}}$,
we have
\[
\Theta_\phi((\tau_i)_{i\in I})
\in \im(\rho)\;\;\Leftrightarrow\;\;
(\tau_i)_{i\in I}\in  \im(s).
\]
In fact, for $i,j\in I$ and $x\in K_i\cap K_j$
we have
$\Sigma(\tau_i(x))=\Sigma(\tau_j(x))$
if and only if $\tau_i(x)=\tau_j(x)$, from
which the assertion follows in view of (\ref{identify-image})
and (\ref{id-image-2}).
Thus
\[
\Phi_\phi(\im(\rho)\cap U_\phi)
=\im(s)\cap \bigoplus_{i\in I}O_{f|_{K_i}},\vspace{-.9mm}
\]
showing that $\im(\rho)$
is a submanifold of $\prod_{i\in I}^{\fb}C^\alpha(K_i,N)$.
Let
\[
\Psi_\phi\colon \im(\rho)\cap U_\phi
\to \im(s)\cap \bigoplus_{i\in I}O_{f|_{K_i}},\;\;
(g_i)_{i\in I}\mto \Phi_\phi((g_i)_{i\in I})
\]
be the corresponding submanifold
chart for $\im(\rho)$.
Then
\[
\rho(O_f')=\im(\rho)\cap U_\phi\quad\mbox{and}\quad
s(O_f)=\im(s)\cap \bigoplus_{i\in I}O_{f|_{K_i}}.\vspace{-.9mm}
\]
Hence $(\phi_f)^{-1}=s^{-1}\circ \Psi_\phi\circ \rho|_{O_f'}\colon O_f'\to O_f$
is a chart for the smooth manifold structure on $C^\alpha(M,N)$
modeled on $\cE$ (the set of all $\Gamma_f$)
which makes\linebreak
$\rho\colon C^\alpha(M,N)\to \im(\rho)$
a $C^\infty$-diffeomorphism.
Note that the smooth manifold structure
on $C^\alpha(M,N)$ which is modeled on~$\cE$
and makes $\rho$ a $C^\infty$-diffeomorphism
is uniquely determined by these properties.
Thus, it is independent of the choice of~$\Sigma$.
On the other hand, the $(\phi_f)^{-1}$
form a $C^\infty$-atlas for a given
local addition~$\Sigma$.
As the definition of the $\phi_f$
does not involve the cover $(K_i)_{i\in I}$,
the smooth manifold structure just constructed
is independent of the choice of $(K_i)_{i\in I}$. $\,\square$
\appendix
\section{Details for Sections~\ref{secprels} and \ref{sec-alpha-tops}}\label{appA}
In this appendix, we provide proofs
for preliminaries in Sections~\ref{secprels} and \ref{sec-alpha-tops}.\\[2mm]
{\bf Proof of Lemma~\ref{reparametrize}.}
The right-hand side $(t,y)\mto y.\gamma(t)$
of the differential equation $\dot{y}(t)=y(t).\gamma(t)$
is $C^k$, whence its solution $\eta$ will
be $C^{k+1}$, if it exists.\\[1.3mm]
To verify existence and uniqueness
of~$\eta$, we may assume that $I$
is a non-degenerate compact interval
with initial point $0$ or endpoint $0$,
since $I$ is covered by such intervals.
Thus, let $I$ be a line segment joining $0$
and $\tau\not=0$. Define $\xi\colon [0,1]\to \cg$
via $\xi(t):=\tau \gamma(t\tau)$.
By the Chain Rule, a $C^1$-function
$\eta\colon I\to G$ with $\eta(0)=e$
satisfies $\delta^\ell\eta=\gamma$
if and only if $\theta\colon [0,1]\to G$, $t\mto \eta(t \tau)$
satisfies $\delta^\ell\theta=\xi$.
The assertion now follows
from the case $I=[0,1]$,
which holds by $C^r$-semiregularity. $\,\square$\\[2mm]
{\bf\small Proof for Lemma~\ref{la-sub-PL}.}
(a) Let $\lambda\colon Y\to F$ be the inclusion map,
which is continuous linear and thus smooth.
If $f|^Y$ is $C^\alpha$,
then also $f=\lambda\circ f|^Y$ is $C^\alpha$,
by the Chain Rule \cite[Lemma~3.16]{Alz}.
Conversely, assume that $f$ is $C^\alpha$
and $f(U)\sub Y$.
It suffices to deduce that $f|^Y$
is $C^\alpha$ if $\alpha\in (\N_0)^n$.
The proof is by induction on~$|\alpha|$,
and establishes in parallel that $d^\beta (f|^Y)=(d^\beta f)|^Y$
for all $\beta\leq\alpha$.
If $|\alpha|=0$, the conclusion holds
since $f|^Y$ is continuous.
If $|\alpha|\geq 1$,
let $j\in \{1,\ldots,n\}$
be minimal with $\alpha_j>0$.
Then $d^\beta(f|^Y)$ exists
for all $\beta\leq\alpha$ such that $\beta_j\leq \alpha_j-1$,
and equals $(d^\beta f)|^Y$.
If $\beta\leq \alpha$ with $\beta_j=\alpha_j$,
let $x\in U^o$ and $y_i\in E_i^{\beta_i}$
for $i\in \{j,\ldots, n\}$.
Then all difference quotients needed to define
\[
d^\beta f(x,0,\ldots, 0,y_j,y_{j+1},\ldots,y_n)
\]
are linear combinations of function values
of $d^{\beta-e_j}f$ and hence in~$Y$.
Since $Y$ is closed,
the limit
$d^\beta f(x,0,\ldots, 0,y_j,y_{j+1},\ldots,y_n)$
is in $Y$ as well, and this remains valid
for $x\in U$, by density of $U^o$ in~$U$.
Thus $(d^\beta f)|^Y$ is a continuous function
which extends $d^\beta (f|^Y_{U^o})$.
We deduce that $f|^Y$ is $C^\alpha$
and $d^\beta(f|^Y)=(d^\beta f)|^Y$.\\[1mm]
(b) If $f$ is $C^\alpha$, then also $\lambda_a\circ f$,
using that $\lambda_a$ is continuous linear and thus smooth.
Conversely, assume that $\lambda_a\circ f$ is $C^\alpha$
for all $a\in A$. Then
\[
Y:=\{(x_a)_{a\in A}\in\prod_{a\in A}F_a\colon (\forall a\leq b)\;
x_a=\lambda_{a,b}(x_b)\}
\]
is a closed vector subspace of $\prod_{a\in A}F_a$
and the map
\[
\lambda\colon F\to Y,\quad x\mto (\lambda_a(x))_{a\in A}
\]
is an isomorphism of topological vector spaces.
Let $\pr_a\colon Y\to F_a$ be the projection
onto the $a$th component.
Since $\pr_a\circ\lambda\circ f=\lambda_a\circ f$
is $C^\alpha$ for all $a\in A$,
the map $\lambda\circ f$ is $C^\alpha$
to $\prod_{a\in A}F_a$ by \cite[Lemma~3.8]{Alz}.
By (a), $\lambda\circ f$
is $C^\alpha$ also as a map to~$Y$. Thus
$f=\lambda^{-1}\circ (\lambda\circ f)$
is $C^\alpha$. $\,\square$\\[2mm]
{\bf Proof of Lemma~\ref{specialized-PL}.}
If $f$ is $C^\alpha$, then $\rho_a\circ f$
is $C^\alpha$ for each $a\in A$, the map $\rho_a$
being smooth.
Assume that, conversely, $\rho_a\circ f$ is smooth
for each $a\in A$.
Write $\psi=(\psi_1,\psi_2)$
with $\psi_1\colon M\to F$ and $\psi_2\colon M\to N$.
Since $\psi_a$ is smooth,
$\psi_a\circ \rho_a\circ f=(\lambda_a\times\id_N)\circ \psi\circ f$ is $C^\alpha$,
whence so is its second component
$\psi_2\circ f$
(see \cite[Lemma~3.8]{Alz}).
Also the first component
$\lambda_a\circ\psi_1\circ f$
is $C^\alpha$ for each $a\in A$,
whence $\psi_1\circ f$
is $C^\alpha$ by Lemma~\ref{la-sub-PL}\,(b).
Hence $\psi\circ f$ is $C^\alpha$,
by \cite[Lemma~3.8]{Alz},
and hence so is $f=\psi^{-1}\circ (\psi\circ f)$. $\,\square$\\[2mm]
{\bf Proof of Lemma~\ref{ingrisch}.}
The proof is by induction on
$m:=m_1+\cdots+m_n$. If $m=n$,
there is nothing to show.
Assume that $m>n$.
After a permutation of $E_1,\ldots, E_n$, we may assume
that $m_n\geq 2$ (cf.\ Lemma~\ref{reorder}).
Let $(\beta_1,\ldots, \beta_{n-1})\in \prod_{i=1}^{n-1}
(\N_0\cup\{\infty\})^{m_i}$,
$\beta_n=(\beta_{n,1},\ldots,\beta_{n,m_n-1})\in (\N_0\cup\{\infty\})^{m_n-1}$
such that $|\beta_i|\leq \alpha_i$
for all $i\in\{1,\ldots,n\}$.
Abbreviate $\beta_n':=(\beta_{n,1},\ldots,\beta_{n,m_n-2})$.
For all $k,\ell\in \N_0$
such that $k+\ell\leq \beta_{n,m_n-1}$,
the map $f$ is $C^{\beta_1,\ldots,\beta_{n-1},\beta_n',k,\ell}$.
Hence
\[
f\colon \prod_{i=1}^{n-1}\prod_{j=1}^{m_i}
U_{i,j}\times U_{n,1}\times
\cdots\times U_{n,m_n-2}\times (U_{n,m_n-1}\times U_{n,m_n})\to F
\]
is $C^{\beta_1,\ldots,\beta_n}$,
by \cite[Lemma~3.12]{Alz}.
By the inductive hypothesis, $f$ is $C^\alpha$. $\,\square$\\[2mm]
The following lemma fills in the details
for \ref{tangent-maps}.
\begin{la}\label{partial-tangents}
Let $M_1,\ldots, M_n$ and $N$ be smooth manifolds
with rough boundary, $M:=M_1\times\cdots\times M_n$
and $f\colon M\to N$ be a $C^\alpha$-map with $\alpha\in (\N_0\cup\{\infty\})^n$.
Then $f(\bar{x},\cdot)\colon M_n\to N$
is $C^{\alpha_n}$ for each $\bar{x}\in \bar{M}:=M_1\times \cdots \times M_{n-1}$
and
\[
h_k\colon
M_1\times\cdots\times M_{n-1}\times T^k(M_n)\to T^kN,\;\,
(\bar{x},v)\mto T^k(f(\bar{x},\cdot))(v)
\]
is a $C^{\alpha-ke_n}$-map
for all $k\in \N_0$ such that $k\leq \alpha_n$.
\end{la}
\begin{proof}
We show by induction on $k_0\in\N$ that the conclusion
holds with $k\leq k_0$ for all functions as described in the lemma,
for all $\alpha$ with $\alpha_n\geq k_0$.
Using local charts, we may assume that $U_j:=M_j$ is a locally
convex subset of a locally convex space $E_j$ for all $j\in \{1,\ldots, n\}$
and $N$ a locally convex subset of a locally convex space~$F$;
thus $f$ is a map $U:=U_1\times\cdots\times U_n\to F$.
The case $k_0=0$ being trivial as $h_0=f$ is $C^\alpha$.
Let $1\leq k_0\leq \alpha_n$ now.
Then
\[
d^{e_n}f\colon U_1\times \cdots\times U_n\times E_n\to F
\]
is a $C^{(\alpha-e_n,0)}$-map.
Being linear in the final argument,
$d^{e_n}f$ is $C^{\alpha-e_n}$
as a map
\[
U_1\times \cdots\times U_{n-1}\times (U_n\times E_n)\to F
\]
of $n$ variables, i.e., as a map on
the domain $T^{e_n}U=
U_1\times U_{n-1}\times TU_n$
(see \cite[Lemma~3.11]{Alz}). Let $\pr_1\colon TU_n=U_n\times E_n\to U_n$
be the projection onto the first component.
Then $g:=f\circ \id_{U_1}\times\cdots\times \id_{U_{n-1}}\times \pr_1\colon
U_1\times \cdots \times U_{n-1}\times TU_n
\to F$ is $C^\alpha$ by the Chain Rule
\cite[Lemma~3.16]{Alz},
and hence $C^{\alpha-e_n}$.
Thus $h_1=(g,d^{e_n}f)$ is $C^{\alpha-e_n}$,
by \cite[Lemma~3.8]{Alz}.
By the inductive hypothesis,
the maps
\[
U_1\times\cdots\times U_{n-1}\times T^j(TU_n)\to T^j(TF),\quad
(\bar{x},v)\mto T^j(h_1(x,\cdot))(v)
\]
are $C^{\alpha-e_n-je_n}$ for all $j\in\{0,\ldots,k_0-1\}$.
It only remains to observe that this map equals
$h_{j+1}$.
\end{proof}
{\bf Proof of Lemma~\ref{pull-and-push}.}
(a) For $\beta\in \N_0^n$ with $\beta\leq\alpha$,
consider the maps
\[
T^\beta\colon C^\alpha(M,N)\to
C(T^\beta M,T^{|\beta|} N),\;\,
f\mto T^\beta f
\]
and
$\tau_\beta\colon C^\alpha(M,L)
\to C(T^\beta N,T^{|\beta|} L)$,
$f\mto T^\beta f$.
Going through the recursive construction
of $T^\beta(g\circ f)$ in \ref{tangent-maps}
for $f\in C^\alpha(M,N)$
and making repeated use of the functoriality of~$T$,
we see that
\begin{equation}
T^\beta(g\circ f)=T^{|\beta|}g \circ T^\beta f.
\end{equation}
Thus $\tau_\beta\circ C^\alpha(M,g)=C(T^\beta M,T^{|\beta|}g)\circ T^\beta$,
which is a continuous map by \cite[Lemma~A.6.3]{GaN}.
The topology on $C^\beta(M,L)$ being initial with respect to the maps
$\tau_\beta$, we deduce that $C^\alpha(M,g)$
is continuous.\\[1mm]
(b)
For $\beta\in \N_0^n$ with $\beta\leq\alpha$,
consider the maps $T^\beta\colon C^\alpha(M,N)\to C(T^\beta M,T^{|\beta|} N)$,
$f\mto T^\beta f$ and
$\tau_\beta\colon C^\alpha(L,N)\to C(T^\beta L,T^{|\beta|} N)$,
$f\mto T^\beta f$.
Going through the recursive construction
of $T^\beta(f\circ g)$ in \ref{tangent-maps}
for $f\in C^\alpha(M,N)$
and making repeated use of the functoriality of~$T$,
we see that
\begin{equation}
T^\beta(f\circ g)=(T^\beta f) \circ h_\beta
\end{equation}
with $h_\beta:= T^{\beta_1} g_1\times\cdots\times T^{\beta_n}g_n$.
Thus $\tau_\beta\circ C^\alpha(g,N)=C(h_\beta,T^{|\beta|}N)\circ T^\beta$,
which is a continuous map by \cite[Lemma~A.6.9]{GaN}.
The topology on $C^\alpha(L,N)$ being initial with respect to the maps
$\tau_\beta$, we deduce that $C^\alpha(g,N)$
is continuous. $\,\square$\\[2mm]
{\bf Proof of Lemma~\ref{la:ascending-union}.}
By definition,
the compact-open $C^\alpha$-topology $\cO$ on $C^\alpha(M,N)$
is initial with respect to the maps $\tau_\beta\colon
C^\alpha(M,N)\to C(T^\beta M,T^{|\beta|} N)$, $f\mto T^\beta f$
for $\beta\in (\N_0)^n$ such that $\beta\leq\alpha$.
As the interiors $(T^\beta K_i)^o$
cover $T^\beta M$, the compact-open
topology on $C(T^\beta M,T^{|\beta|} N)$
is initial with respect to the restriction maps
$\rho_{\beta,i}\colon C(T^\beta M,T^{|\beta|} N)\to C(T^\beta K_i,T^{|\beta|} N)$,
by \cite[Lemma~A.6.11]{GaN}.
By transitivity of initial topologies, $\cO$
is initial with respect to the mappings $\rho_{\beta,i}\circ \tau_\beta$.
Let $\rho_i\colon C^\alpha(M,N)\to C^\alpha(K_i,N)$
the restriction map.
The compact-open $C^\alpha$-topology on $C^\alpha(K_i,N)$ being
initial with respect to the mappings
$\tau_{\beta,i}\colon$\linebreak
$C^\alpha(K_i,N)\to C(T^\beta K_i,T^{|\beta|} N)$, $f\mto T^\beta f$,
we deduce from
\[
\rho_{\beta,i}\circ \tau_\beta=\tau_{\beta,i}\circ \rho
\]
that $\cO$ is initial with respect to
the maps $\rho_i$. $\,\square$.\\[2mm]
{\bf Proof of Lemma~\ref{top-sub}.}
The case $n=1$ is well known.
The general case follows as
$T^\beta S=T^{\beta_1}S_1\times\cdots\times T^{\beta_n}S_n$
and $T^\beta M=T^{\beta_1}M_1\times\cdots\times T^{\beta_n}M_n$. $\,\square$\\[2mm]
{\bf Proof of Lemma~\ref{maps-to-sub}.}
The inclusion map $\lambda \colon S\to N$
is smooth.
By Lemma~\ref{top-sub},
the inclusion map $T^{|\beta|}\lambda\colon T^{|\beta|} S\to T^{|\beta|} N$
is a topological embedding,
for each $\beta\in(\N_0)^n$
such that $\beta\leq\alpha$.
Thus $(T^{|\beta|} \lambda)_*\colon C(T^\beta M,T^{|\beta|} S)
\to C(T^\beta M,T^{|\beta|} N)$
is a topological embedding for the compact-open
topologies (see, e.g., \cite[Lemma~A.6.5]{GaN}).
The compact-open $C^\alpha$-topology $\cO$
on $C^\alpha(M,S)$, which is initial with respect to the maps
$\tau_{\beta,S}\colon C^\alpha(M,S)\to C(T^\beta M, T^{|\beta|}S)$, $f\mto T^\beta f$
is therefore also initial with respect to the mappings
$(T^{|\beta|} \lambda)_*\circ \tau_{\beta,S}$.
The compact-open $C^\alpha$-topology on $C^\alpha(M,N)$
is initial with respect to the maps $\tau_{\beta, N}\colon
C^\alpha(M,N)\to C(T^\beta M,T^{|\beta|} N)$, $f\mto T^\beta f$.
As $(T^{|\beta|} \lambda)_*\circ \tau_{\beta,S}=\tau_{\beta,N}\circ \lambda_*$,\linebreak
we see that
the topology
$\cO$
is initial with respect to
the inclusion map\linebreak
$\lambda_*\colon C^\alpha(M,S)\to C^\alpha(M,N)$.
Thus $\cO$ is the induced topology. $\,\square$\\[2.3mm]
{\bf Proof of Lemma~\ref{maps-to-tvs-1}.}
For each $k\in\N_0$, $T^kF=F^{2^k}$ is a locally convex space.
For each $\beta\in (\N_0)^n$ such that
$\beta\leq\alpha$,
the map
\[
T^\beta\colon C^\alpha(M,F)\to C(T^\beta M,T^{|\beta|}F),\,\;
f\mto T^\beta f
\]
is linear. In fact, $T^k\colon C^k(N,F)\to C(T^kN,T^kF)$
is linear for each smooth manifold~$N$
with rough boundary \cite[proof of Proposition 4.1.11]{GaN} and $k\in \N_0$,
establishing linearity if $n=1$.
If $n\geq 2$, the preceding entails that
$T^{(0,\ldots,0,\beta_n)}f(v)=T^{\beta_n}(f(x_1,\ldots, x_{n-1},\cdot))(v_n)$
is linear in $f$ for all $x_j\in M_j$ for $j\in\{1,\ldots,n-1\}$
and $v_n\in T^{\beta_n}M_n$,
showing that $T^{(0,\ldots,0,\beta_n)}f$
is linear in~$f$.
Likewise, $g$ and $T^{(0,\ldots,0,\beta_{k-1},\ldots,\beta_n)}f$
is linear in $f$ in the recursive
construction in \ref{tangent-maps},
which gives the assertion for $n\geq 2$.
Thus
\[
C^\alpha(M,F)\to\prod_{\beta\leq\alpha}C(T^\beta M,T^{|\beta|}F),\;
f\mto (T^\beta f)_{\beta\leq\alpha}
\]
is a linear map. It is a homeomorphism
onto its image, which is a locally convex space.
Hence also $C^\alpha(M,F)$ is a locally convex space. $\,\square$\\[2mm]
{\bf Proof of Lemma~\ref{c-alpha-top-product}.}
(a) For each $k\in \N_0$,
the topology on $T^k F=F^{2^k}$
is initial with respect to the linear maps
$T^k\lambda_i=\lambda_i^{2^k}\colon F^{2^k}\to F_i^{2^k}$.
For each $\beta\in \N_0^n$ with $\beta\leq\alpha$,
the compact-open topology on
$C(T^\beta M,T^{|\beta|}F)$
is therefore initial with respect to the mappings
\[
C(T^\beta M,T^{|\beta|}\lambda_i)\colon C(T^\beta M,T^{|\beta|} F)\to C(T^\beta M,
T^{|\beta|} F_i)
\]
for $i\in I$, see \cite[Lemma~A.6.4]{GaN}.
Thus, the compact-open $C^\alpha$-topology $\cO$ on $C^\alpha(M,F)$
is initial with respect to the maps
$C(T^\beta M,T^{|\beta|}\lambda_i)\circ T^\beta$
with $T^\beta \colon C^\alpha(M,F)\to C(T^\beta M,T^{|\beta|}F)$.
As $T^{\beta}(\lambda_i\circ f)=(T^{|\beta|}\lambda_i) \circ (T^\beta f)$,
writing $\tau_{i,\beta}(g):=T^\beta g$ for $g\in C^\alpha(M,F_i)$
we have
\[
C(T^\beta M,T^{|\beta|}\lambda_i)\circ T^\beta
=\tau_{i,\beta}\circ C^\alpha(M,\lambda_i).
\]
The topology on $C^\alpha(M,F_i)$ being initial with respect to
the mappings\linebreak
$\tau_{i,\beta}\colon C^\alpha(M,F_i)\to C(T^\beta M,T^{|\beta|}F_i)$
for $\beta\leq\alpha$, we deduce that $\cO$
is initial with respect to the mappings $C^\alpha(M,\lambda_i)=(\lambda_i)_*$.\\[1mm]
(b) By \cite[Lemma~3.8]{Alz}, the linear map
$\Theta$ is a bijection. The topology on~$F$
being initial with respect to the maps $\pr_i$,
(a) shows that the topology on $C^\alpha(M,F)$
is initial with respect to the maps $(\pr_i)_*$
and hence makes $\Theta$ a topological
embedding. Hence $\Theta$ is a homeomorphism, being bijective.\\[1mm]
(c)
By \cite[Lemma~3.8]{Alz}, $\Psi$
is a bijection.
By Lemma~\ref{pull-and-push}, $\Psi$
is continuous. To see that $\Psi^{-1}$
is continuous, we prove its
continuity at a given element
$(f_1,f_2)$\linebreak
in $C^\alpha(M,N_1)\times C^\alpha(M,N_2)$.
For $x\in M$, pick a chart
$\phi_{x,i}\colon U_{x,i}\to V_{x,i}\sub E_{x,i}$ of $N_i$ around
$f_i(x)$, for $i\in\{1,2\}$.
There exist compact full submanifolds
$K_{x,j}$ of $M_j$ for $j\in\{1,\ldots, n\}$
such that $K_x:=K_{x,1}\times\cdots\times K_{x,n}
\sub (f_1,f_2)^{-1}(U_{x,1}\times U_{x,2})$
and $x\in K_x^o$. By Lemma~\ref{la:ascending-union},
the topology on $C^\alpha(M,N_1\times N_2)$
is initial with respect to the restriction maps
\[
\rho_x\colon C^\alpha(M,N_1\times N_2)\to C^\alpha(K_x,N_1\times N_2).
\]
It thus suffices to show that $\rho_x\circ\Psi^{-1}$
is continuous at $(f_1,f_2)$
for all $x\in M$. Now $\rho_x\circ \Psi^{-1}=\Psi_x^{-1}\circ (\rho_{x,1}\times\rho_{x,2})$
using the continuous restriction maps
$\rho_{x,i}\colon C^\alpha(M,N_i)\to C^\alpha(K_x,N_i)$
for $i\in \{1,2\}$ and the map
\[
\Psi_x\colon C^\alpha(K_x,N_1\times N_2)\to C^\alpha(K_x,N_1)\times C^\alpha(K_x,N_2)
\]
taking a function to its pair of components.
Thus, it suffices to show that $\Psi_x^{-1}$
is continuous at $(f_1|_{K_x}, f_2|_{K_x})$.
Now $f_i|_{K_x}$ is contained in the open subset
$C^\alpha(K_x,U_{x,i})$
of $C^\alpha(K_x,N_i)$,
on which the latter induces the compact-open
$C^\alpha$-topology, by Lemma~\ref{maps-to-sub}.
The map $\Psi^{-1}$ takes this set
onto $C^\alpha(M,U_{x,1}\times U_{x,2})$,
on which $C^\alpha(M,N_1\times N_2)$
induces the compact-open $C^\alpha$-topology.
It thus suffices to show that
$\Psi_x^{-1}$ is continuous at $(f_1|_{K_x},f_2|_{K_x})$
as a map
\[
C^\alpha(K_x,U_{x,1})\times C^\alpha(K_x,U_{x,2})\to
C^\alpha(K_x,U_{x,1}\times U_{x,2}).
\]
Now $(\phi_{x,j})_*\colon C^\alpha(K_x,U_{x,j})\to C^\alpha(K_x,V_{x,i})$
is a homeomorphism for $i\in\{1,2\}$
and also $(\phi_{x,1}\times \phi_{x,2})_*\colon
C^\alpha(K_x,U_{x,1}\times U_{x,2})\to
C^\alpha(K_x,V_{x,1}\times V_{x,2})$
is a homeomorphism, by Lemma~\ref{pull-and-push}.
It thus suffices to show that
the mapping\linebreak
$(\phi_{x,1}\times\phi_{x,2})_*\circ \Psi_x^{-1}\circ
((\phi_{x,1})_*\times(\phi_{x,2})_*)^{-1}\colon$
\[
C^\alpha(K_x,V_{x,1})\times C^\alpha(K_x,V_{x,2})\to
C^\alpha(K_x,V_{x,1}\times V_{x,2})
\]
is continuous. But this mapping is a restriction of the
homeomorphism\linebreak
$C^\alpha(K_x,E_{x,1})\times C^\alpha(K_x,E_{x,2})\to
C^\alpha(K_x,E_{x,1}\times E_{x,2})$ discussed in (b). $\,\square$\\[2mm]
{\bf Proof of Lemma~\ref{c-alpha-top-mult}.}
The scalar multiplication $\sigma \colon \R\times TN\to TN$
being smooth, the map $\sigma_*\colon C^\alpha(M,\R\times TN)\to C^\alpha(M,TN)$,
$h\mto \sigma\circ h$ is continuous (see Lemma~\ref{pull-and-push}).
Hence $\mu=\sigma_*\circ \Psi^{-1}$ is continuous,
using the homeomorphism $\Psi\colon C^\alpha(M,\R\times TN)\to
C^\alpha(M,\R)\times C^\alpha(M,TN)$
from Lemma~\ref{c-alpha-top-product}. $\,\square$\\[2mm]
{\bf Proof of Lemma~\ref{exp-law-not-pure}.}
Let $(U_i)_{i\in I}$
be the family of pairwise distinct
connected components of~$N$
and $(V_j)_{j\in J}$ be the family of components
of~$M$.
Then
\[
r\colon C^\beta(M,E)\to \prod_{j\in J}C^\beta(V_j,E),\;\,
f\mto (f|_{V_j})_{j\in J}\vspace{-.4mm}
\]
is a bijective linear map;
by Lemma~\ref{la:ascending-union},
it is a homeomorphism.
Likewise,
\[
\rho\colon C^{\alpha,\beta}(N\times M,E)\to\prod_{(i,j)\in I\times J}
C^{\alpha,\beta}(U_i\times V_j,E),\;\,
f\mto (f|_{U_i\times V_j})_{(i,j)\in I\times J}\vspace{-.4mm}
\]
and
$R\colon C^\alpha(N,C^\beta(M,E))\to \prod_{i\in I} C^\alpha(U_i,C^\beta(M,E))$,
$f\mto (f|_{U_i})_{i\in I}$
are isomorphisms of topological vector spaces.
By Lemma~\ref{pull-and-push}, the mapping
$C^\alpha(U_i,r)\colon$\linebreak
$C^\alpha(U_i,C^\beta(M,E))\to
C^\alpha(U_i,\prod_{j\in J}C^\beta(V_j,E))$
is an isomorphism of topological
vector spaces
and so is the map
\[
\Theta_i\colon C^\alpha\Big(U_i,\prod_{j\in J}C^\beta(V_j,E)\Big)\to\prod_{j\in J}
C^\alpha(U_i,C^\beta(V_j,E))\vspace{-.7mm}
\]
taking a map to its family of components
(see Lemma~\ref{c-alpha-top-product}\,(b)).
Hence
\[
\Xi:=\prod_{i\in I}\Theta_i\circ \prod_{i\in I}C^\alpha(U_i,r)\circ R\colon
C^\alpha(N,C^\beta(M,E))\to
\prod_{(i,j)\in I\times J}C^\alpha(U_i,C^\beta(V_j,E))\vspace{-.4mm}
\]
is an isomorphism of topological vector spaces.
By \cite[Theorem~B]{Alz},
the map $\Phi_{i,j}\colon C^{\alpha,\beta}(U_i\times V_j,E))\to
C^\alpha(U_i,C^\beta(V_j,E))$, $f\mto f^\vee$
is linear and a topological embedding, whence so is
\[
\Psi:=\prod_{(i,j)\in I\times J}\Phi_{i,j}\colon
\prod_{(i,j)\in I\times J}C^{\alpha,\beta}(U_i\times V_j,E)\to
\prod_{(i,j)\in I\times J}C^\alpha(U_i,C^\beta(V_j,E)).\vspace{-.4mm}
\]
Evaluating at $x\in N$ and then in $y\in M$
(say $x\in U_i$ and $y\in V_j$),
we verify that
\[
f^\vee =(\Xi^{-1}\circ \Psi\circ \rho)(f)
\]
for all $f\in C^{\alpha,\beta}(N\times M,E)$,
whence $f^\vee\in C^\alpha(N,C^\beta(M,E))$
and $\Phi$ makes sense as a map to the latter
space. We have a commutative diagram
\[
\begin{array}{rcl}
C^{\alpha,\beta}(N\times M,E) & \stackrel{\Phi}{\longrightarrow} & C^\alpha(N,C^\beta(M,E))\\
\rho \downarrow \;\;\;\;\;\;\;\;\;\; & & \;\;\;\;\;\;\;\;\;\;\; \downarrow \Xi\\
\prod_{i,j}C^{\alpha,\beta}(U_i\times V_j,E)&
\stackrel{\Psi}{\longrightarrow} & \prod_{i,j}C^\alpha(U_i,C^\beta(V_j,E))
\end{array}
\]
where the vertical arrows are homeomorphisms
and $\Psi$ is a topological embedding.
Hence $\Phi$ is a topological embedding.
If $M$ is locally
compact, then so are the $V_j$, whence
each of the maps $\Phi_{i,j}$ is a homeomorphism
by \cite[Theorem~4.4]{Alz}
and hence also~$\Psi$. Then also $\Phi=\Xi^{-1}\circ \Psi\circ\rho$
is a homeomorphism. $\,\square$\\[2mm]
{\bf Proof of Lemma~\ref{maps-to-tvs-2}.}
Let $\cO$ be the compact-open
$C^\alpha$-topology on $C^\alpha(U,F)$
and $\cT$ be the initial topology with respect to
the maps
\[
d^\beta\colon C^\alpha(U,F)\to C(U\times E_1^{\beta_1}\times\cdots
\times E_n^{\beta_n},F)
\]
for $\beta\in \N_0^n$ such that $\beta\leq \alpha$.
We claim that, for each $\beta$ as before,
there exist a continuous linear map $\lambda_\beta\colon T^{|\beta|} F\to F$
and $C^\infty$-maps
$\theta_{\beta,j}\colon U_j\times E_j^{\beta_j}\to T^{\beta_j}U_j$
for $j\in\{1,\ldots,n\}$ such that, for all $f\in C^\alpha(U,F)$,
\begin{equation}\label{d-via-T}
d^\beta f(x_1,\ldots, x_n,y_1,\ldots, y_n)
=(\lambda_\beta\circ T^\beta f)(\theta_{\beta,1}(x_1,y_1),\ldots,
\theta_{\beta,n}(x_n,y_n))
\end{equation}
holds for all
$(x_1,\ldots, x_n)\in U$ and $(y_1,\ldots, y_n)\in E_1^{\beta_1}\times\cdots
\times E_n^{\beta_n}$.
Consider the map
$\pi_\beta\colon U\times E_1^{\beta_1}\times \cdots \times E_n^{\beta_n}
\to (U_1\times E_1^{\beta_1})\times \cdots\times (U_n\times E_n^{\beta_n})$,
\[
(x_1,\ldots, x_n,y_1,\ldots, y_n)\mto ((x_1,y_1),\ldots, (x_n,y_n)).
\]
If the claim is true, setting $\theta_\beta:=\theta_{\beta,1}\times\cdots\times \theta_{\beta,n}$
we have
\[
d^\beta =   C(U\times E_1^{\beta_1}\times \cdots\times E_n^{\beta_n},\lambda_\beta)\circ
C(\pi_\beta,T^{|\beta|} F) \circ C(\theta_\beta,T^{|\beta|} F )\circ T^\beta,
\]
which is a continuous
$C(U\times E_1^{\beta_1}\times \cdots\times E_n^{\beta_n},F)$-valued
function on $(C^\alpha(U,F),\cO)$
by Lemmas~A.5.3 and A.5.9 in \cite{GaN}.
Thus $\cT\sub \cO$.
The claim is established by induction on $|\beta|$.
If $|\beta|=1$, then $\beta=e_j$ for some $j\in \{1,\ldots, n\}$.
Using $\pr_2\colon F\times F\to F$, $(v,w)\mto w$,
we have
\[
d^{e_j} f(x_1,\ldots, x_n,y_1,\ldots, y_n)=(\pr_2\circ T^{e_j} f)(\theta_{e_j,1}(x_1,y_1),\ldots,
\theta_{e_j,n}(x_n,y_n))
\]
with $\theta_{e_j,j}(x_j,y_j):=(x_j,y_j)$
for all $(x_j,y_j)\in U_j\times E_j$
and $\theta_{e_j,i}(x_i,y_i):=x_i$
if $i\not=j$ and $(x_i,y_i)\in U_i\times E_i^0= U_i\times \{0\}$.
Assume the claim holds for $\beta$;
thus $d^\beta f$ is of the form~(\ref{d-via-T}).
Let $k\in\{1,\ldots,n\}$
be minimal with $\beta_k\not=0$.
For $j\in\{1,\ldots, k\}$, $(x_1,\ldots, x_n)\in U$,
$y_j=(v,w)\in E_j^{\beta_j}\times E_j$
and $y_i\in E_i^{\beta_i}$ if $i\not=j$,
we then have
\begin{eqnarray*}
\lefteqn{d^{\beta+e_j}f(x_1,\ldots, x_n,y_1,\ldots, y_n)}\qquad\qquad\\
&=& (d\lambda_\beta \circ T^{\beta+e_j} f)(\theta_{\beta+e_j,1}(x_1,y_1),\ldots,
\theta_{\beta+e_j,n}(x_n,y_n))
\end{eqnarray*}
of the desired form
with $\theta_{\beta+e_j,i}:=\theta_{\beta,i}$
for $i\not=j$
and $\theta_{\beta+e_j,j}(x_j,v,w):=T\theta_{\beta,j}(x,v,w,0)$.\\[1mm]
To see that $\cO\sub\cT$,
we show that, for each $\beta\in\N_0^n$ such that $\beta\leq\alpha$,
there exist $m_\beta\in\N$,
multindices $\gamma_{\beta,a}\leq \beta$ for $a\in\{1,\ldots, m_\beta\}$,
continuous linear functions $\lambda_{\beta,a}\colon F\to T^{|\beta|}F$
and smooth functions $\xi_{\beta,a,j}\colon T^{\beta_j}U_j\to
E_j^{\beta_j}$ such that
\begin{eqnarray}
\lefteqn{T^\beta f(y_1,\ldots, y_n)}\notag \\
\hspace*{-40mm}\!\!\!\!\!\!\!\!\!\!\!\! &=&\sum_{a=1}^{m_\beta}
\lambda_{\beta,a}(d^{\hspace*{.2mm}\gamma_{\beta,a}}f(\theta_{1,\beta_1}(y_1),\ldots,
\theta_{n,\beta_n}(y_n),
\xi_{\beta,a,1}(y_1),\ldots,\xi_{\beta,a,n}(y_n))\;\;\; \label{ugly}
\end{eqnarray}
for all $(y_1,\ldots,y_n)\in \prod_{j=1}^nT^{\beta_j}U_j=T^\beta U$,
where
\[
\theta_{j,k}\colon T^kU_j=U_j\times E_j^{2^k-1}\to U_j
\]
is the projection
onto the first component for $j\in\{1,\ldots, n\}$
and $k\in \N_0$
(if we identify $U_j\times E_j^0$ with $U_j$ for $k=0$).
The map $\Xi_{\beta,a}\colon T^\beta U\to U\times E_1^{\beta_1}\times\cdots
\times E_n^{\beta_n}$, $(y_1,\ldots,y_n)\mto
(\theta_{1,\beta_1}(y_1),\ldots, \theta_{n,\beta_n}(y_n),
\xi_{\beta,a,1}(y_1),\ldots,\xi_{\beta,a,n}(y_n))$ is $C^\infty$ and
\[
T^\beta=\sum_{a=1}^{m_\beta}C(T^\beta U,\lambda_{\beta,a})\circ
C(\Xi_{\beta,a},F)\circ d^{\,\gamma_{\beta,a}}
\]
is a continuous $C(T^\beta U, T^{|\beta|}F)$-valued
function on $(C^\alpha(U,F),\cT)$;
so $\cO\sub\cT$.
The proof is by induction on $|\beta|$.
If $|\beta|=1$, then $\beta=e_j$ for some $j$
and
\begin{eqnarray*}
T^{e_j}f(y_1,y_2)
\!\!\! &=&\!\!\! (\lambda_1\circ f)(\theta_{1,\beta_1}(y_1),\ldots,\theta_{n,\beta_n}(y_n))\\
\!\!\!& &\!\!\!
 +(\lambda_2\circ d^{e_j}f)(\theta_{1,\beta_1}(y_1),\ldots,\theta_{n,\beta_n}(y_n),
\pr_2(y_j))\\
\!\!\! &=& \!\!\!(\lambda_1\circ f)(\theta_{1,\beta_1}(y_1),\ldots,\theta_{n,\beta_n}(y_n),
\xi_{e_j,1,1}(y_1),\ldots, \xi_{e_j,1,n}(y_n))\\
\!\!\! & &\!\!\! +(\lambda_2\circ d^{e_j}f)(\theta_{1,\beta_1}(y_1),\ldots,\theta_{n,\beta_n}(y_n),
\xi_{e_j,2,1}(y_1),\ldots \xi_{e_j,2,n}(y_n))
\end{eqnarray*}
with $\xi_{e_j,1,i}(y_i):=0\in E_i^0$
for $i\in\{1,\ldots,n\}$,
$\xi_{e_j,2,j}(y_j):=\pr_2(y_j)$
and $\xi_{e_j,2,i}(y_i)$\\
$:=0\in E_i^0$
for $i\not=j$, using
$\pr_2\colon TU_j=U_j\times E_j\to E_j$,
$\lambda_1\colon F\to F\times F$, $v\mto (v,0)$
and $\lambda_2\colon F\to F\times F$, $v\mto (0,v)$.
Note that we identified $U$ with $U\times (E_1)^0\times \cdots\times (E_n)^0$.
If $\beta\leq \alpha$ with $|\beta|\geq 1$
is given,
let $k\in \{1,\ldots, n\}$ with $\beta_k\geq 1$
be minimal. Let $j\in \{1,\ldots, k\}$
and assume that $\beta':=\beta+ e_j\leq \alpha$.
Write $\beta_1',\ldots, \beta_n'$
for the components of $\beta'$.
Consider the continuous linear
map $\lambda_1\colon T^{|\beta|} F\to T^{|\beta|}F\times T^{|\beta|}F$,
$v\mto (v,0)$ and define $\lambda_2$ analogously.
Keeping the other variables fixed and
differentiating in the $y_j$-variable,
(\ref{ugly}) implies
that
\begin{eqnarray*}
\lefteqn{T^{\beta+e_j}f(y_1,\ldots, y_n)}\\
\!\!\! &=&\!\!\!\sum_{a=1}^{m_\beta}\lambda_1(
\lambda_{\beta,a}(d^{\hspace*{.2mm}\gamma_{\beta,a}}f(\theta_{1,\beta_1'}(y_1),\ldots,
\theta_{n,\beta_n'}(y_n),
\xi_{\beta',a,1}(y_1),\ldots,\xi_{\beta',a,n}(y_n)))\\
\!\!\! & & \!\!\! \!\!+\!
\sum_{a=1}^{m_\beta}\lambda_2(
\lambda_{\beta,a}(d^{\hspace*{.2mm}\gamma_{\beta,a}}f(\theta_{1,\beta_1'}(y_1),\ldots,
\theta_{n,\beta_n'}(y_n),
\eta_{\beta',a,1}(y_1),\ldots,\eta_{\beta',a,n}(y_n)))\\
\!\!\! & & \!\!\! \!\!+\!
\sum_{a=1}^{m_\beta}\lambda_2(
\lambda_{\beta,a}(d^{\hspace*{.2mm}\gamma_{\beta,a}+e_j}f(\theta_{1,\beta_1'}(y_1),\ldots,
\theta_{n,\beta_n'}(y_n),
\zeta_{\beta',a,1}(y_1),\ldots,\zeta_{\beta',a,n}(y_n)))
\end{eqnarray*}
for all $(y_1,\ldots, y_n)\in T^{\beta+e_j}U$,
where $\xi_{\beta',a,j}(y_j):=\xi_{\beta,a,j}(\pr_1(y_j))$,
$\xi_{\beta',a,i}(y_i):=\xi_{\beta,a,i}(y_i)$
for $i\not=j$,
$\eta_{\beta',a,j}(y_j):=d\xi_{\beta,a,j}(y_j)$,
$\eta_{\beta',a,i}(y_i):=\xi_{\beta,a,i}(y_i)$ for $i\not=j$,
$\zeta_{\beta',a,j}(y_j):=(\xi_{\beta,a,j}(\pr_1(y_j)),d\theta_{j,\beta_j}(y_j))$
and $\zeta_{\beta',a,i}(y_i):=\xi_{\beta,a,i}(y_i)$
for $i\not=j$,
using the map $\pr_1\colon T^{\beta_j+1}U_j=T^{\beta_j}U_j\times T^{\beta_j}E_j\to
T^{\beta_j}U_j$. 
Thus also $T^{\beta+e_j}f$
is of the desired form. $\,\square$

Helge Gl\"{o}ckner, Universit\"{a}t Paderborn, Warburger Str.\ 100,
33098 Paderborn, Germany; glockner@math.uni-paderborn.de\\[2.4mm]
Alexander Schmeding,
Nord universitet, H\o{}gskoleveien 27, 7601 Levanger,\linebreak
Norway; alexander.schmeding@nord.no\vfill
\end{document}